\documentclass[aap,preprint]{imsart}

\RequirePackage{amsthm,amsmath,amsfonts,amssymb, mathtools}
\RequirePackage[numbers,sort&compress]{natbib}
\RequirePackage[colorlinks,linkcolor=blue,citecolor=blue,urlcolor=blue]{hyperref}
\RequirePackage{graphicx}

\startlocaldefs

\makeatletter
\newcommand{\assumptionlabel}[2]{\begingroup
  \def\@currentlabel{#1}\phantomsection\label{#2}\endgroup}
\makeatother

\theoremstyle{plain}

\newtheorem{theorem}{Theorem}[section]
\newtheorem{lemma}[theorem]{Lemma}
\newtheorem{proposition}[theorem]{Proposition}
\theoremstyle{definition}

\newtheorem{assumption}{Assumption}
\newtheorem*{mixedsumdefinition}{Definition}
\renewcommand{\theassumption}{A\arabic{assumption}}
\newcounter{arithmeticstep}

\DeclareMathOperator{\tr}{tr}
\DeclareMathOperator{\Var}{Var}
\DeclareMathOperator{\Cov}{Cov}
\DeclareMathOperator{\diag}{diag}

\endlocaldefs

\begin{document}

\begin{frontmatter}
\title{Effective-dimension bounds in the central limit theorem for centered Euclidean balls}
\runtitle{CLT for centered Euclidean balls}

\begin{aug}
\author[A]{\fnms{Alexey}~\snm{Naumov} 
\ead[label=e1]{anaumov@hse.ru}}
\author[A]{\fnms{Sergey}~\snm{Samsonov}\ead[label=e2]{svsamsonov@hse.ru}
}
\author[A]{\fnms{Vladimir}~\snm{Spokoiny}\ead[label=e3]{spokoiny@wias-berlin.de}}
\author[A]{\fnms{Vladimir}~\snm{Ulyanov}\ead[label=e4]{vulyanov@cs.msu.ru}}
\address[A]{HSE University \printead[presep={ ,\ }]{e1,e2,e3,e4}}


\end{aug}

\begin{abstract}
We establish a nonasymptotic Gaussian approximation bound over centered Euclidean balls for sums of independent mean zero random vectors in \(\mathbb R^p\) that need not be identically distributed. Let \(S_n\) denote the sum, \(\Sigma=\operatorname{Var}(S_n)\ne0\), and let \(G\sim N(0,\Sigma)\). Define the effective rank and the normalized squared spectral mass by
\[
r(\Sigma)=\frac{\operatorname{tr}\Sigma}{\|\Sigma\|},
\qquad
v(\Sigma)=\frac{\operatorname{tr}(\Sigma^2)}{\|\Sigma\|^2}.
\]
Under a uniform directional fourth-moment condition, covariance-balance assumptions, and the spectral conditions \(r(\Sigma)\le C_\Sigma v(\Sigma)\) and \(v(\Sigma)\ge6\), we prove
\[
\sup_{z\ge0}
\left|
\mathbb P\{\|S_n\|^2\le z\}
-
\mathbb P\{\|G\|^2\le z\}
\right|
\le
C\,\frac{r(\Sigma)}{n}.
\]
The constant depends only on the parameters in the assumptions and is independent of the ambient dimension and individual small eigenvalues of the covariance matrix. The approximation error thus tends to zero when \(r(\Sigma)=o(n)\), even if \(p\) grows faster than \(n\). We also provide an example showing that the order $r(\Sigma)/n$ is sharp in general. The same estimate extends to random elements of a separable Hilbert space with trace class covariance. 
The proof combines Esseen smoothing, truncation,
covariance partitioning, moment matching, frame/lattice estimates, a
second-order endpoint expansion and comparison estimates expressed in terms of the covariance matrices.
\end{abstract}



\end{frontmatter}
\tableofcontents

\section{Introduction}\label{sec:introduction}
Let $\boldsymbol X_1,\ldots,\boldsymbol X_n$
be independent mean zero random vectors in $\mathbb R^p, \, p \in \mathbb N$, and write $\boldsymbol S_n=\sum_{i=1}^n\boldsymbol X_i$. Let $\Sigma=\Var(\boldsymbol S_n)$ and $\widetilde{\boldsymbol S}\sim\mathcal N(0,\Sigma)$ be a Gaussian element in $\mathbb R^p$. 
For a class $\mathcal A$ of Borel sets, the Gaussian approximation
error is
\begin{equation}\label{eq:intro-distance}
 \Delta_n(\mathcal A)=\sup_{A\in\mathcal A}
 \left|\mathbb P(\boldsymbol S_n\in A)
       -\mathbb P(\widetilde{\boldsymbol S}\in A)\right|.
\end{equation}
We seek nonasymptotic bounds on $\Delta_n(\mathcal A)$ with explicit dependence on $n$, the dimension $p$, and the spectrum of $\Sigma$.
The dimension may grow with $n$. We also discuss the extension to a
separable Hilbert space, where $p=\infty$ and $\Sigma$ is a trace class
covariance operator. 
Three particularly important choices of $\mathcal A$ are convex sets,
Euclidean balls, and coordinate rectangles. These classes lead to
different rates and different forms of dimension and covariance
dependence.

The early works of Bergstr\"om~\cite{Bergstrom1945} and Esseen~\cite{Esseen1945}
initiated the study of quantitative bounds in the multivariate
central limit theorem. Subsequent developments include
\cite{Sazonov1968,Bhattacharya1975,Nagaev1976,Senatov1980,Sazonov1981}.
For normalized i.i.d.\ sums in fixed dimension, the classical
Berry--Esseen rate $n^{-1/2}$ is available under a finite third
absolute moment. When the dimension increases, both the constants
and the moment factors become essential parts of the bound.

For balls in a Hilbert space, the development of the theory has
addressed three related questions: the optimal dependence on the
sample size, the weakest moment assumptions, and the dependence
on the covariance eigenvalues; see \cite[Section~1]{PU2013}.
Results in Sazonov~\cite{Sazonov1988} and Zalesskii et al.~\cite{ZSU1991} form part of this
development. In particular, Zalesskii et al.~\cite{ZSU1991} obtained an $n^{-1/2}$
bound under a finite third moment, with a constant expressed through
the first six covariance eigenvalues. The bound applies to balls
with an arbitrary center, with additional dependence on that
center. Its spectral dependence is sharp in this setting. This
six-eigenvalue result concerns third-moment approximation for
translated balls; it does not by itself determine the spectral
threshold for an $n^{-1}$ bound over centered balls.
For partial sums of stationary linear processes in Hilbert space,
Jirak~\cite{Jirak2018} obtained Berry--Esseen bounds for ball probabilities.

Faster rates are possible for centered balls and ellipsoids.
For centered ellipsoids with axes parallel to the coordinate axes,
Esseen~\cite{Esseen1945} obtained the rate $n^{-p/(p+1)}$ under a
fourth-moment assumption.
Bentkus and G\"otze~\cite{BentkusGoetze1996,BentkusGoetze1997} established $n^{-1}$
rates for quadratic forms under suitable dimension conditions.
G\"otze and Zaitsev~\cite{GoetzeZaitsev2014} proved that, for i.i.d.\ summands with a
finite fourth moment, a nonsingular covariance, and a nondegenerate
quadratic form in fixed dimension $p\ge5$, the approximation error
is $O(n^{-1})$. For centered balls, their explicit finite-dimensional estimate contains
a constant depending on $p$, the inverse square root of the
covariance determinant, and a standardized fourth moment. Thus,
already in dimension five, an $n^{-1}$ rate is available, but that
estimate deteriorates when the smallest covariance eigenvalue
becomes small relative to the largest. Their general theorem also covers Hilbert spaces,
with a constant depending on the covariance operator and the
quadratic form. Related results for balls and short asymptotic
expansions were obtained by
G\"otze and Ulyanov~\cite{GotzeUlyanov2000}
and Bogatyrev et al.~\cite{BogatyrevGoetzeUlyanov2006}.
These results motivate bounds of the same order in $n$ while
controlling the constants uniformly over growing dimensions and
varying covariance spectra. Ulyanov~\cite[Theorem~2 and its corollary]{Ulyanov1987} treated independent Hilbert-space summands with different covariance operators. His centered-ball estimate includes the power $2-\varepsilon$ of the sum of truncated third moments, under a separate condition on the covariance spectrum.

The dependence on dimension was studied systematically by
Bentkus~\cite{Bentkus1986,Bentkus2003,Bentkus2005}. For the class
$\mathcal C_p$ of all convex Borel sets and nonsingular $\Sigma$,
his results yield
\begin{equation}\label{eq:intro-bentkus}
 \Delta_n(\mathcal C_p)
 \le Cp^{1/4}\sum_{i=1}^n
       \mathbb E\|\Sigma^{-1/2}\boldsymbol X_i\|^3.
\end{equation}
The i.i.d.\ result in Bentkus~\cite{Bentkus2003} has the explicit constant
$400$, and Bentkus~\cite{Bentkus2005} extends the $p^{1/4}$ dependence to
nonidentically distributed summands. 
The standardized moments also contain dimension dependence: for
$\Var(\boldsymbol Y_1)=\Sigma$,
$\mathbb E\|\Sigma^{-1/2}\boldsymbol Y_1\|^3\ge p^{3/2}$.
Thus, when this moment is of order $p^{3/2}$,
\eqref{eq:intro-bentkus} is of order $p^{7/4}/\sqrt n$.
The transformation $x\mapsto\Sigma^{-1/2}x$ is natural for the affine-invariant class of convex sets,
but the resulting moment bound does not become smaller merely because some
eigenvalues of $\Sigma$ are small. 
Restricting the class of sets can improve the approximation rate.
For balls and identity covariance, the $p^{1/4}$ factor in
\eqref{eq:intro-bentkus} can be replaced by an absolute constant
\citep{Bentkus2003}. A further improvement is possible for
balls centered at the origin. At the level of a formal Edgeworth
expansion, the first correction is odd and integrates to zero over
a centered ball. Turning this cancellation into a uniform bound
requires control of the remainder, including possible lattice
effects.
For balls with arbitrary centers, Zhilova~\cite{Zhilova2019}
compared sums with matching moments and applied the bounds to the weighted
bootstrap. Her finite-sample Edgeworth-type expansions~\cite{Zhilova2022Edgeworth}
include explicit higher-moment terms; for translated balls the odd term
need not vanish.

Stein's method supplies another major approach to these questions.
G\"otze~\cite{Gotze1991} developed a multivariate normal approximation
argument for nonsmooth test functions. The concentration-inequality
approach of Chen and Shao~\cite{ChenShao2007}, formulated for real-valued nonlinear
statistics, was extended to multivariate approximation by
Chen and Fang~\cite{ChenFang2015}; for independent standardized sums their convex-set
bound has a factor of order $p^{1/2}$ multiplying the sum of third
absolute moments. Shao and Zhang~\cite{ShaoZhang2022} established Berry--Esseen
bounds for multivariate nonlinear statistics, with applications to
$M$-estimators and averaged stochastic gradient descent.
Chen et al.~\cite{ChenGoldsteinShao2011} give a systematic account of Stein's
method.  Rai\v{c}~\cite{Raic2019}
obtained the explicit factor $42p^{1/4}+16$ in \eqref{eq:intro-bentkus} by Stein's method.
More recently, Fang and Koike~\cite{FangKoike2024Balls} obtained bounds
based on fourth moments for convex sets and balls, improving the
dimension dependence available from classical third-moment bounds.
For $\boldsymbol S_n=n^{-1/2}\sum_{i=1}^n\boldsymbol Y_i$
with i.i.d.\ isotropic $\boldsymbol Y_i$ satisfying
$\mathbb E\|\boldsymbol Y_1\|^4\le Cp^2$, their bounds have orders
$p^{5/4}/\sqrt n$ for convex sets and $p/\sqrt n$ for all balls,
up to logarithmic factors.
For the same isotropic sums, Cai et al.~\cite[Theorem~2.4 and Remark~2.5]{CaiYangZhang2026} obtain the
convex-set bound $Cp^{5/4}/\sqrt n$ without the logarithmic factor.

A different line of work concerns the class $\mathcal R_p$ of
coordinate rectangles $\prod_{j=1}^p[a_j,b_j]$ and the distributions
of coordinatewise maxima. Chernozhukov et al.~\cite{CCK2013,CCK2017} developed Gaussian
and bootstrap approximations with logarithmic dependence on $p$,
allowing the ambient dimension to be much larger than the sample size.
For example, for $p,n\ge3$, consider
$\boldsymbol S_n=n^{-1/2}\sum_{i=1}^n\boldsymbol Y_i$, with
$n^{-1}\sum_i\mathbb E Y_{ij}^2=1$ for every $j$ and
$|Y_{ij}|\le B_n$, where $B_n\ge1$. A representative bound is
\begin{equation}\label{eq:intro-cck}
 \Delta_n(\mathcal R_p)
 \le C\left\{\frac{B_n^2\log^7(pn)}n\right\}^{1/6}.
\end{equation}
Under the same bounded-coordinate setting,
Chernozhukov et al.~\cite{CCKK2022} improved the corresponding rate to
$C\{B_n^2\log^5(pn)/n\}^{1/4}$.
These results allow arbitrary correlations, including singular
covariance matrices. Their variance and tail conditions concern
individual coordinates, rather than the smallest eigenvalue of the
full covariance matrix. Faster rates are available with additional structure.
Fang and Koike~\cite{FangKoike2021} used Stein's method to obtain high-dimensional
rectangle bounds for nonsingular Gaussian targets.
For bounded coordinates and a correlation matrix whose smallest
eigenvalue is bounded away from zero,
Chernozhukov et al.~\cite{CCKoike2023} proved the nearly optimal bound
\begin{equation}\label{eq:intro-cckoike}
 \Delta_n(\mathcal R_p)
 \le C\left\{\frac{B_n^2\log^3 p}n\right\}^{1/2}\log n,
\end{equation}
where the constant depends on that lower bound.
Fang et al.~\cite{FKLZ2023} showed that rates of order $n^{-1/2}$ up to
logarithmic factors can also hold for a singular covariance matrix
under suitable coordinate tail bounds and when all three-coordinate
correlation submatrices are uniformly nondegenerate; their result is stated for lower orthants, or
multivariate Kolmogorov distance. Under a different structural
assumption, namely i.i.d.\ log-concave summands,
Fang and Koike~\cite{FangKoike2024Logconcave} obtained a rectangle bound
$C\{\log^{13}(pn)/n\}^{1/2}$ without restrictions on the covariance
matrix. Thus covariance nondegeneracy is a feature of particular
fast-rate results, rather than a universal requirement of the
rectangle CLT. Logarithmic dependence on ambient dimension and
dependence on an intrinsic spectral dimension are nevertheless
different objectives, adapted to different classes of statistics.

A related direction is to express statistical error bounds through an effective dimension rather than the number of coordinates. For example, for Gaussian observations, Koltchinskii and Lounici~\cite{KoltchinskiiLounici2017Cov}
established sharp expectation and concentration bounds for the
operator-norm error of the sample covariance in terms of the effective rank
\begin{equation}\label{eq:intro-effective-rank}
 r(\Sigma)=\tr\Sigma/\|\Sigma\|.
\end{equation}
This effective rank can remain moderate when $p$ is very large and
many eigenvalues are small. Beyond Gaussian observations,
Mendelson and Zhivotovskiy~\cite{MendelsonZhivotovskiy2020} constructed robust covariance
estimators under $L_4$--$L_2$ norm equivalence, with error bounds
depending on effective rank rather than explicitly on the ambient
dimension. Bounds expressed through effective rank have become
central to the analysis of covariance operators.
For robust mean estimation, Lugosi and Mendelson~\cite{LugosiMendelson2019} obtained
sub-Gaussian deviation bounds under finite second moments,
expressed through the trace and operator norm of the covariance.
Koltchinskii and Lounici~\cite{KoltchinskiiLounici2017Proj} developed concentration and
normal approximation results for spectral projectors in which
effective rank, spectral gaps, and variance parameters describe
the statistical complexity.
Koltchinskii et al.~\cite{KoltchinskiiLofflerNickl2020} developed efficient estimators
of linear functionals of principal components under effective-rank
conditions.
Koltchinskii~\cite{Koltchinskii2025Functionals} established effective-rank
bounds for estimating smooth covariance functionals using
jackknife bias reduction.

For likelihood-based inference, Spokoiny~\cite{Spokoiny2017} derived
finite-sample Fisher and Wilks expansions for quadratically
penalized likelihood in terms of an effective dimension.
Furthermore, for general covariance matrices, G\"otze et al.~\cite{GNSU2019} obtained Gaussian comparison and anti-concentration bounds for balls in Hilbert space in terms of spectral characteristics. Their comparison bounds use the trace norm of the covariance difference, rather than the inverse-covariance quantities arising from Gaussian relative entropy. Such results emphasize the value of working in the original Euclidean geometry when small-variance directions contribute little to the statistic of interest.

Related spectral notions of effective dimension also play a central
role in regression.
Caponnetto and De Vito~\cite{CaponnettoDeVito2007} used a regularized effective dimension
to obtain optimal convergence rates for kernel least-squares
estimators.
Hsu et al.~\cite{HsuKakadeZhang2012} analyzed random-design ridge regression
using regularized spectral sums to express prediction-error bounds.
Bartlett et al.~\cite{BartlettLongLugosiTsigler2020} characterized benign
overfitting in linear regression through effective ranks of the
covariance spectrum and its tails.
Cheng and Montanari~\cite{ChengMontanari2024} obtained dimension-free approximations
for ridge-regression risk, including models with Hilbert-space
covariates and regimes close to interpolation.
Bach~\cite{Bach2024} analyzed double descent in linear regression with
random projections using spectral quantities and effective degrees
of freedom. The intrinsic dimensions in these regression problems
need not coincide with \eqref{eq:intro-effective-rank}; they depend
on the statistical problem and the regularization, when present. Together,
these developments motivate distributional approximations that
reflect covariance structure even when the number of coordinates
is very large.

\paragraph*{Contributions} The present paper addresses this objective for the class
$$
\mathcal B_0=\bigl\{\{x\in\mathbb R^p:\|x\|^2\le z\}:z\ge0\bigr\}
$$ 
of centered Euclidean
balls. For any nonzero covariance matrix $\Sigma$, let
$r=r(\Sigma)$ be its effective rank as in \eqref{eq:intro-effective-rank}.
Under conditions \textup{(\ref{ass:A1})--(\ref{ass:A3})} (see below) we prove
\begin{equation}\label{eq:intro-main}
\Delta_n(\mathcal B_0) 
 \le C\frac rn.
\end{equation}
The constant depends only on the parameters in
\textup{(\ref{ass:A1})--(\ref{ass:A3})}, not on $p$ or individual small eigenvalues.
To our knowledge, this is the first bound under the stated
assumptions that combines the $n^{-1}$ rate with an effective-rank
factor.
An isotropic scale mixture gives a matching lower bound of order $r/n$ when $r\to\infty$ and $r=o(n)$; see Proposition~\ref{prop:effective-rank-lower-bound}. Thus the approximation error tends to zero when $r=o(n)$ even if
$p$ grows much faster than $n$. We note that the bound~\eqref{eq:intro-main} also holds for independent mean zero, square-integrable
random elements of a separable real Hilbert space, under the
operator versions of \textup{(\ref{ass:A1})--(\ref{ass:A3})} and with trace class
covariance $\Sigma$. Thus the ambient Hilbert space may be infinite-dimensional; see the details below in Section \ref{sec: Discussion}.

The endpoint condition concerns the aggregate quantity
$v(\Sigma)=\tr(\Sigma^2)/\|\Sigma\|^2$; it is stronger than requiring
six positive eigenvalues. When $\Sigma$ is proportional to the
identity on its range, $v(\Sigma)$ equals its rank, so the theorem
includes the six-dimensional case.
Conversely, in rank exactly six, $v(\Sigma)\ge6$ forces
$\Sigma$ to be proportional to the identity on its range.
For a general spectrum, the
constant remains uniform as individual smaller eigenvalues tend
to zero, provided \textup{(\ref{ass:A1})--(\ref{ass:A3})} continue to hold with the same
parameters. The contribution is this uniform effective-rank
control rather than the existence of an
$n^{-1}$ rate in a fixed dimension.

Potential applications include confidence balls for vector means
and regularized estimators; score and likelihood-ratio statistics
that admit quadratic approximations; Gaussian and multiplier
bootstrap calibration of norm-based statistics; inference for
finite-dimensional approximations in kernel and functional models;
and squared errors of estimated principal components or spectral
projectors after linearization. Related settings are discussed in
\cite{Spokoiny2017,SpokoinyZhilova2015,GNSU2019,
KoltchinskiiLounici2017Proj,ShaoZhang2022}.
The $n^{-1}$ approximation term is particularly useful when one
seeks accuracy beyond a first-order $n^{-1/2}$ bound and the
effective dimension is moderate. In each application, the
covariance and moment assumptions must hold for the relevant
linearized sum, and nonlinear remainders and covariance estimation
errors require their own control.

\section{Main result}
\label{sec:gaussian}

Let $\Sigma_i=\Var(\boldsymbol X_i)$ and suppose that
$\Sigma=\sum_{i=1}^n\Sigma_i\ne0$. Denote its eigenvalues by
$s_1=\|\Sigma\|\ge s_2\ge\cdots\ge s_p\ge0$ and write
\[
 r=r(\Sigma)=\frac{\tr\Sigma}{\|\Sigma\|},\qquad
 v=v(\Sigma)=\frac{\tr(\Sigma^2)}{\|\Sigma\|^2}.
\]
The covariance may be singular. If $\Sigma=0$, both
$\boldsymbol S_n$ and $\widetilde{\boldsymbol S}$ vanish almost surely,
so $\Delta_n(\mathcal B_0)=0$ and the ratios above are unnecessary.
The constants
$C_1,K_4,C_\Sigma$ in \textup{(\ref{ass:A1})--(\ref{ass:A3})} are independent of $n$ and $p$.

\begin{assumption}\label{ass:A1}
There is a constant $C_1\ge1$ such that
\[
 n\|\Sigma_i\|\le C_1\|\Sigma\|,\qquad
 n\tr\Sigma_i\le C_1\tr\Sigma,\qquad i=1,\ldots,n.
\]

\end{assumption}

\begin{assumption}\label{ass:A2}
There is a constant $K_4\ge1$ such that, for every real
vector $u$ and every $i=1,\ldots,n$,
\begin{equation}\label{eq:2.1}
 \mathbb E|u^\top\boldsymbol X_i|^4
       \le K_4(u^\top\Sigma_i u)^2.
\end{equation}

\end{assumption}

\begin{assumption}\label{ass:A3}
\begin{enumerate}
\item[(i)]\assumptionlabel{(\theassumption)(i)}{ass:A3-i}
There is a constant $C_\Sigma\ge1$ such that
\[
 r\le C_\Sigma v.
\]
\item[(ii)]\assumptionlabel{(\theassumption)(ii)}{ass:A3-ii}
$v\ge6$.
\end{enumerate}
\end{assumption}
The roles of assumptions \textup{(\ref{ass:A1})--(\ref{ass:A3})} are explained in the discussion below. Since $s_j^2\le\|\Sigma\|s_j$ for every $j$, we have $v\le r$. Thus \textup{(\ref{ass:A3})} gives 
\begin{equation}\label{eq:proof-spectral}
 6\le v\le r\le C_\Sigma v.
\end{equation}
The following theorem is the main result of this paper.
\begin{theorem}\label{thm:2.1}
Suppose that $\boldsymbol X_1,\ldots,\boldsymbol X_n$ are independent
mean zero random vectors in $\mathbb R^p$, and let
$\Sigma=\Var(\boldsymbol S_n)\ne0$.
Let $\widetilde{\boldsymbol S}\sim\mathcal N(0,\Sigma)$.
Under \textup{(\ref{ass:A1})--(\ref{ass:A3})}, there is a constant $C$, depending only on
$C_1,K_4,C_\Sigma$, such that
\begin{equation}\label{eq:main-bound}
 \Delta_n(\mathcal B_0)
 \le C\frac rn.
\end{equation}
The constant is independent of $p,n$ and the individual eigenvalues.
\end{theorem}

\paragraph*{Confidence balls and further directions}
Let $Y_1,\ldots,Y_n$ be i.i.d.\ observations in $\mathbb R^p$
with mean $\theta$ and known covariance $\Gamma\ne0$. Suppose the
increments $\boldsymbol X_i=(Y_i-\theta)/\sqrt n$,
whose total covariance is $\Gamma$, satisfy
\textup{(\ref{ass:A1})--(\ref{ass:A3})}; here \textup{(\ref{ass:A1})} holds with $C_1=1$.
Write $\overline Y_n=n^{-1}\sum_iY_i$. For $0<\alpha<1$, let
$q_{1-\alpha}$ be the $(1-\alpha)$-quantile of $\|G\|^2$, where
$G\sim\mathcal N(0,\Gamma)$, and define
\[
 \mathcal C_n(\alpha)
 =\{u\in\mathbb R^p:n\|\overline Y_n-u\|^2\le q_{1-\alpha}\}.
\]
Theorem~\ref{thm:2.1} and continuity of the
Gaussian squared-norm distribution imply
\[
 \left|\mathbb P_\theta\{\theta\in\mathcal C_n(\alpha)\}
       -(1-\alpha)\right|
 \le C\frac{r(\Gamma)}n,
 \qquad r(\Gamma)=\frac{\tr\Gamma}{\|\Gamma\|}.
\]
Consequently, bounded effective rank gives an $O(n^{-1})$ coverage
error, uniformly over the nominal level $\alpha$.
For unknown covariance, natural further directions are calibration
using an estimated Gaussian covariance or a multiplier bootstrap.
The Gaussian comparison and anti-concentration bounds of
G\"otze et al.~\cite{GNSU2019} provide tools
for controlling covariance perturbations. Likelihood-based
confidence sets, including under model misspecification as studied
by Spokoiny and Zhilova~\cite{SpokoinyZhilova2015}, give another
direction. These extensions require additional control of covariance
estimation, conditional bootstrap approximation, and non-linear
remainders; the bound above identifies the Gaussian approximation
component of the coverage error.

Non-asymptotic analysis of multivariate statistics with explicit approximation error bounds under high-dimensional framework $p/n \to c \in (0, 1)$ see in Fujikoshi et al.~\cite{fujikoshi2010}. 

\section{Discussion}\label{sec: Discussion}
Assumption \textup{(\ref{ass:A1})} controls the contribution of each increment
and permits a bound in terms of $n$ and the total covariance.
Assumption \textup{(\ref{ass:A2})} is a uniform $L_4$--$L_2$ norm equivalence
(also called the bounded kurtosis condition)
for the one-dimensional marginals:
\[
 \|u^\top\boldsymbol X_i\|_{L_2}
 \le \|u^\top\boldsymbol X_i\|_{L_4}
 \le K_4^{1/4}\|u^\top\boldsymbol X_i\|_{L_2},
 \qquad u\in\mathbb R^p,\quad i=1,\ldots,n,
\]
where $\|Y\|_{L_q}=(\mathbb E|Y|^q)^{1/q}$. The same constant is
required for every direction and increment, so this is stronger
than merely requiring finite fourth moments. This condition is
used by Mendelson and Zhivotovskiy~\cite{MendelsonZhivotovskiy2020} in robust covariance
estimation with effective-rank bounds. A broad class of examples
is provided by centered random vectors with log-concave densities:
the moment comparison of Bobkov~\cite[Theorem~3]{Bobkov2001Khinchin},
applied to linear functions, gives \textup{(\ref{ass:A2})} with a numerical
constant independent of dimension. Gaussian vectors satisfy
\textup{(\ref{ass:A2})} with $K_4=3$. 
The ratio
\[
 \frac vr=\frac{\tr(\Sigma^2)}{\|\Sigma\|\tr\Sigma}
 =\sum_{j=1}^p\frac{s_j}{\sum_\ell s_\ell}\,\frac{s_j}{\|\Sigma\|}
\]
is a weighted average of the relative eigenvalues $s_j/\|\Sigma\|$.
Assumption \textup{\ref{ass:A3-i}} bounds this average below by $1/C_\Sigma$.
Moreover, \textup{(\ref{ass:A3})} gives a lower bound on the sixth eigenvalue
relative to the largest. Applying Lemma~\ref{reserve6:sixth-eigenvalue}
to $\Sigma/\|\Sigma\|$ gives
\[
 v\le5+\frac{s_6}{\|\Sigma\|}(r-5),\qquad
 \frac{s_6}{\|\Sigma\|}\ge\frac{v-5}{r-5}\ge\frac1{6C_\Sigma-5}.
\]
There is no corresponding lower bound on the remaining eigenvalues.
The trace comparison in \textup{\ref{ass:A3-i}} cannot simply be omitted.
The following proposition gives an example of the array such that
$r_n/v_n\to\infty$ and $r_n/n\to0$, but its approximation error
does not vanish. 

\begin{proposition}
\label{prop:ce-without-A3}
For every integer $k\ge1$, there exists a triangular array,
identically distributed within each row, satisfying \textup{(\ref{ass:A1})} and
\textup{(\ref{ass:A2})} with $C_1=1$ and $K_4=6$, such that
\[
 \|\Sigma_n\|=1,\qquad r_n=k+\sqrt n,\qquad v_n=k+\frac1n,
\]
and
\begin{equation}\label{eq:ce-A3-conclusion}
 \frac{r_n}{v_n}\to \infty,\qquad
 \frac{r_n^{3/2}}n\to 0,\qquad
 \liminf_{n\to\infty}\Delta_n(\mathcal B_0)\ge c_k>0.
\end{equation}
In particular, $r_n/n\to0$. Here $c_k=\mathbb E\Phi(-Y_k)$, where $Y_k\sim\chi_k^2$ and $\Phi$ is
the standard normal distribution function. 
\end{proposition}
Note that this construction satisfies \textup{\ref{ass:A3-ii}} whenever $k\ge6$.
Its proof is given in Section~\ref{sec:counterexample-proofs}.

The lower spectral restriction cannot be omitted altogether.
Proposition~\ref{prop:ce-v4} gives a bounded isotropic example with
$r=v=4$ for which $n\Delta_n(\mathcal B_0)$ is unbounded.
 
\begin{proposition}
\label{prop:ce-v4}
Let $\boldsymbol\xi_i=(\varepsilon_{i1},\ldots,\varepsilon_{i4})$, where
all $\varepsilon_{ij}$ are independent Rademacher variables, and put
\[
 \boldsymbol X_{n,i}=\frac{\boldsymbol\xi_i}{\sqrt n},\qquad
 \boldsymbol S_n=\frac1{\sqrt n}\sum_{i=1}^n\boldsymbol\xi_i.
\]
Then \textup{(\ref{ass:A1}), (\ref{ass:A2}), \ref{ass:A3-i}}
hold with $C_1=C_\Sigma=1$ and $K_4=3$. Here $\Sigma_n=I_4$, so
$r_n=v_n=4$ and \textup{\ref{ass:A3-ii}} fails. Nevertheless,
\begin{equation}\label{eq:ce-v4-conclusion}
 \limsup_{n\to\infty}n\Delta_n(\mathcal B_0)=\infty.
\end{equation}
\end{proposition}
The proof of Proposition~\ref{prop:ce-v4} is given in Section~\ref{sec:counterexample-proofs}. We leave open whether, under \textup{(\ref{ass:A1})}, \textup{(\ref{ass:A2})} and
\textup{\ref{ass:A3-i}}, the bound
\[
 \Delta_n(\mathcal B_0)\le C\frac rn
\]
continues to hold when $v=5$, with $C$ depending only on
$C_1,K_4,C_\Sigma$. More generally, can \textup{\ref{ass:A3-ii}} be
weakened to $v\ge5$ with the same uniformity? The question concerns
uniform bounds over independent, possibly nonidentically distributed
arrays and varying covariance spectra. For i.i.d.\ sums with a
fixed nonsingular covariance in dimension at least five, an
$n^{-1}$ rate with covariance-dependent constants is already known
from G\"otze and Zaitsev~\cite{GoetzeZaitsev2014}.
The present argument uses six nondegenerate directions and does
not settle the uniform question at $v=5$.

For isotropic sums in $\mathbb R^p$ with uniformly bounded coordinate
fourth moments, Fang and Koike~\cite[Corollary~3.1]{FangKoike2024Balls}
obtain a centered-ball bound of order $n^{-1/8}+(p/n)^{1/6}$.
Their Proposition~3.1 shows, by a Gaussian scale-mixture example,
that $p=o(n)$ is necessary in general. This distinguishes the
dimension-growth condition for convergence from the quantitative
dependence on $n$. Their general Theorem~3.2 also yields, under our
assumptions \textup{(\ref{ass:A1})--(\ref{ass:A3})},
\begin{equation}\label{eq:discussion-fang-koike}
 \Delta_n(\mathcal B_0)
 \le C\left\{n^{-1/8}+\left(\frac rn\right)^{1/6}\right\}.
\end{equation}
Thus Fang and Koike's result already gives convergence when $r=o(n)$.
Theorem~\ref{thm:2.1} gives the quantitative rate $Cr/n$ in this regime.
The following example shows that this rate cannot in general be improved.

Portnoy~\cite{Portnoy1986} used a mixed normal example to study
the effect of growing dimension. In the isotropic formulation of
Zhilova~\cite[Theorem~4.2]{Zhilova2022Edgeworth}, a summand has the law of
a standard Gaussian vector multiplied by an independent standard normal
scalar. The example below uses a bounded multiplier instead. For both
choices the unconditional ball distance is at most $Cd/n$ when $r=v=d$,
even if the supremum is taken over all Borel sets, as shown at the end of the proof of Proposition~\ref{prop:effective-rank-lower-bound}. For the Gaussian multiplier this upper bound is incompatible with the $cd^{3/2}/n$ lower bound stated immediately before Theorem~4.2 of Zhilova~\cite{Zhilova2022Edgeworth}. The coupling statement of that theorem concerns a different quantity.
A specialization of the Gaussian scale-mixture construction in
Fang and Koike~\cite[Proposition~3.1]{FangKoike2024Balls} gives the
following quantitative lower bound.
\begin{proposition}\label{prop:effective-rank-lower-bound}
Let $d=d_n\ge6$ be integers such that $d\to\infty$ and $d=o(n)$.
Let $Z_1,\ldots,Z_n$ be independent $N(0,I_d)$ vectors, and let
$\eta_1,\ldots,\eta_n$ be independent of these vectors and of each
other, with
\[
 \mathbb P(\eta_i=0)=\frac12,\qquad
 \mathbb P(\eta_i=\sqrt2)=\mathbb P(\eta_i=-\sqrt2)=\frac14.
\]
Set $\boldsymbol X_i=\eta_iZ_i/\sqrt n$ and
$\boldsymbol S_n=\sum_i\boldsymbol X_i$. Then
\textup{(\ref{ass:A1})--(\ref{ass:A3})} hold with $C_1=C_\Sigma=1$ and $K_4=6$,
and $\Sigma=I_d$, so $r=v=d$. Moreover,
\begin{equation}\label{eq:effective-rank-lower-bound}
 \liminf_{n\to\infty}\frac n{r}\Delta_n(\mathcal B_0)
 \ge\frac{\phi(1)}4>0,
\end{equation}
where $\phi(x)=(2\pi)^{-1/2}e^{-x^2/2}$.
\end{proposition}

Since $r=d$ in this example, Proposition~\ref{prop:effective-rank-lower-bound}
and Theorem~\ref{thm:2.1} bound the centered-ball distance from below and above
by constant multiples of $r/n$ along these arrays. In particular, the power
one of the effective rank cannot be lowered in a uniform $n^{-1}$ bound.
The proof of the lower bound is given in Section~\ref{sec:counterexample-proofs}.

The linear dependence on $r$ comes from the terms indexed by pairs of
distinct summands in Proposition~\ref{prop:moderate}. If both increments are
estimated only by their norms, those terms give $Cr^{3/2}/n$. For small
frequencies, a term with three pairs must include a scalar product between
the two independent increments, since three directions come from each.
Keeping that scalar product gives an additional factor $r^{-1/2}$ by the
operator-norm bound in \textup{(\ref{ass:A1})}. For the other terms one obtains the same saving or an extra power of the frequency.
Lemma~\ref{lem:sharp-small-pair} gives the precise estimate.
The upper limit in Esseen's smoothing integral can then be taken as
$U=n/(r\sqrt v)$, for which the remainder is $C/(U\sqrt v)=Cr/n$.

We finish this section with the Hilbert space extension of the main result. The bound of Theorem \ref{thm:2.1} also holds for independent mean zero, square-integrable
random elements of a separable real Hilbert space, under the
operator versions of \textup{(\ref{ass:A1})--(\ref{ass:A3})} and with trace class
covariance $\Sigma$. 

It suffices by scalar rescaling to prove the extension for $\|\Sigma\|=1$. Let $P_N$ project onto its first $N$
eigenvectors with eigenvalues $s_j$, and put
\[
 r_N=\sum_{j\le N}s_j,\qquad v_N=\sum_{j\le N}s_j^2,\qquad
 t_N=\sqrt{(6-v_N)_+}.
\]
Here $(a)_+=\max(a,0)$.
For independent standard normal variables $g_i$, independent of
the original random elements, define
\[
 \boldsymbol X_i^{(N)}
   =\bigl(P_N\boldsymbol X_i,\sqrt{t_N/n}\,g_i\bigr).
\]
The total covariance has eigenvalues $s_1,\ldots,s_N,t_N$.
For sufficiently large $N$, its operator norm is one,
its trace is $r_N+t_N\to r$, and its squared spectral mass is
$\max(v_N,6)\ge6$. Assumptions \textup{(\ref{ass:A1})--(\ref{ass:A3})} hold with
constants $2C_1$, $\max(K_4,3)$, and $2C_\Sigma$, uniformly in $N$.
For the moment condition, independence gives
$\mathbb E(U+\sigma g)^4
=\mathbb EU^4+6\sigma^2\mathbb EU^2+3\sigma^4
\le\max(K_4,3)(\mathbb EU^2+\sigma^2)^2$.
The finite-dimensional bound is $C(r_N+t_N)/n$, which tends to $Cr/n$. The approximating
sums and Gaussian elements can be coupled so that their squared
norms converge in probability to their Hilbert space counterparts. Passing to
distribution functions at $z\pm\eta$, then letting $N\to\infty$
and $\eta\downarrow0$, gives \eqref{eq:intro-main}, since the
Gaussian squared norm has a continuous distribution. Finite rank
covariance is covered directly by restriction to its range.

\section{Proof of Theorem~\texorpdfstring{\ref{thm:2.1}}{2.1}}\label{sec:proofs}
Without loss of generality, we may assume that $\|\Sigma\|=1$
and work on the image $\operatorname{Im}(\Sigma)$. Thus $r=\tr\Sigma$ and
$v=\tr(\Sigma^2)$ throughout the proof.

We first truncate and center the increments, and compare the resulting
Gaussian covariance with the original one. This reduction is common
to both spectral regimes. Section~\ref{sec:case-v-ge7} proves the
theorem for $v\ge7$, using a bounded-array estimate with the margin $v\ge13/2$ after truncation. Section~\ref{sec:case-v-near6}
treats $6\le v<7$, including the possible loss of order $n^{-1}$ in
$v$ caused by truncation. Its additional comparison estimates are
proved in Appendix~\ref{sec:reserve6}.

Put
\begin{gather*}
 \mathfrak f(t)=\mathbb E e^{it\|\boldsymbol S_n\|^2/2},\qquad
 \widetilde{\mathfrak f}(t)
       =\mathbb E e^{it\|\widetilde{\boldsymbol S}\|^2/2}, \qquad
 D(t)=|\widetilde{\mathfrak f}(t)|
      =\prod_{j=1}^p(1+t^2s_j^2)^{-1/4}.
\end{gather*}
Unless stated otherwise, constants depend only on
$C_1,K_4,C_\Sigma$ and, for bounded arrays, the truncation constant $L_0$. In particular, they do not depend on a positive gap
between $v$ and six.

\paragraph*{Role of the Fourier and arithmetic estimates}
We apply Esseen's smoothing inequality with upper integration limit
$U=n/(r\sqrt v)$. In the case $v\ge7$, the replacement argument estimates
the integral up to $\min(U,\tau)$, where $\tau=C_Hn^{1/3}/\sqrt v$.
When $U>\tau$, the remaining integral involves the characteristic
function of the original sum. The moment assumptions allow lattice
laws, whose characteristic functions can have large peaks away from
zero. Gaussian density and comparison bounds alone do not control
these peaks. The needed conclusion is the integrated estimate
$\int_\tau^U|\mathfrak f(t)|\,dt/t\le C/n$ of
Proposition~\ref{prop:high-frequency}. We use the real symmetrization inequality in
G\"otze and Zaitsev~\cite[Lemma~6.1]{GoetzeZaitsev2014}, attributed
there to Bentkus and G\"otze~\cite[Lemma~5.1]{BentkusGoetze1997},
to reduce the quadratic phase to bilinear forms in independent sign
sums. The cross-frame construction in Appendix~\ref{sec:appendixB}
then selects six directions on each side with a uniformly controlled
cross matrix. After Poisson summation, the rotation estimate of
Buterus et al.~\cite[Theorem~5.12]{ButerusGoetzeHilleMargulis2022}
controls the average size of the resulting lattice sums. In
Lemma~\ref{lem:arithmetic} its condition is $\beta s>2$, satisfied
by $s=6$ and $\beta=1/2$. Appendix~\ref{sec:appendixC} carries out
this reduction for the nonidentically distributed array, with
constants controlled by the assumption parameters. The comparison
estimates in Appendices~\ref{sec:appendixA} and~\ref{sec:reserve6}
handle moderate frequencies.

\subsection{Truncation and Gaussian comparisons}
\begin{lemma}
\label{lem:new-truncation}
Set $E_i=\{\|\boldsymbol X_i\|>\sqrt r\}$ and define
\[
 \mu_i=\mathbb E[\boldsymbol X_i\mathbf1_{E_i^c}],\qquad
 Y_i=\boldsymbol X_i\mathbf1_{E_i^c}-\mu_i,\qquad
 T_i=\Cov(Y_i),\qquad \mu=\sum_i\mu_i.
\]
Put $\epsilon_n=2K_4C_1/n$. For $n$ sufficiently large,
\begin{align}
 \sum_i\mathbb P(E_i)&\le K_4C_1^2/n,
 &\|Y_i\|&\le2\sqrt r,\label{new:trunc-tail}\\
 (1-\epsilon_n)\Sigma_i&\preceq T_i\preceq\Sigma_i,
 &\|\Sigma^{-1/2}\mu\|&\le K_4C_1^{3/2}/n.
 \label{new:trunc-cov}
\end{align}
Moreover,
\begin{equation}\label{new:trunc-fourth}
 \mathbb E|u^\top Y_i|^4
 \le16K_4(1-\epsilon_n)^{-2}(u^\top T_i u)^2.
\end{equation}
\end{lemma}
\begin{proof}
For any real matrix $M$, Minkowski's inequality applied to squared
row projections gives
\begin{equation}\label{new:matrix-fourth}
 \mathbb E\|M\boldsymbol X_i\|^4
 \le K_4\{\tr(M\Sigma_iM^\top)\}^2.
\end{equation}
In particular,
\[
 \mathbb E\|\boldsymbol X_i\|^4
 \le K_4(\tr\Sigma_i)^2\le K_4C_1^2r^2/n^2.
\]
Consequently $p_i:=\mathbb P(E_i)\le K_4C_1^2/n^2$.
Since $\|\mu_i\|\le\sqrt r$, the norm bound also follows.
Centering of $\boldsymbol X_i$ gives
$\mu_i=-\mathbb E[\boldsymbol X_i\mathbf1_{E_i}]$ and
\[
 \Sigma_i-T_i
 =\mathbb E[\boldsymbol X_i\boldsymbol X_i^\top\mathbf1_{E_i}]
     +\mu_i\mu_i^\top.
\]
For every real $u$, Cauchy--Schwarz yields
\[
 \mathbb E[(u^\top\boldsymbol X_i)^2\mathbf1_{E_i}]
 \le\sqrt{K_4p_i}\,u^\top\Sigma_i u
 \le\frac{K_4C_1}{n}u^\top\Sigma_i u.
\]
The square $(u^\top\mu_i)^2$ is bounded by the same expectation.
This proves the covariance inequalities. For the mean, H\"older's
inequality followed by Cauchy--Schwarz over the index $i$ gives
\begin{align*}
 |u^\top\mu|
 &\le K_4^{1/4}\sum_i(u^\top\Sigma_i u)^{1/2}p_i^{3/4}\\
 &\le\frac{K_4C_1^{3/2}}{n^{3/2}}
             \sum_i(u^\top\Sigma_i u)^{1/2}
 \le\frac{K_4C_1^{3/2}}n(u^\top\Sigma u)^{1/2}.
\end{align*}
Duality proves the mean bound in \eqref{new:trunc-cov}.
Finally, Minkowski's and Jensen's inequalities imply
$\mathbb E|u^\top Y_i|^4\le16\mathbb E|u^\top\boldsymbol X_i|^4$;
the lower covariance comparison gives \eqref{new:trunc-fourth}.
\end{proof}

The squared-norm density estimate in the next lemma follows from
\cite[Theorem~1]{BobkovNaumovUlyanov2021}, which gives two-sided
bounds for the maximum density of a weighted sum of chi-square
variables; see also \cite[Theorem~2.6]{GNSU2019} for the upper bound.
We show how to convert the squared-norm density
bound into the norm-density bound needed to transfer the centering
error of the truncated, non-Gaussian sum. For the covariance
comparison, we use \cite[Corollary~2.3]{GNSU2019}.
\begin{lemma}
\label{lem:undo-truncation}
Let $G\sim N(0,\Sigma)$ with $\|\Sigma\|=1$ and
$v=\tr(\Sigma^2)\ge11/2$. Write $r=\tr\Sigma$.
The law of $\|G\|$ has density $f_R$ satisfying
\begin{equation}\label{eq:gaussian-norm-density}
 \|f_R\|_\infty\le C\sqrt{r/v}.
\end{equation}
If $(1-\epsilon)\Sigma\preceq\Sigma'\preceq\Sigma$,
$0\le\epsilon\le1/2$, and $G'\sim N(0,\Sigma')$, then
\begin{equation}\label{eq:gaussian-relative-covariance}
 \sup_{R\ge0}\left|\mathbb P\{\|G'\|\le R\}
                   -\mathbb P\{\|G\|\le R\}\right|
 \le C\epsilon r/\sqrt v.
\end{equation}
The constants are numerical. In addition, the density $g$ of
$\|G\|^2/2$ satisfies $\|g\|_\infty\le C/\sqrt v$.
\end{lemma}
\begin{proof}
Let $s_1\ge s_2\ge\cdots>0$ be the positive eigenvalues of
$\Sigma$, so $s_1=1$. The upper bound in
\cite[Theorem~1]{BobkovNaumovUlyanov2021}, applied to
$\|G\|^2=\sum_j s_jZ_j^2$ with independent standard normal $Z_j$,
gives, after rescaling, for the density $g$ of $Q=\|G\|^2/2$,
\[
 \|g\|_\infty
 \le C\left\{\Bigl(\sum_j s_j^2\Bigr)
                   \Bigl(\sum_{j\ge2}s_j^2\Bigr)\right\}^{-1/4}
 =C\{v(v-1)\}^{-1/4}\le C/\sqrt v.
\]
The last inequality uses $v\ge11/2$. The change of variables
$f_R(R)=R g(R^2/2)$ alone does not give a uniform bound on $f_R$,
because its prefactor is unbounded. The following elementary
identity supplies the additional control.

Let $e_j$ be the exponential density of mean $s_j$.
The gamma density identity for the summand $s_jZ_j^2/2$, whose
shape is $1/2$ and scale is $s_j$, yields
\begin{equation}\label{eq:gaussian-size-bias}
 xg(x)=\frac12\sum_js_j(g*e_j)(x).
\end{equation}
To see this, multiplying the density of that summand by its
argument increases the gamma shape from $1/2$ to $3/2$ and
multiplies its density by $s_j/2$. Increasing the shape by one
is equivalent to adding an independent exponential variable of
mean $s_j$. Convolve with the other independent summands and sum
over $j$ to obtain \eqref{eq:gaussian-size-bias}. Hence
\[
 \sup_{x\ge0}xg(x)\le\frac r2\|g\|_\infty.
\]
Since $f_R(R)=R g(R^2/2)$ for $R>0$, putting $x=R^2/2$ gives
\[
 f_R(R)=\sqrt{2\{xg(x)\}g(x)}
 \le\sqrt r\|g\|_\infty\le C\sqrt{r/v}.
\]
This proves \eqref{eq:gaussian-norm-density}.

For the covariance comparison, put
\[
 \kappa_*(A)=\left\{\tr(A^2)
               \bigl[\tr(A^2)-\|A\|^2\bigr]\right\}^{-1/4}.
\]
Corollary~2.3 of \cite{GNSU2019}, together with their upper bound
on the spectral factor in Corollary~2.2, gives
\[
 \sup_{R\ge0}\left|\mathbb P(\|G'\|\le R)-\mathbb P(\|G\|\le R)\right|
 \le C\{\kappa_*(\Sigma)+\kappa_*(\Sigma')\}
             \|\Sigma-\Sigma'\|_1,
\]
where $\|\cdot\|_1$ is the nuclear norm. Here
$\kappa_*(\Sigma)=\{v(v-1)\}^{-1/4}\le C/\sqrt v$.
Eigenvalue monotonicity under
$(1-\epsilon)\Sigma\preceq\Sigma'\preceq\Sigma$ gives
$\kappa_*(\Sigma')\le(1-\epsilon)^{-1}\kappa_*(\Sigma)$.
Also $\Sigma-\Sigma'\succeq0$ and
$\|\Sigma-\Sigma'\|_1=\tr(\Sigma-\Sigma')\le\epsilon r$.
Since $\epsilon\le1/2$, these inequalities prove
\eqref{eq:gaussian-relative-covariance}.
\end{proof}

\begin{lemma}\label{lem:truncation-transfer}
Assume \textup{(\ref{ass:A1})--(\ref{ass:A3})} and let $Y_i,T_i,\mu$ be as in
Lemma~\ref{lem:new-truncation}. Put
\[
 \Sigma'=\sum_iT_i,\quad q=\|\Sigma'\|,\quad
 \widehat X_i=Y_i/\sqrt q,\quad
 \widehat\Sigma=\Sigma'/q,\quad
 \widehat r=\tr\widehat\Sigma,\quad
 \widehat v=\tr(\widehat\Sigma^2).
\]
For all sufficiently large $n$, the array $(\widehat X_i)$ satisfies
\textup{(\ref{ass:A1}), (\ref{ass:A2}), \ref{ass:A3-i}} with uniformly enlarged constants, and
$\|\widehat X_i\|\le3\sqrt{\widehat r}$ almost surely. With
$\epsilon_n=2K_4C_1/n$,
\begin{gather}
 (1-\epsilon_n)r\le\widehat r\le r/(1-\epsilon_n),\qquad
 (1-\epsilon_n)^2v\le\widehat v\le v/(1-\epsilon_n)^2,
 \label{eq:truncated-spectrum}\\
 \widehat v\ge11/2,\qquad
 s_6(\widehat\Sigma)\ge(1-\epsilon_n)s_6(\Sigma).
 \label{eq:truncated-sixth}
\end{gather}
If $\widehat\Delta_n(\mathcal B_0)$ denotes its centered-ball Gaussian approximation
error, then
\begin{equation}\label{eq:truncation-transfer}
 \Delta_n(\mathcal B_0)\le\widehat\Delta_n(\mathcal B_0)+\frac Cn+
                   C\frac r{n\sqrt v}.
\end{equation}
\end{lemma}
\begin{proof}
Lemma~\ref{lem:new-truncation} gives
$(1-\epsilon_n)\Sigma\preceq\Sigma'\preceq\Sigma$ and
$1-\epsilon_n\le q\le1$. Eigenvalue monotonicity proves
\eqref{eq:truncated-spectrum} and \eqref{eq:truncated-sixth}; the
bound $\widehat v\ge11/2$ follows from $v\ge6$ for sufficiently
large $n$. The balance constants increase by at most
$(1-\epsilon_n)^{-1}$, and \eqref{new:trunc-fourth} supplies the
directional fourth-moment bound. Moreover,
\[
 \frac{\widehat r}{\widehat v}
 =\frac{q\tr\Sigma'}{\tr((\Sigma')^2)}
 \le\frac{C_\Sigma}{(1-\epsilon_n)^2}.
\]
The norm bound follows from $\|Y_i\|\le2\sqrt r$ and
$q\widehat r=\tr\Sigma'\ge(1-\epsilon_n)r$.

Write $S_Y=\sum_iY_i$ and let $G'\sim N(0,\Sigma')$.
Scaling radii by $\sqrt q$ identifies their centered-ball distance
with $\widehat\Delta_n(\mathcal B_0)$. Outside an event of probability $C/n$,
$\boldsymbol S_n=S_Y+\mu$, and \eqref{new:trunc-cov} gives
$a:=\|\mu\|\le C/n$. The triangle inequality sandwiches the ball
probabilities of $S_Y+\mu$ between those of $S_Y$ at radii $R-a$
and $R+a$, taking the lower probability to be zero for $R<a$.
After comparison with $G'$, the radius shift costs at most
$a\|f_{\|G'\|}\|_\infty\le C/n$, by
Lemma~\ref{lem:undo-truncation} and $\widehat r/\widehat v\le C$.
Finally, \eqref{eq:gaussian-relative-covariance} costs at most
$Cr/(n\sqrt v)$. Adding these errors proves
\eqref{eq:truncation-transfer}.
\end{proof}


\subsection{The regime \texorpdfstring{$v\ge7$}{v >= 7}}
\label{sec:case-v-ge7}
We first prove an auxiliary bound under the stronger condition $v\ge13/2$.
This margin allows for the covariance loss when an array with
$v\ge7$ is truncated.
\begin{proposition}
\label{prop:bounded-array-main}
Suppose that \textup{(\ref{ass:A1})--(\ref{ass:A3})} hold, $v\ge13/2$, and
$\|\boldsymbol X_i\|\le L_0\sqrt r$ almost surely, where $L_0$ is independent of $n$ and $p$.
Then $\Delta_n(\mathcal B_0)\le Cr/n$, where $C$ depends only on
$C_1,K_4,C_\Sigma,L_0$.
\end{proposition}
\begin{proof}
Let $\Delta_n(\mathcal B_0)$ denote the supremum in Theorem~\ref{thm:2.1}.
For $n<N_0r$ the assertion follows from $\Delta_n(\mathcal B_0)\le1$.
It suffices to assume
\begin{equation}\label{eq:proof-nontrivial}
 n\ge N_0r,
\end{equation}
where $N_0$ is a sufficiently large constant. Choose
\begin{equation}\label{eq:proof-cutoffs}
 U=\frac{n}{r\sqrt v},\qquad
 \tau=C_H\frac{n^{1/3}}{\sqrt v},\qquad
 \ell=\min(U,\tau),
\end{equation}
with $C_H$ from Proposition~\ref{prop:high-frequency}.
The Kolmogorov distance is unchanged on replacing squared norms by
half squared norms. Lemma~\ref{lem:gaussian-smoothing} gives
\begin{equation}\label{eq:proof-smoothing}
 \Delta_n(\mathcal B_0)\le C\int_0^U
       \frac{|\mathfrak f(t)-\widetilde{\mathfrak f}(t)|}{t}\,dt
       +\frac C{U\sqrt v}.
\end{equation}
By the choice of $U$, the smoothing remainder is
\begin{equation}\label{eq:proof-smoothing-remainder}
 \frac{C}{U\sqrt v}=\frac{Cr}{n}.
\end{equation}

By the integrated assertion of Proposition~\ref{prop:moderate}, with
$n\ge N_0 r$ and $\ell\le\tau$, we have
\begin{equation}\label{eq:proof-moderate}
 \int_0^\ell
 \frac{|\mathfrak f(t)-\widetilde{\mathfrak f}(t)|}{t}\,dt
 \le \frac{Cr}{n}.
\end{equation}

If $U>\tau$, the inequality $r\ge13/2>4$ gives
$U\le n/(4\sqrt v)$, so Proposition~\ref{prop:high-frequency} gives
\begin{equation}\label{eq:proof-original-tail}
 \int_\tau^U|\mathfrak f(t)|\,\frac{dt}{t}\le C/n.
\end{equation}
For the Gaussian part, logarithmic concavity and Bernoulli's
inequality imply
\[
 D(t)\le(1+t^2)^{-v/4}
       \le(1+vt^2/6)^{-3/2},\qquad v\ge6.
\]
Consequently
\begin{equation}\label{eq:proof-gaussian-tail}
 \int_\tau^\infty D(t)\,\frac{dt}{t}
 \le\left(\frac6v\right)^{3/2}
             \int_\tau^\infty t^{-4}\,dt
 =\frac{2\sqrt6}{C_H^3n}.
\end{equation}
Combining \eqref{eq:proof-moderate}--\eqref{eq:proof-gaussian-tail}
with \eqref{eq:proof-smoothing} and
\eqref{eq:proof-smoothing-remainder} proves the proposition.
\end{proof}

\begin{proof}[Proof of Theorem~\ref{thm:2.1} when $v\ge7$]
By $\Delta_n(\mathcal B_0)\le1$, we may assume $n\ge N_0r$ with $N_0$
sufficiently large in terms of $C_1,K_4,C_\Sigma$.
Apply Lemma~\ref{lem:truncation-transfer}. Increasing $N_0$ ensures
\[
 \widehat v\ge(1-\epsilon_n)^2v\ge13/2.
\]
Proposition~\ref{prop:bounded-array-main} therefore applies to the
normalized truncated array, with uniformly enlarged constants, and
gives $\widehat\Delta_n(\mathcal B_0)\le C\widehat r/n\le Cr/n$
by \eqref{eq:truncated-spectrum}. The transfer cost in
\eqref{eq:truncation-transfer} is at most $C/n+Cr/(n\sqrt v)\le Cr/n$,
since $r\ge v\ge7$. This proves the theorem in the first regime.
\end{proof}

\subsection{The endpoint regime \texorpdfstring{$6\le v<7$}{6 <= v < 7}}
\label{sec:case-v-near6}
\begin{proof}[Proof of Theorem~\ref{thm:2.1} when $6\le v<7$]
Now $6\le r\le7C_\Sigma$, so $r$ is bounded in terms of the
assumption constants. Lemma~\ref{reserve6:sixth-eigenvalue} gives
$s_6(\Sigma)\ge(6C_\Sigma-5)^{-1}$. We may assume $n$ sufficiently
large, absorbing the remaining finite range into the constant.
Apply Lemma~\ref{lem:truncation-transfer}. Its cost is $C/n$, and
the normalized bounded array satisfies
\[
 \widehat r\le8C_\Sigma,\qquad
 \widehat v\ge6-C/n\ge11/2,\qquad
 s_6(\widehat\Sigma)\ge\frac1{2(6C_\Sigma-5)}.
\]
Thus the local assumptions \eqref{endpoint6:local-assumptions} hold
with constants independent of $n$ and $p$, even if $\widehat v<6$.

Let $\widehat{\mathfrak f}$ and $\widehat{\mathfrak g}$ be the
characteristic functions of the half squared norms of this array
and its Gaussian counterpart. Choose $\tau=T_0n^{1/3}$, with $T_0$
large enough that
$\tau\ge C_Hn^{1/3}/\sqrt{\widehat v}$ in
Proposition~\ref{prop:high-frequency}. The two comparison results in
Appendix~\ref{sec:reserve6}, Propositions~\ref{endpoint6:first-comparison}
and~\ref{reserve6:comparison}, give
\begin{equation}\label{eq:near6-moderate}
 \int_0^\tau
 \frac{|\widehat{\mathfrak f}(t)-\widehat{\mathfrak g}(t)|}{t}\,dt
 \le C/n.
\end{equation}
The second proposition bounds the integral over $(0,\infty)$,
so it places no restriction on $T_0$.

Set $U=n/(\widehat r\sqrt{\widehat v})$. For sufficiently large $n$, $U\ge\tau$,
and $U\le n/(4\sqrt{\widehat v})$ since
$\widehat r\ge\widehat v\ge11/2$. Proposition~\ref{prop:high-frequency}
therefore gives
\[
 \int_\tau^U|\widehat{\mathfrak f}(t)|\,\frac{dt}{t}\le C/n.
\]
The lower bound on the sixth eigenvalue gives
$|\widehat{\mathfrak g}(t)|\le C(1+t^2)^{-3/2}$, and hence its tail
integral from $\tau$ is at most $C\tau^{-3}=C/n$.
Esseen smoothing in Lemma~\ref{lem:gaussian-smoothing}, whose
remainder is $C/(U\sqrt{\widehat v})\le C/n$, now yields
$\widehat\Delta_n(\mathcal B_0)\le C/n$. Adding the truncation transfer cost gives
$\Delta_n(\mathcal B_0)\le C/n\le Cr/n$ in this regime, completing the proof.
\end{proof}

\section{Proofs of the Propositions from Section \ref{sec: Discussion}}
\label{sec:counterexample-proofs}

\begin{proof}[Proof of Proposition \ref{prop:ce-without-A3}]
Set $p_n=k+n^2$ and
\begin{equation}\label{eq:ce-A3-covariance}
 \Sigma_n=\diag\bigl(I_k,n^{-3/2}I_{n^2}\bigr).
\end{equation}
Then $\|\Sigma_n\|=1$, $r_n=k+\sqrt n$, and $v_n=k+1/n$.
In particular, $r_n/v_n\to\infty$, while $v_n\ge6$ whenever
$k\ge6$.

Let $\boldsymbol Z_{n,1},\ldots,\boldsymbol Z_{n,n}$ be independent
$\mathcal N(0,\Sigma_n)$ vectors. Independently of these vectors,
let $\eta_1,\ldots,\eta_n$ be independent random variables satisfying
\[
 \mathbb P(\eta_i=0)=\mathbb P(\eta_i=\sqrt2)=\frac12.
\]
Define
\[
 \boldsymbol X_{n,i}=\frac{\eta_i\boldsymbol Z_{n,i}}{\sqrt n}.
\]
The vectors are centered, and $\mathbb E\eta_i^2=1$ gives
\[
 \Sigma_{n,i}=\Var(\boldsymbol X_{n,i})=\frac{\Sigma_n}{n},
 \qquad \Var(\boldsymbol S_n)=\Sigma_n.
\]
Both inequalities in \textup{(\ref{ass:A1})} hold with $C_1=1$. Since
$\mathbb E\eta_i^4=2$, the Gaussian fourth-moment identity gives
\[
 \mathbb E|u^\top\boldsymbol X_{n,i}|^4
 =\frac{6}{n^2}(u^\top\Sigma_nu)^2
 =6(u^\top\Sigma_{n,i}u)^2.
\]
Thus \textup{(\ref{ass:A2})} holds with $K_4=6$. For every real $t$, we also have
\begin{equation}\label{eq:ce-A3-subGaussian}
 \mathbb E\exp(tu^\top\boldsymbol X_{n,i})
 =\frac12+\frac12\exp\bigl(t^2u^\top\Sigma_{n,i}u\bigr)
 \le\exp\bigl(t^2u^\top\Sigma_{n,i}u\bigr).
\end{equation}
Thus the increments in this example also have uniformly sub-Gaussian
one-dimensional marginals, relative to their own variances.

Put
\[
 T_n=\frac1n\sum_{i=1}^n\eta_i^2.
\]
Conditionally on $\eta_1,\ldots,\eta_n$, the sum
$\boldsymbol S_n$ is Gaussian with covariance $T_n\Sigma_n$.
Consequently, on an auxiliary probability space we may represent
\begin{equation}\label{eq:ce-A3-mixture}
 \|\boldsymbol S_n\|^2\overset{d}=T_nQ_n,\qquad
 Q_n\overset{d}=\|\widetilde{\boldsymbol S}_n\|^2,
\end{equation}
where $T_n$ and $Q_n$ are independent. By
\eqref{eq:ce-A3-covariance},
\[
 Q_n\overset{d}=Y_k+n^{-3/2}W_n,\qquad
 Y_k\sim\chi_k^2,\quad W_n\sim\chi_{n^2}^2,
\]
with $Y_k$ and $W_n$ independent. Since
\[
 \mathbb E(n^{-3/2}W_n)=\sqrt n,\qquad
 \Var(n^{-3/2}W_n)=\frac2n,
\]
we obtain
\begin{equation}\label{eq:ce-A3-Gaussian-limit}
 Q_n-\sqrt n\ \xrightarrow{d}\ Y_k.
\end{equation}
The variables $\eta_i^2-1$ are Rademacher. The one-dimensional central
limit theorem therefore yields
\[
 \sqrt n(T_n-1)=\frac1{\sqrt n}\sum_{i=1}^n(\eta_i^2-1)
 \ \xrightarrow{d}\ Z,\qquad Z\sim\mathcal N(0,1).
\]
Independence in \eqref{eq:ce-A3-mixture} gives joint convergence of
$Q_n-\sqrt n$ and $\sqrt n(T_n-1)$ to independent $Y_k$ and $Z$.
Since $T_n\to1$ in probability, Slutsky's theorem gives
\begin{equation}\label{eq:ce-A3-sum-limit}
 T_nQ_n-\sqrt n
 =T_n(Q_n-\sqrt n)+\sqrt n(T_n-1)
 \ \xrightarrow{d}\ Y_k+Z.
\end{equation}
Both limiting distribution functions are continuous at zero. At the
squared radius $z_n=\sqrt n$, it follows that
\[
 \mathbb P(\|\widetilde{\boldsymbol S}_n\|^2\le\sqrt n)
 \longrightarrow\mathbb P(Y_k\le0)=0,
\]
whereas
\[
 \mathbb P(\|\boldsymbol S_n\|^2\le\sqrt n)
 \longrightarrow\mathbb P(Y_k+Z\le0)
 =\mathbb E\Phi(-Y_k)=c_k>0.
\]
Hence $\liminf_n\Delta_n(\mathcal B_0)\ge c_k$. Finally,
\[
 \frac{r_n^{3/2}}n
 =\frac{(k+\sqrt n)^{3/2}}n
 \sim n^{-1/4}\longrightarrow0,
\]
which proves \eqref{eq:ce-A3-conclusion}.
\end{proof}

\begin{proof}[Proof of Proposition \ref{prop:ce-v4}]
We have $\Sigma_{n,i}=I_4/n$ and $\Sigma_n=I_4$. Thus
\textup{(\ref{ass:A1}), \ref{ass:A3-i}} hold with
$C_1=C_\Sigma=1$, while $v_n=4$ violates \textup{\ref{ass:A3-ii}}. For every
$u\in\mathbb R^4$,
\[
 \mathbb E|u^\top\boldsymbol X_{n,i}|^4
 =\frac1{n^2}\left(3\|u\|^4-2\sum_{j=1}^4u_j^4\right)
 \le3(u^\top\Sigma_{n,i}u)^2,
\]
which verifies \textup{(\ref{ass:A2})} with $K_4=3$. The vectors
$\boldsymbol\xi_i$ have deterministic norm two. Their projections
also satisfy, for every $t\in\mathbb R$,
\[
 \mathbb E\exp(tu^\top\boldsymbol X_{n,i})
 =\prod_{j=1}^4\cosh\left(\frac{tu_j}{\sqrt n}\right)
 \le\exp\left(\frac{t^2\|u\|^2}{2n}\right).
\]

Take $n=2m$, where $m$ is an odd positive integer, and define
\[
 K_j=\frac12\sum_{i=1}^{2m}\varepsilon_{ij}\in\mathbb Z,
 \qquad j=1,\ldots,4.
\]
The variables $K_1,\ldots,K_4$ are independent, and
\begin{equation}\label{eq:ce-v4-norm}
 \|\boldsymbol S_{2m}\|^2=\frac2m\sum_{j=1}^4K_j^2.
\end{equation}
We first establish the uniform local lower bound
\begin{equation}\label{eq:ce-v4-local}
 \mathbb P(K_j=a)=2^{-2m}\binom{2m}{m+a}
 \ge\frac{c}{\sqrt m},\qquad |a|\le\sqrt m,\quad a\in\mathbb Z,
\end{equation}
for sufficiently large $m$, with an absolute $c>0$. Stirling's
formula gives $2^{-2m}\binom{2m}{m}\ge c_1m^{-1/2}$. By symmetry,
it suffices to consider $a\ge0$, for which
\[
 \frac{\mathbb P(K_j=a)}{\mathbb P(K_j=0)}
 =\prod_{\ell=1}^a\frac{m-\ell+1}{m+\ell}
 =\prod_{\ell=1}^a\left(1-\frac{2\ell-1}{m+\ell}\right).
\]
For $m\ge16$ and $a\le\sqrt m$, each fraction in the last product
lies in $[0,1/2]$. Using $\log(1-x)\ge-2x$ on this interval gives
\[
 \log\frac{\mathbb P(K_j=a)}{\mathbb P(K_j=0)}
 \ge-\frac2m\sum_{\ell=1}^a(2\ell-1)
 =-\frac{2a^2}{m}\ge-2.
\]
This proves \eqref{eq:ce-v4-local}.

Let
\[
 r_4(m)=\#\{a\in\mathbb Z^4:a_1^2+a_2^2+a_3^2+a_4^2=m\}.
\]
Every vector counted by $r_4(m)$ satisfies $|a_j|\le\sqrt m$.
Independence, \eqref{eq:ce-v4-norm}, and \eqref{eq:ce-v4-local}
therefore imply
\begin{equation}\label{eq:ce-v4-atom}
 \mathbb P(\|\boldsymbol S_{2m}\|^2=2)
 =\sum_{\substack{a\in\mathbb Z^4\\\sum_{j=1}^4a_j^2=m}}
      \prod_{j=1}^4\mathbb P(K_j=a_j)
 \ge\frac{c^4}{m^2}r_4(m).
\end{equation}
Jacobi's four-square identity, with signs and order included, states
that
\[
 r_4(m)=8\sum_{\substack{d\mid m\\4\nmid d}}d;
\]
see \cite{Hirschhorn1982}. Since $m$ is odd, this reduces to
\begin{equation}\label{eq:ce-v4-Jacobi}
 r_4(m)=8\sigma_1(m),\qquad \sigma_1(m)=\sum_{d\mid m}d.
\end{equation}

If a distribution function $F$ has a jump of size $a$ at $x$ and
$H$ is continuous at $x$, then
\[
 a=F(x)-F(x-)
 \le |F(x)-H(x)|+|F(x-)-H(x)|
 \le2\sup_y|F(y)-H(y)|.
\]
The target squared norm has the continuous $\chi_4^2$ distribution.
Applying this observation to \eqref{eq:ce-v4-atom} and
\eqref{eq:ce-v4-Jacobi} gives, with an absolute constant $c_2>0$,
\begin{equation}\label{eq:ce-v4-lower}
 \Delta_{2m}(\mathcal B_0)\ge c_2\frac{\sigma_1(m)}{m^2},
 \qquad (2m)\Delta_{2m}(\mathcal B_0)\ge2c_2\frac{\sigma_1(m)}m.
\end{equation}

To make the last ratio unbounded using only elementary estimates,
choose
\[
 m_M=\prod_{j=1}^M(2j-1),\qquad n_M=2m_M.
\]
These integers $m_M$ are odd, and each $2j-1$, $j\le M$, divides
$m_M$. The divisor substitution $d\mapsto m_M/d$ yields
\[
 \frac{\sigma_1(m_M)}{m_M}
 =\sum_{d\mid m_M}\frac1d
 \ge\sum_{j=1}^M\frac1{2j-1}\longrightarrow\infty.
\]
It follows from \eqref{eq:ce-v4-lower} that
\[
 n_M\Delta_{n_M}(\mathcal B_0)\longrightarrow\infty,
\]
which proves \eqref{eq:ce-v4-conclusion}. Since $r_n=4$, any
bound of the form \eqref{eq:main-bound} would instead imply
$n\Delta_n(\mathcal B_0)\le4C$, a contradiction.
\end{proof}

\begin{proof}[Proof of Proposition \ref{prop:effective-rank-lower-bound} ]
The covariance and moment identities
\[
 \Sigma_i=I_d/n,\qquad
 \mathbb E|u^\top\boldsymbol X_i|^4
       =6\{u^\top\Sigma_i u\}^2
\]
verify the assumptions. Conditional on the multipliers,
\[
 \|\boldsymbol S_n\|^2\overset{\mathrm d}=T_nQ_d,
 \qquad T_n=\frac1n\sum_{i=1}^n\eta_i^2,
 \qquad Q_d\sim\chi_d^2,
\]
with $T_n$ and $Q_d$ independent. The variables $\eta_i^2-1$
are Rademacher, so $T_n-1$ is symmetric and
\[
 \mathbb E(T_n-1)^2=\frac1n,\qquad
 \mathbb E(T_n-1)^4=\frac{3n^2-2n}{n^4}\le\frac3{n^2}.
\]
Write $F_d,f_d$ for the distribution function and density of $Q_d$.
Put $z_d=d+\sqrt{2d}$ and $h_d(t)=F_d(z_d/t)$ for $t>0$,
with $h_d(0)=1$. Since
$f_d'(z)=\{(d-2)/(2z)-1/2\}f_d(z)$,
\[
 h_d''(1)=2z_df_d(z_d)+z_d^2f_d'(z_d)
         =\frac{z_d(d+2-z_d)}2 f_d(z_d).
\]
For $j=0,1,2,3$ and all sufficiently large $d$, Fourier inversion
of the chi-square characteristic function gives
\[
 \|f_d^{(j)}\|_\infty
 \le\frac1{2\pi}\int_{\mathbb R}|t|^j(1+4t^2)^{-d/4}\,dt
 \le C_jd^{-(j+1)/2}.
\]
Differentiating $h_d$ four times therefore yields
$\sup_{1/2\le t\le3/2}|h_d^{(4)}(t)|\le Cd^2$.
Taylor expansion through degree three on $|T_n-1|\le1/2$,
symmetry, and the bound
$\mathbb P(|T_n-1|>1/2)\le2e^{-n/8}$ give, for $d\le n$,
\begin{align*}
 \mathbb E h_d(T_n)-h_d(1)
 &=\frac{h_d''(1)}{2n}
       +O\left(\frac{d^2}{n^2}+e^{-cn}\right)\\
 &=\frac{z_d(d+2-z_d)f_d(z_d)}{4n}
       +O\left(\frac{d^2}{n^2}+e^{-cn}\right).
\end{align*}
The exponentially small term also absorbs the omitted Taylor
moments on $|T_n-1|>1/2$, since $d\le n$.
By Stirling's formula,
$\sqrt{2d}\,f_d(d+\sqrt{2d})\to\phi(1)$. Consequently,
\[
 \mathbb P\{\|\boldsymbol S_n\|^2\le z_d\}
       -\mathbb P\{Q_d\le z_d\}
 =-\left\{\frac{\phi(1)}4+o(1)\right\}\frac dn,
\]
because $d=o(n)$. Taking absolute values proves
\eqref{eq:effective-rank-lower-bound}.

To check the upper bound for the two mixtures, write
$P_t=N(0,tI_d)$, with $P_0$ the point mass at zero. In both cases
$\boldsymbol S_n$ has law $\mathbb E P_{T_n}$, where
$T_n=n^{-1}\sum_i\eta_i^2$ satisfies $\mathbb E T_n=1$ and
$\Var(T_n)=\Var(\eta_1^2)/n$. For $t>0$ let $q_t$ be the density of
$P_t$. Differentiating in $t$ and using $Q\sim\chi_d^2$ gives
\[
 \|\partial_t q_1\|_1\le\sqrt{d/2},\qquad
 \|\partial_t^2q_t\|_1
 =\frac{\mathbb E|(Q-d)^2-4(Q-d)-2d|}{4t^2}
 \le\frac{Cd}{t^2}.
\]
Thus, uniformly in Borel sets $B$,
$|P_t(B)-P_1(B)-(t-1)P'_1(B)|\le Cd(t-1)^2$:
use Taylor's formula for $t\ge1/2$, and use $|t-1|\ge1/2$ and
$|P'_1(B)|\le\sqrt{d/2}$ for $0\le t<1/2$.
Averaging cancels the linear term and gives $Cd\,\Var(T_n)$, as claimed.
\end{proof}

\begin{appendix}
\section{Real characteristic-function estimates}\label{sec:appendixA}
The goal of this appendix is Proposition~\ref{prop:moderate}, which bounds the
characteristic-function difference at moderate frequencies in the proof of
Theorem~\ref{thm:2.1} in Section~\ref{sec:case-v-ge7}.
All vector arguments and all shifts in the characteristic functions
in this appendix are real. Complex matrices occur only as explicit
resolvents after exact Gaussian averaging.
After the truncation reduction in Section~\ref{sec:proofs}, the
non-Gaussian increments in this appendix satisfy
$\|\boldsymbol X_i\|\le L_0\sqrt r$ almost surely, where $L_0$ is independent of $n$ and $p$.

\begin{mixedsumdefinition}[mixed sums]
Write $\Sigma=\sum_{i=1}^n\Sigma_i$, $\|\Sigma\|=1$ and
$r=\tr\Sigma$, and let $C_1$ be the constant in \textup{(\ref{ass:A1})}.
For $B,L>0$ and $K\ge3$, let $\mathcal M(B,L,K)$ be the class of
sums $W=\sum_{j=1}^{N} V_j$ of independent centered random vectors
in $\mathbb R^p$, where $N\ge0$ is any integer. With $T_j=\Cov(V_j)$
we require
\begin{equation*}
 \tag{M}\label{eq:mixed-sum}
 \|T_j\|\le\frac{BC_1}{n},\qquad
 \tr T_j\le\frac{BC_1r}{n}\quad(1\le j\le N),\qquad
 \sum_{j=1}^{N}T_j\preceq B\Sigma.
\end{equation*}
Each $V_j$ is either Gaussian or satisfies $\|V_j\|\le L\sqrt r$
almost surely. In both cases we require
\[
 \mathbb E|u^\top V_j|^4\le K(u^\top T_ju)^2
 \qquad(u\in\mathbb R^p).
\]
We call $W$ a mixed sum. The constants in estimates for this class may depend on $B,L,K$ but not on $N$.
\end{mixedsumdefinition}

For example, let $G_i\sim N(0,\Sigma_i)$ be independent of one
another and of the original array. For disjoint $I,J\subseteq\{1,\ldots,n\}$, the mixed sum belongs to the following class:
\[
 \sum_{i\in I}\boldsymbol X_i+\sum_{i\in J}G_i\mathrel{\in}\mathcal M(1,L_0,\max\{K_4,3\}).
\]
This follows from \textup{(\ref{ass:A1})--(\ref{ass:A2})}.
If $W'$ is an independent copy of $W\in\mathcal M(B,L,K)$, then
$W-W'\in\mathcal M(2B,L,K)$: the individual covariance matrices are
unchanged and their sum is doubled.
Higher projection moments of the bounded non-Gaussian
increments follow from truncation and \textup{(\ref{ass:A2})}. The replacement
remainders use moments of order at most four of each original
direction increment; products of third moments involve distinct
independent increments.
Fix
\begin{equation}\label{eq:A-margin-parameters}
 v_0=6+\delta,\qquad
 \varepsilon=\frac{\delta}{8v_0},\qquad
 \kappa=1-4\varepsilon=1-\frac{\delta}{2v_0},\qquad
 c_*=\frac1{8\cdot14^2}.
\end{equation}
Set $\delta=1/2$.
The restriction $v\ge v_0=13/2$ is used in the first spectral
regime, since Lemma~\ref{lem:truncation-transfer} ensures
$\widehat v\ge(1-\epsilon_n)^2v\ge13/2$ for all sufficiently large
$n$ when the original array has $v\ge7$.
The range $6\le v<7$ is treated separately in
Appendix~\ref{sec:reserve6}, so fixing $\delta$ does not impose a
positive gap between the original $v$ and six.
Thus $0<\varepsilon<1/8$, $1/2<\kappa<1$, and
$\kappa v_0/4=3/2+\delta/8$. These choices leave a positive margin
in the weighted integrals after the covariance partition.

\begin{lemma}
\label{lem:gaussian-smoothing}
Suppose $\Sigma\succ0$, $\|\Sigma\|=1$, and
$v=\tr(\Sigma^2)\ge11/2$. Let $F_G$ and
$\widetilde{\mathfrak f}$ be the distribution and characteristic
functions of $\|\widetilde{\boldsymbol S}\|^2/2$, where
$\widetilde{\boldsymbol S}\sim\mathcal N(0,\Sigma)$.
Then $F_G$ has a bounded continuous density $g$, with
$\|g\|_\infty\le C/\sqrt v$. If $F$ is a distribution function with
finite first absolute moment and characteristic function $f$, then
for every $U>0$,
\begin{equation}\label{eq:esseen-gaussian-smoothing}
 \sup_x|F(x)-F_G(x)|
 \le C\int_0^U\frac{|f(t)-\widetilde{\mathfrak f}(t)|}{t}\,dt
       +\frac C{U\sqrt v}.
\end{equation}
\end{lemma}
\begin{proof}
Lemma~\ref{lem:undo-truncation} gives $\|g\|_\infty\le C/\sqrt v$.
The density has a continuous version by Fourier inversion, since
$|\widetilde{\mathfrak f}(t)|=\prod_j(1+t^2s_j^2)^{-1/4}
\le(1+t^2)^{-v/4}$ is integrable. Esseen's smoothing
inequality~\cite{Esseen1945} now gives \eqref{eq:esseen-gaussian-smoothing}
by conjugate symmetry; the finite first moments ensure that the
integrand is bounded near zero.
\end{proof}

\begin{lemma}
\label{lem:directional-moments}\label{lem:new-projection-moments}
Under \textup{(\ref{ass:A1})}, for every $W\in\mathcal M(B,L,K)$ and every $a\in\mathbb R^p$,
\begin{align}
 \mathbb E|a^\top W|^4
 &\le C(a^\top\Sigma a)^2,\label{new:fourth-projection}\\
 \mathbb E|a^\top W|^6
 &\le C\left\{(a^\top\Sigma a)^3
       +\frac rn\|a\|^4(a^\top\Sigma a)\right\}.
 \label{new:sixth-projection}
\end{align}
Consequently, suppose a deterministic complex matrix $R$ satisfies
\[
 \|R\|\le1,\qquad\|R\Sigma R^*\|\le b^2,
 \qquad b^4\ge r/n.
\]
Then for every real vector $u$ and $1\le q\le6$,
\begin{equation}\label{new:resolvent-projection}
 \|u^\top RW\|_{L^q}\le C\|u\|b.
\end{equation}
For $q\le4$, the condition $b^4\ge r/n$ is unnecessary.
\end{lemma}
\begin{proof}
Put $q_j=a^\top T_ja$ for $1\le j\le N$. The uniform fourth moment and the exact expansion
of the fourth moment of a sum of independent centered scalars imply
\[
 \mathbb E|a^\top W|^4
 =\sum_j\mathbb E|a^\top V_j|^4
        +6\sum_{j<k}q_jq_k
 \le C\Big(\sum_jq_j\Big)^2
 \le C(a^\top\Sigma a)^2.
\]
For a bounded constituent, the norm bound gives
\[
 \mathbb E|a^\top V_j|^6
 \le Cr\|a\|^2q_j^2.
\]
By \eqref{eq:mixed-sum}, $q_j\le BC_1\|a\|^2/n$ and
$\sum_jq_j\le B a^\top\Sigma a$. Thus the sum of these sixth moments is at most
$C(r/n)\|a\|^4(a^\top\Sigma a)$.
For Gaussian constituents the sum is at most
$15\sum_jq_j^3\le C(a^\top\Sigma a)^3$.
The nonzero index multiplicities in the sixth-power expansion of
the centered independent sum are $6$, $4+2$, $3+3$, and $2+2+2$.
The last three types contribute at most $C(\sum_jq_j)^3$:
use the fourth-moment bound and
$|\mathbb E(a^\top V_j)^3|\le
 (q_j\mathbb E|a^\top V_j|^4)^{1/2}\le Cq_j^{3/2}$.
This proves \eqref{new:sixth-projection}.
For the complex projection, apply these bounds separately to the
real and imaginary parts of $R^\top u$. Each corresponding variance
is at most $b^2\|u\|^2$, and each coefficient norm is at most
$\|u\|$. Hence
\[
 \mathbb E|u^\top RW|^6
 \le C\|u\|^6\{b^6+(r/n)b^2\}
 \le C\|u\|^6b^6.
\]
Lower moments follow by Lyapunov's inequality; for $q\le4$, apply
\eqref{new:fourth-projection} directly.
\end{proof}

For the independent Gaussian component kept in each replacement,
$R=(I-it\Sigma/2)^{-1}$ satisfies
$\|R\|\le1$ and
$\|R\Sigma R^*\|\le b(t)^2:=\min(1,|t|^{-1})$.
Thus \eqref{new:resolvent-projection} holds whenever
$n\ge r$ and $t^2r/n\le1$.
Together with $n\ge r$, the frequency condition
$h(t)=n/(t^2+|t|)\ge C(r+1)$ implies $t^2r/n\le1$,
and hence $b(t)^4\ge r/n$.
For a Gaussian prefix, the bound
$b_\alpha(t)^2=\min\{1,3/(2\alpha|t|)\}$ also satisfies
$b_\alpha(t)^4\ge r/n$ on this range, because $0\le\alpha\le1$.
Only orders at most four are needed in the first comparison.

The next bound is a consequence of the Banach-space
Fuk--Nagaev inequality of Einmahl and Li~\cite[Theorem~3.1]{EinmahlLi2007}.
The proof checks its application to our mixed sums, including the
polynomial term supplied by the sixth moments after truncation.

\begin{lemma}
\label{lem:new-quadratic-tail}
Assume \textup{(\ref{ass:A1})} and let $V\in\mathcal M(B,L,K)$ satisfy $\Cov(V)\preceq2\Sigma$.
For each $b>0$, there are constants $c_b,C_b>0$ such that
\begin{equation}\label{new:quadratic-tail}
 \mathbb P\{t^2\|V\|^2>bn\}
 \le e^{-c_bn/t^2}+C_b\frac{r^3|t|^6}{n^4},
 \qquad n/t^2\ge C_b(r+1).
\end{equation}
\end{lemma}
\begin{proof}
For each bounded constituent of covariance $T_j$,
\eqref{new:matrix-fourth} and its norm bound give
$\mathbb E\|V_j\|^6\le Cr(\tr T_j)^2$.
For a Gaussian constituent the corresponding bound is
$15(\tr T_j)^3$. Using
$\max_j\tr T_j\le Cr/n$ and $\sum_j\tr T_j\le Cr$ yields
\begin{equation}\label{new:aggregate-sixth}
 \sum_j\mathbb E\|V_j\|^6
 \le C\left(\frac{r^3}{n}+\frac{r^3}{n^2}\right)
 \le C\frac{r^3}{n}.
\end{equation}
Apply the Banach-space Fuk--Nagaev inequality of
\cite[Theorem~3.1]{EinmahlLi2007} with exponent six
in Euclidean space. Here the weak variance is
$\Lambda=\|\operatorname{Cov}(V)\|\le2$, while
$\mathbb E\|V\|\le\sqrt{2r}$. Taking both auxiliary parameters in the cited inequality equal
to one (these are unrelated to the choice of $\delta$ in
\eqref{eq:A-margin-parameters}) gives
\[
 \mathbb P\{\|V\|\ge2\mathbb E\|V\|+x\}
 \le e^{-x^2/(3\Lambda)}
      +Cx^{-6}\sum_j\mathbb E\|V_j\|^6.
\]
Set $a=\sqrt{bn}/|t|$. Increase $C_b$ so that
$2\sqrt{2r}\le a/2$, and take $x=a/2$.
Equation~\eqref{new:aggregate-sixth} then proves
\eqref{new:quadratic-tail}.
\end{proof}

\begin{lemma}\label{lem:A.3}
Let $\Sigma$ be positive semidefinite with eigenvalues $0\le s_j\le1$
and $v=\sum_js_j^2\ge v_0$. For every $c>0$,
\begin{equation}\label{eq:A-spectral-envelope}
 \det(I+ct^2\Sigma^2)^{-1/4}
 \le(1+ct^2)^{-v/4}\le(1+ct^2)^{-\kappa v/4}.
\end{equation}
Moreover, for every $q\in\{1,2,3\}$,
\begin{equation}\label{eq:A.3}
 \int_0^\infty t^{q-1}(1+ct^2)^{-\kappa v/4}\,dt
 \le C_{q,c,\delta}v^{-q/2}.
\end{equation}
\end{lemma}
\begin{proof}
Concavity gives $\log(1+ct^2x)\ge x\log(1+ct^2)$ for $0\le x\le1$.
Sum this inequality over $x=s_j^2$ and use $\kappa<1$ to obtain
\eqref{eq:A-spectral-envelope}.

For the integral, substitute $u=t\sqrt v$. Since $v/v_0\ge1$,
Bernoulli's inequality gives
\[
 (1+cu^2/v)^{v/v_0}\ge1+cu^2/v_0.
\]
Consequently,
\begin{align*}
 \int_0^\infty t^{q-1}(1+ct^2)^{-\kappa v/4}\,dt
 &\le v^{-q/2}\int_0^\infty
       u^{q-1}(1+cu^2/v_0)^{-\kappa v_0/4}\,du\\
 &=v^{-q/2}\int_0^\infty
       u^{q-1}(1+cu^2/v_0)^{-3/2-\delta/8}\,du.
\end{align*}
The last integral is finite: it is integrable at zero since $q>0$,
and its tail is of order $u^{q-4-\delta/4}$, with exponent strictly
less than $-1$ for $q\le3$. More explicitly, one may take
\[
 C_{q,c,\delta}=\frac12\left(\frac{v_0}{c}\right)^{q/2}
 B\left(\frac q2,\frac32+\frac\delta8-\frac q2\right).
\]
This proves \eqref{eq:A.3}.
\end{proof}

The next lemma implements the symmetrization and Gaussian
randomization method for the present nonidentical mixed sums.
The bound is expressed in terms of the covariances $Q$ and $R$;
the two additional terms come from bounding tail probabilities.

\begin{lemma}
\label{lem:shifted-cf}\label{lem:new-shifted-cf}
Assume \textup{(\ref{ass:A1})}. Let $A,C\in\mathcal M(B,L,K)$ be independent
and suppose $Q=\Cov(A)\preceq\Sigma$ and $R=\Cov(C)\preceq\Sigma$. Uniformly in $z\in\mathbb R^p$,
\begin{equation}\label{new:shifted-cf}
 \begin{gathered}
 \left|\mathbb E e^{it\|A+C+z\|^2/2}\right|
 \le\det\left(I+\frac{t^2}{2}Q^{1/2}RQ^{1/2}\right)^{-1/4}
   +Ce^{-cn/t^2}+C\frac{r^{3/2}|t|^3}{n^2},\\
 n/t^2\ge C(r+1).
 \end{gathered}
\end{equation}
\end{lemma}
\begin{proof}
For a constituent $V_j$ with covariance $T_j$, symmetrization and
$\cos x\le1-x^2/2+x^4/24$ give
\[
 |\varphi_{V_j}(u)|^2
 \le1-u^\top T_ju+K(u^\top T_ju)^2.
\]
Because $\|T_j\|\le C/n$, this implies, after multiplication,
\begin{equation}\label{new:local-cf}
 |\varphi_A(u)|\le e^{-u^\top Qu/4},\qquad
 |\varphi_C(u)|\le e^{-u^\top Ru/4},
 \qquad \|u\|^2\le c_0n.
\end{equation}
The local domain is Euclidean; the potentially larger domain
$u^\top\Sigma u\le c_0n$ does not follow from \textup{(\ref{ass:A1})}.
Let $C'$ be an independent copy of $C$. Then
$\widetilde C=C-C'\in\mathcal M(2B,L,K)$ and
$\Cov(\widetilde C)=2R\preceq2\Sigma$, so
Lemma~\ref{lem:new-quadratic-tail} applies to $\widetilde C$.
Cauchy--Schwarz and integration over $A$ give
\[
 |\mathbb E e^{it\|A+C+z\|^2/2}|^2
 \le\mathbb E|\varphi_A(t\widetilde C)|
 \le\mathbb E e^{-t^2\widetilde C^\top Q\widetilde C/4}
       +Ce^{-cn/t^2}+C r^3|t|^6/n^4,
\]
where \eqref{new:quadratic-tail} controls
$t^2\|\widetilde C\|^2>c_0n$.
For an independent standard Gaussian vector $\gamma$,
\[
 \mathbb E e^{-t^2\widetilde C^\top Q\widetilde C/4}
 =\mathbb E_\gamma
   |\varphi_C(tQ^{1/2}\gamma/\sqrt2)|^2.
\]
Since $\|Q\|\le1$ and $\tr Q\le r$,
\[
 \mathbb E e^{\gamma^\top Q\gamma/4}
 =\det(I-Q/2)^{-1/2}\le e^{r/2}.
\]
Consequently
$\mathbb P\{t^2\gamma^\top Q\gamma/2>c_0n\}
\le e^{r/2-c_0n/(2t^2)}\le e^{-cn/t^2}$.
On the complementary event apply the squared second bound in
\eqref{new:local-cf}. Real Gaussian integration now gives
\[
 \mathbb E_\gamma|\varphi_C(tQ^{1/2}\gamma/\sqrt2)|^2
 \le\det\left(I+\frac{t^2}{2}Q^{1/2}RQ^{1/2}\right)^{-1/2}
        +e^{-cn/t^2}.
\]
Combining the last two bounds and taking square roots proves
\eqref{new:shifted-cf}. All Gaussian randomization integrands are
bounded, so their use of Fubini's theorem requires no further moments.
\end{proof}

In the subsequent lemmas we use $h(t)=n/(t^2+|t|)$.
The stronger restriction $h(t)\ge C(r+1)$ implies the domain
of \eqref{new:shifted-cf}, and its exponential error is bounded
by $Ce^{-ch(t)}$. The term $Cr^{3/2}|t|^3/n^2$ in \eqref{new:shifted-cf} is unchanged.

\begin{lemma}\label{lem:partition}
Assume \textup{(\ref{ass:A1})}, $\|\Sigma\|=1$, and write
$r=\operatorname{tr}\Sigma$. Let $\mathcal O\subseteq\{1,\ldots,n\}$
with $|\mathcal O|\le2$. Set $L=14$ and use $\varepsilon$ from
\eqref{eq:A-margin-parameters}. There is a constant $C$, depending only
on $C_1,\delta$, such that, if $n\ge Cr$, $\{1,\ldots,n\}\setminus\mathcal O$
admits a deterministic partition $I_1,\ldots,I_L$ for which the matrices
$Q_\ell=\sum_{i\in I_\ell}\Sigma_i$ satisfy
\begin{equation}\label{moderate:HS-partition}
 0\preceq Q_\ell\preceq\Sigma,\qquad
 \|LQ_\ell-\Sigma\|_{\mathrm{HS}}\le\varepsilon,
 \qquad 1\le\ell\le L.
\end{equation}
Here $\|M\|_{\mathrm{HS}}^2=\operatorname{tr}(M^\top M)$ is the
Hilbert--Schmidt norm. The partition depends only on the covariance
matrices and the omitted indices, so it works for every choice of
original or Gaussian replacement increments on the remaining indices.
\end{lemma}

\begin{proof}
Put
$E=\sum_{i\in\mathcal O}\Sigma_i$, and set $S=\Sigma-E$.
Assign each remaining index independently and uniformly to one of
$L$ classes, and let $Q_\ell$ be the corresponding random covariance
sums. For class $\ell$, the centered fluctuation is
$LQ_\ell-S=\sum_{i\notin\mathcal O}(L\mathbf1_{\{i\text{ in class }\ell\}}-1)\Sigma_i$.
The scalar coefficients are independent, centered, and have variance
$L-1$. Expanding the squared Hilbert--Schmidt norm therefore gives
\begin{align}
 \mathbb E\sum_{\ell=1}^{L}\|LQ_\ell-\Sigma\|_{\mathrm{HS}}^2
 &=L(L-1)\sum_{i\notin\mathcal O}\operatorname{tr}(\Sigma_i^2)
       +L\operatorname{tr}(E^2)\notag\\
 &\le\frac{L(L-1)C_1r}{n}
       +\frac{4LC_1^2r}{n^2}.
 \label{moderate:HS-partition-expectation}
\end{align}
To verify the inequality, \textup{(\ref{ass:A1})} and positivity imply
\[
 \operatorname{tr}(\Sigma_i^2)
 \le\frac{C_1}{n}\operatorname{tr}\Sigma_i,\qquad
 \sum_{i\notin\mathcal O}\operatorname{tr}\Sigma_i\le r.
\]
Also $\|E\|\le2C_1/n$ and $\tr E\le2C_1r/n$, so
$\operatorname{tr}(E^2)\le\|E\|\tr E\le4C_1^2r/n^2$.
The cross term with the deterministic matrix $E$ vanishes in
\eqref{moderate:HS-partition-expectation} because
$\mathbb E(LQ_\ell-S)=0$.

Since $r\ge1$, taking the constant in $n\ge Cr$ sufficiently large
makes the right side of \eqref{moderate:HS-partition-expectation} at most
$\varepsilon^2/2$. Hence there exists an assignment for which the sum
of the nonnegative squared errors is at most $\varepsilon^2$.
Each individual error is then at most $\varepsilon$, proving
\eqref{moderate:HS-partition}. Positivity and $Q_\ell\preceq\Sigma$
hold for every assignment.
\end{proof}

\begin{lemma}\label{lem:balanced-damping}
Let $\|\Sigma\|=1$, $v=\operatorname{tr}(\Sigma^2)\ge v_0$, and $L=14$.
Suppose $Q,R$ are positive semidefinite matrices satisfying
$\|LQ-\Sigma\|_{\mathrm{HS}},\|LR-\Sigma\|_{\mathrm{HS}}\le\varepsilon$,
with the parameters in \eqref{eq:A-margin-parameters}. Then
\begin{equation}\label{moderate:balanced-determinant}
 \det\left(I+\frac{t^2}{2}Q^{1/2}RQ^{1/2}\right)^{-1/4}
 \le(1+c_*t^2)^{-\kappa v/4},\qquad t\in\mathbb R.
\end{equation}
In particular, Lemma~\ref{lem:A.3} gives, for $q=1,\ldots,3$,
\begin{equation}\label{moderate:balanced-integrals}
 \int_0^\infty t^{q-1}(1+c_*t^2)^{-\kappa v/4}\,dt
 \le C_{q,\delta}v^{-q/2}.
\end{equation}
\end{lemma}
\begin{proof}
Write $LQ=\Sigma+F$, $LR=\Sigma+G$, and $K=Q^{1/2}RQ^{1/2}$.
The norm bounds and the Hilbert--Schmidt Cauchy--Schwarz inequality give
\[
 \|K\|\le b_0:=\frac{(1+\varepsilon)^2}{L^2},\qquad
 \operatorname{tr}K=\operatorname{tr}(QR)
 \ge\frac{v-2\varepsilon\sqrt v-\varepsilon^2}{L^2}.
\]
The trace estimate follows by expanding
$L^2\operatorname{tr}(QR)=v+\operatorname{tr}(\Sigma F)
+\operatorname{tr}(\Sigma G)+\operatorname{tr}(FG)$.
Since $v\ge1$,
\[
 \frac{\operatorname{tr}K}{b_0}
 \ge v\frac{1-2\varepsilon-\varepsilon^2}{(1+\varepsilon)^2}
 \ge(1-4\varepsilon)v=\kappa v.
\]
For the second inequality, subtracting
$(1-4\varepsilon)(1+\varepsilon)^2$ from the numerator gives
$6\varepsilon^2+4\varepsilon^3\ge0$.
Apply concavity to the eigenvalues $\lambda_j\in[0,b_0]$ of $K$:
\[
 \sum_j\log(1+t^2\lambda_j/2)
 \ge\frac{\operatorname{tr}K}{b_0}\log(1+b_0t^2/2)
 \ge\kappa v\log(1+c_*t^2),
\]
where $b_0/2\ge c_*$. Exponentiation proves
\eqref{moderate:balanced-determinant}; \eqref{moderate:balanced-integrals}
is \eqref{eq:A.3} with $c=c_*$.
\end{proof}

\begin{lemma}\label{lem:prefix-order}
Assume \textup{(\ref{ass:A1})}, let $\|\Sigma\|=1$, and put
$r=\operatorname{tr}\Sigma$, $v=\operatorname{tr}(\Sigma^2)$. There is a deterministic ordering of
the increments such that, after relabeling and writing
$Q_k=\sum_{i\le k}\Sigma_i$,
\begin{equation}\label{moderate:prefix-order}
 \max_{0\le k\le n}
 \left\|Q_k-\frac{k}{n}\Sigma\right\|_{\mathrm{HS}}
 \le 4\sqrt{\frac{C_1r}{n}}.
\end{equation}
The ordering depends only on the covariance matrices. Define
\begin{equation}\label{moderate:prefix-error}
 \eta_n=4\sqrt{C_1r/n}+\frac{2C_1\sqrt r}{n}
 \le6\sqrt{C_1r/n},\qquad n\ge C_1.
\end{equation}
For this order, every covariance $Q_k$ with at most two summands
removed differs from $(k/n)\Sigma$ in Hilbert--Schmidt norm by at
most $\eta_n$. The covariance of the original suffix
$\sum_{\ell>k+1}\boldsymbol X_\ell$ differs from
$(1-k/n)\Sigma$ by at most the same amount.
\end{lemma}

\begin{proof}
It suffices to consider $n\ge2$. Put $C_i=\Sigma_i-\Sigma/n$.
Then $\sum_i C_i=0$, and positivity and \textup{(\ref{ass:A1})} give
\[
 \sum_i\|C_i\|_{\mathrm{HS}}^2
 =\sum_i\operatorname{tr}(\Sigma_i^2)-v/n
 \le \frac{C_1r}{n}.
\]
For a uniform random permutation $\pi$, set
$T_k=\sum_{i\le k}C_{\pi(i)}$ and let $\mathcal F_k$ be generated
by its first $k$ entries. Sampling without replacement gives
\[
 \mathbb E\|T_k\|_{\mathrm{HS}}^2
 =\frac{k(n-k)}{n(n-1)}
       \sum_i\|C_i\|_{\mathrm{HS}}^2.
\]
Moreover, $T_k/(n-k)$, $0\le k<n$, is a martingale with values in
the Hilbert space of symmetric matrices: conditionally on
$\mathcal F_k$, the next matrix has mean $-T_k/(n-k)$.
For $m=\lfloor n/2\rfloor$, Doob's $L^2$ inequality therefore yields
\begin{align*}
 \mathbb E\max_{0\le k\le m}\|T_k\|_{\mathrm{HS}}^2
 &\le \frac{4n^2}{(n-m)^2}\mathbb E\|T_m\|_{\mathrm{HS}}^2\\
 &=\frac{4nm}{(n-1)(n-m)}
       \sum_i\|C_i\|_{\mathrm{HS}}^2
 \le8\sum_i\|C_i\|_{\mathrm{HS}}^2.
\end{align*}
Apply the same argument to the first $\lfloor n/2\rfloor$ entries
of the reversed permutation. Since $T_n=0$, these reversed sums
cover all the remaining prefixes, including when $n$ is odd.
Adding the two maximal bounds shows that
\[
 \mathbb E\max_{0\le k\le n}\|T_k\|_{\mathrm{HS}}^2
 \le16\sum_i\|C_i\|_{\mathrm{HS}}^2
 \le\frac{16C_1r}{n}.
\]
At least one permutation satisfies \eqref{moderate:prefix-order}.
Finally, \textup{(\ref{ass:A1})} implies
$\|\Sigma_i\|_{\mathrm{HS}}\le C_1\sqrt r/n$.
Deleting at most two summands, or subtracting the first omitted
index from a suffix, adds at most $2C_1\sqrt r/n$ to the prefix
error. The last inequality in \eqref{moderate:prefix-error} follows
when $n\ge C_1$.
\end{proof}

For mixed sums in the next two lemmas, we use $L_*>0$ and
$K_*\ge3$ in Definition~\eqref{eq:mixed-sum}: $L_*\sqrt r$ is the
almost-sure norm bound for each non-Gaussian summand, and $K_*$ is the
directional fourth-moment constant for all summands. By
Lemma~\ref{lem:third-moment-proxy}, $L_*=2L_0$ and $K_*=8K_4$ suffice
in our applications.

\begin{lemma}\label{lem:prefix-damping}
Assume \textup{(\ref{ass:A1})--(\ref{ass:A2})}, $\|\Sigma\|=1$ and $v\ge v_0$,
with $v_0,\kappa,c_*$ as in \eqref{eq:A-margin-parameters}.
Let $I\subseteq\{1,\ldots,n\}$, and let $(Y_i)_{i\in I}$ satisfy the
conditions defining $\mathcal M(1,L_*,K_*)$, with
$T_i:=\Cov(Y_i)\preceq\Sigma_i$ for every $i\in I$.
Put $W=\sum_{i\in I}Y_i$ and $S=\Cov(W)$.
Suppose $Q\succeq0$ and $\alpha\in[0,1]$ satisfy
\begin{equation}\label{moderate:partial-covariances}
 Q,S\preceq\Sigma,
 \qquad
 \|Q-\alpha\Sigma\|_{\mathrm{HS}}\le\eta_n,
 \qquad
 \|S-(1-\alpha)\Sigma\|_{\mathrm{HS}}\le\eta_n,
\end{equation}
where $\eta_n$ is given by \eqref{moderate:prefix-error}.
There exist $c,C>0$, depending only on $C_1,K_4,L_*,K_*$ and $\delta$,
such that the following hold for $n\ge Cr$.
\begin{enumerate}
\item If $\alpha\le1/2$, there is a deterministic partition of
$I$ into fourteen classes. These can be paired so that the sum $B$ over
each of the seven pairs satisfies
\[
 \sup_{z\in\mathbb R^p}
 \left|\mathbb E e^{it\|B+z\|^2/2}\right|
 \le H(t),\qquad h(t)\ge C(r+1),
\]
where
\begin{equation}\label{moderate:H-balanced}
 H(t)=(1+c_*t^2)^{-\kappa v/4}+C e^{-c h(t)}
       +C\frac{r^{3/2}|t|^3}{n^2},
 \qquad h(t)=\frac{n}{t^2+|t|}.
\end{equation}
\item If $\alpha\ge1/2$, then
\begin{equation}\label{moderate:prefix-determinant}
 \det(I+t^2Q^2)^{-1/4}
 \le(1+c_*t^2)^{-\kappa v/4},\qquad t\in\mathbb R.
\end{equation}
\end{enumerate}
\end{lemma}

\begin{proof}
For the first assertion, write $\beta=1-\alpha\ge1/2$.
Assign the indices of $W$ independently and uniformly to $L=14$
classes, whose covariance sums are $S_1,\ldots,S_L$. The centered
Hilbert--Schmidt calculation used in Lemma~\ref{lem:partition} gives
\[
 \mathbb E\sum_{\ell=1}^L\|LS_\ell-S\|_{\mathrm{HS}}^2
 =L(L-1)\sum_{i\in I}\tr(T_i^2)
 \le\frac{L(L-1)C_1r}{n}.
\]
The inequality follows from $\|T_i\|\le C_1/n$ and
$\sum_{i\in I}\tr T_i=\tr S\le r$. There is consequently an assignment with
\[
 \|LS_\ell-\beta\Sigma\|_{\mathrm{HS}}
 \le\eta_n+\sqrt{L(L-1)C_1r/n}
 \le\beta\varepsilon,\qquad 1\le\ell\le L,
\]
after increasing the constant in $n\ge Cr$.
Consider two classes with covariances $S_\ell,S_{\ell'}$ and put
$K=S_\ell^{1/2}S_{\ell'}S_\ell^{1/2}$. Write
$LS_\ell=\beta(\Sigma+F)$ and
$LS_{\ell'}=\beta(\Sigma+F')$, where
$\|F\|_{\mathrm{HS}},\|F'\|_{\mathrm{HS}}\le\varepsilon$.
Then
\[
 \|K\|\le b_0:=\frac{\beta^2(1+\varepsilon)^2}{L^2},
 \qquad
 \operatorname{tr}K
 \ge\frac{\beta^2}{L^2}
        (v-2\varepsilon\sqrt v-\varepsilon^2).
\]
The algebra in Lemma~\ref{lem:balanced-damping} gives
$\operatorname{tr}K/b_0\ge\kappa v$. Concavity of the logarithm
and $b_0/2\ge1/(8L^2)=c_*$ imply
\[
 \det(I+t^2K/2)^{-1/4}
 \le(1+c_*t^2)^{-\kappa v/4}.
\]
Pair the fourteen classes into seven blocks and apply
Lemma~\ref{lem:shifted-cf} to each pair. Its covariance matrices are
bounded by $\Sigma$, so the constants in both error terms are
uniform. This proves the first assertion.

For the second, $\|Q/\alpha-\Sigma\|_{\mathrm{HS}}\le2\eta_n
\le\varepsilon$ after the same enlargement of $C$. Hence
\[
 \|Q/\alpha\|\le1+\varepsilon,
 \qquad
 \|Q/\alpha\|_{\mathrm{HS}}\ge\sqrt v-\varepsilon.
\]
Because $v\ge1$,
\[
 \frac{\operatorname{tr}((Q/\alpha)^2)}{(1+\varepsilon)^2}
 \ge v\frac{(1-\varepsilon)^2}{(1+\varepsilon)^2}
 \ge(1-4\varepsilon)v=\kappa v.
\]
Apply logarithmic concavity to the eigenvalues of $Q^2$, whose
maximum is at most $\alpha^2(1+\varepsilon)^2$. Since
$\alpha\ge1/2$, the resulting logarithm is at least
$\log(1+t^2/4)\ge\log(1+c_*t^2)$.
This proves \eqref{moderate:prefix-determinant}.
\end{proof}

\begin{lemma}\label{lem:unused-block}
Under the conditions of Lemma~\ref{lem:prefix-damping}, let
$G_Q\sim N(0,Q)$ be independent of $W$ and put
$\Phi_t(x)=e^{it\|x\|^2/2}$. Define
\begin{equation}\label{moderate:resolvent-factor}
 b_0(t)=1,
 \qquad
 b_\alpha(t)=\min\left\{1,
              \sqrt{\frac{3}{2\alpha|t|}}\right\},
 \quad \alpha>0, t\ne0.
\end{equation}
If $n\ge Cr$ and $h(t)\ge C(r+1)$, then for $0\le k\le4$ and deterministic real
vectors $u_1,\ldots,u_k,z$,
\begin{equation}\label{moderate:derivative-damping}
 \begin{split}
 &\left|\mathbb E D^k\Phi_t(G_Q+W+z)[u_1,\ldots,u_k]\right|\\
 &\quad\le C_kH(t)\prod_{\ell=1}^k\|u_\ell\|
       \sum_{b=0}^{\lfloor k/2\rfloor}|t|^{k-b}
          \{b_\alpha(t)+\|z\|\}^{k-2b}.
 \end{split}
\end{equation}
For $k=0$, the product is empty and equals one. In particular, taking
$Q=0$, $W=\boldsymbol S_n$, and $\alpha=0$ gives
\begin{equation}\label{moderate:unweighted-damping}
 \sup_{z\in\mathbb R^p}
 \left|\mathbb E e^{it\|\boldsymbol S_n+z\|^2/2}\right|
 \le CH(t).
\end{equation}
\end{lemma}

\begin{proof}
By conjugation it suffices to consider $t>0$. Set $A=(I-itQ)^{-1}$.
We first show that
\begin{equation}\label{moderate:resolvent-bound}
 \|A\|\le1,
 \qquad \|A\Sigma A^*\|\le b_\alpha(t)^2.
\end{equation}
The first bound follows from the spectral decomposition of $Q$.
Since $\Sigma\preceq I$, it also gives
$A\Sigma A^*\preceq I$. For $\alpha>0$,
\eqref{moderate:partial-covariances} implies
$\alpha\Sigma\preceq Q+\eta_n I$. Furthermore,
\[
 \|AQA^*\|
 =\max_{\lambda\in\operatorname{spec}(Q)}
      \frac{\lambda}{1+t^2\lambda^2}
 \le\frac1{2t}.
\]
The condition on $h(t)$ gives
$t^2\le n/[C(r+1)]$. By $v\le r$ and
\eqref{moderate:prefix-error}, it follows, after increasing $C$, that
$t\eta_n\le6\sqrt{C_1/C}\le1$. Therefore
\[
 A\Sigma A^*
 \preceq\frac1\alpha(AQA^*+\eta_n AA^*)
 \preceq\frac{3}{2\alpha t}I.
\]
This proves \eqref{moderate:resolvent-bound}, including the case
$\alpha=0$ by the first bound alone.

Gaussian integration, valid also for singular $Q$, gives
\begin{equation}\label{moderate:gaussian-prefix-transform}
 \mathbb E_{G_Q}\Phi_t(G_Q+x)
 =\det(I-itQ)^{-1/2}
       \exp\{itx^\top A x/2\},\qquad x\in\mathbb R^p.
\end{equation}
The square root is chosen continuously from $t=0$.
The exponential on the right has modulus at most one.
Differentiating this identity shows that
\begin{equation}\label{moderate:gaussian-polynomial-identity}
 \begin{split}
 &\mathbb E D^k\Phi_t(G_Q+W+z)[u_1,\ldots,u_k]\\
 &\qquad=
 \mathbb E\{\Phi_t(G_Q+W+z)P_{k,Q}(W+z)\},
 \end{split}
\end{equation}
where $P_{k,Q}(x)$ is a sum over partitions of $\{1,\ldots,k\}$
into singletons and pairs. A singleton $\ell$ contributes
$it u_\ell^\top A x$; a pair $\{\ell,m\}$ contributes
$it u_\ell^\top A u_m$. In particular, $P_{k,Q}$ depends on $W+z$
but contains no Gaussian random variables. All differentiations are
justified by finite Gaussian moments and the requisite moments of
$W$. The identity follows by real Gaussian integration.

For each subsum $V=\sum_{i\in J}Y_i$ with $J\subseteq I$, we have
$\Cov(V)\preceq S\preceq\Sigma$; hence
Lemma~\ref{lem:directional-moments} implies
\begin{equation}\label{moderate:resolvent-projection-moment}
 \|u^\top A V\|_{L^q}
 \le C_q\|u\|b_\alpha(t),\qquad 1\le q\le4.
\end{equation}
For complex coefficients, split the projection into real and
imaginary parts and apply the fourth-moment bound
\eqref{new:fourth-projection}. Lower orders follow by Lyapunov's
inequality. Equation~\eqref{moderate:resolvent-bound} then proves
\eqref{moderate:resolvent-projection-moment}.

Suppose first that $\alpha\le1/2$. Partition $W$ into the seven
blocks in Lemma~\ref{lem:prefix-damping}. Expand every singleton
factor in $P_{k,Q}(W+z)$ over these blocks and $z$.
Each resulting monomial involves at most $k\le4$ blocks, so at
least one block is absent. Condition on $G_Q$ and on all the other
blocks. The monomial is measurable under this conditioning, and the
remaining exponential in \eqref{moderate:gaussian-polynomial-identity}
has a real shift. The unused block therefore supplies the factor
$H(t)$. Apply H\"older's inequality and
\eqref{moderate:resolvent-projection-moment} to the remaining
singleton projections. A pair contributes at most
$t\|u_\ell\|\|u_m\|$, and a singleton at $z$ contributes at most
$t\|u_\ell\|\|z\|$. A partition with $b$ pairs consequently has
the bound displayed in \eqref{moderate:derivative-damping}.

If $\alpha\ge1/2$, integrate $G_Q$ in
\eqref{moderate:gaussian-polynomial-identity} first. The modulus of
its conditional expectation is at most
$\det(I+t^2Q^2)^{-1/4}$, uniformly in the real shift $W+z$.
Use \eqref{moderate:prefix-determinant} and then
\eqref{moderate:resolvent-projection-moment} directly on the
polynomial. This gives the same bound. The number of partitions with
$b$ pairs is $k!/[2^b b!(k-2b)!]$; these coefficients depend only on $k$ and are absorbed
in $C_k$.
\end{proof}

\begin{lemma}\label{lem:position-sums}
Let $b_\alpha(t)$ be defined by
\eqref{moderate:resolvent-factor}, and put
\[
 F_4(\alpha,t)=t^2+t^3b_\alpha(t)^2+t^4b_\alpha(t)^4.
\]
For $0<t\le n$ and $\alpha_i=(i-1)/n$,
\begin{equation}\label{moderate:position-sums}
 \frac1n\sum_{i=1}^n F_4(\alpha_i,t)\le C(t^2+t^3),
\end{equation}
where $C$ is numerical.
\end{lemma}
\begin{proof}
Put $a=3n/(2t)$. Since $b_0(t)=1$ and
$b_{k/n}(t)^4=\min\{1,(a/k)^2\}$ for $k\ge1$, integral comparison gives
\[
 \sum_{k=0}^{n-1} b_{k/n}(t)^4
 \le 1+\int_0^\infty\min\{1,(a/x)^2\}\,dx
 =1+2a\le 4n/t.
\]
Together with $b_\alpha(t)^2\le1$, this proves \eqref{moderate:position-sums}.
\end{proof}

The following bounded-scalar construction adapts the
third-moment-matching proxy in G\"otze and
Ulyanov~\cite[Lemma~2.8]{GotzeUlyanov2000}, attributed there to
Sazonov~\cite[p.~85]{Sazonov1981}.
Our multiplier need not be centered because each $\boldsymbol X_i$
is centered; the proof verifies the moment identities and
fourth-moment bounds separately for each increment.

\begin{lemma}\label{lem:third-moment-proxy}
Let $\boldsymbol X_i'$ be an independent copy of $\boldsymbol X_i$,
let $\eta_i$ have the Bernoulli distribution with success probability
$1/8$, and let $G_i\sim N(0,\Sigma_i/2)$. Take these objects to be
mutually independent and independent of the original array, and put
\begin{equation}\label{moderate:third-moment-proxy}
 A_i=2\eta_i\boldsymbol X_i',\qquad Z_i=A_i+G_i.
\end{equation}
Then $\mathbb EA_i=\mathbb EZ_i=0$ and
\begin{equation}\label{moderate:proxy-matching}
 \operatorname{Cov}(A_i)=\Sigma_i/2,\qquad
 \operatorname{Cov}(Z_i)=\Sigma_i,\qquad
 \mathbb EZ_i^{\otimes3}=\mathbb E\boldsymbol X_i^{\otimes3}.
\end{equation}
Let $V_i$ be any of $\boldsymbol X_i,A_i,Z_i,G_i$ or a centered Gaussian
vector with the covariance of one of these increments. Then
\[
 \mathbb E|u^\top V_i|^4
 \le 8K_4\bigl(u^\top\Cov(V_i)u\bigr)^2,
 \qquad u\in\mathbb R^p,\quad i=1,\ldots,n.
\]
For $1\le q\le4$,
\begin{equation}\label{moderate:increment-moments-explicit}
 \mathbb E\|V_i\|^q\le C_q(r/n)^{q/2}.
\end{equation}
Also $\|A_i\|\le2L_0\sqrt r$ for the bounded array considered in
this appendix.
\end{lemma}
\begin{proof}
For $\zeta_i=2\eta_i$ we have $\mathbb E\zeta_i^2=1/2$ and
$\mathbb E\zeta_i^3=1$. Independence and centering give the first
and second moment identities. Every mixed cubic term in
$\mathbb E(A_i+G_i)^{\otimes3}$ vanishes: a term containing one
Gaussian factor has its zero mean, and a term containing two has
$\mathbb EA_i=0$. The pure Gaussian cubic term vanishes as well.
This proves \eqref{moderate:proxy-matching}.

For a real vector $u$ put $q_i=u^\top\Sigma_i u$. The fourth
moment assumption and independence give
\[
 \mathbb E|u^\top A_i|^4\le2K_4q_i^2
       =8K_4(q_i/2)^2,\qquad
 \mathbb E|u^\top Z_i|^4\le(2K_4+9/4)q_i^2.
\]
In the second inequality the mixed fourth term is $6(q_i/2)^2$
and the Gaussian fourth term is $3(q_i/2)^2$.
The statements include zero-variance directions, which vanish
almost surely. Gaussian increments have standardized fourth
moment three. Thus the standardized bounds are uniform.

If $T_i$ is the covariance of any of these increments, Minkowski's
inequality applied to squared coordinates gives
\[
 (\mathbb E\|V_i\|^4)^{1/2}
 \le C\operatorname{tr}T_i\le C C_1r/n.
\]
Here $T_i\preceq\Sigma_i$, and only the trace part of
\textup{(\ref{ass:A1})} is used. Lyapunov's inequality proves
\eqref{moderate:increment-moments-explicit} for lower orders.
The bound on $A_i$ follows directly from its definition.
\end{proof}

\begin{lemma}\label{lem:proxy-prefix}
Use the ordering in Lemma~\ref{lem:prefix-order}, and write
$Q_k=\sum_{j\le k}\Sigma_j$. Define
\[
 \gamma_i=\frac{i-1}{2n},\qquad
 \widehat Q_i=\frac12Q_{i-1},\qquad
 W_i^A=\sum_{j<i}A_j+\sum_{j>i}\boldsymbol X_j.
\]
Since $Z_j=A_j+G_j$, the full sum
$\sum_{j<i}Z_j+\sum_{j>i}\boldsymbol X_j$ equals
$W_i^A+\sum_{j<i}G_j$.
The Gaussian vector $\sum_{j<i}G_j$ is independent of $W_i^A$.
Its covariance $\widehat Q_i$, together with
$S_i^A=\operatorname{Cov}(W_i^A)$, satisfies
\eqref{moderate:partial-covariances} with $\alpha=\gamma_i\le1/2$.
Lemmas~\ref{lem:prefix-damping} and~\ref{lem:unused-block}
apply to this pair with $H$ as defined in
\eqref{moderate:H-balanced} and constants independent of $i$.
The derivative estimates hold for $0\le k\le4$.
\end{lemma}
\begin{proof}
The covariance identity is
\[
 S_i^A=\Sigma-\tfrac12Q_{i-1}-\Sigma_i.
\]
Both covariances are bounded above by $\Sigma$. The prefix ordering
and $\|\Sigma_i\|_{\mathrm{HS}}\le C_1\sqrt r/n$ give
\[
 \|\widehat Q_i-\gamma_i\Sigma\|_{\mathrm{HS}}
 \le2\sqrt{C_1r/n},\qquad
 \|S_i^A-(1-\gamma_i)\Sigma\|_{\mathrm{HS}}
 \le2\sqrt{C_1r/n}+C_1\sqrt r/n\le\eta_n.
\]
By Lemma~\ref{lem:third-moment-proxy},
$W_i^A\in\mathcal M(1,2L_0,8K_4)$: its summand at $j<i$ has covariance
$\Sigma_j/2$, its summand at $j>i$ has covariance $\Sigma_j$, and
index $i$ is omitted. Thus the hypotheses of Lemma~\ref{lem:prefix-damping}
hold, and Lemma~\ref{lem:unused-block} applies to the independent
Gaussian vector $\sum_{j<i}G_j$.
\end{proof}

The next derivative estimate follows from the exact Gaussian
convolution formula. The proof bounds the moments of projections of the sum $W$
with constants independent of the individual covariance eigenvalues.

\begin{lemma}\label{lem:fixed-gaussian-reserve}
Let $I,J\subseteq\{1,\ldots,n\}$ be disjoint and put
$W=\sum_{i\in I}A_i+\sum_{j\in J}B_j$, with $A_i$ from
Lemma~\ref{lem:third-moment-proxy} and $B_j\sim N(0,\Sigma_j/2)$.
Assume that the summands and $G_0\sim N(0,\Sigma/2)$ are mutually
independent.
Set
\[
 D_{1/2}(t)=\det(I+t^2\Sigma^2/4)^{-1/4},\qquad
 \Phi_t(x)=e^{it\|x\|^2/2},\qquad F(z)=\mathbb E\Phi_t(G_0+W+z).
\]
If $n\ge Cr$ and $h(t)\ge C(r+1)$, then, for $0\le k\le6$
and deterministic real vectors $u_1,\ldots,u_k,z$,
\begin{equation}\label{moderate:fixed-reserve-derivatives}
 |D^kF(z)[u_1,\ldots,u_k]|
 \le C_kD_{1/2}(t)|t|^{k/2}
       (1+\sqrt{|t|}\|z\|)^k\prod_{a=1}^k\|u_a\|.
\end{equation}
For $k=0$ one can take $C_0=1$ and the empty product equals one.
The constants are independent of $p,n$ and individual eigenvalues.
\end{lemma}
\begin{proof}
By conjugation assume $t>0$. Set $R=(I-it\Sigma/2)^{-1}$ and
$b^2=\min(1,t^{-1})$. Spectral decomposition gives
\[
 \|R\|\le1,\qquad
 \|R\Sigma R^*\|
 =\max_{s\in\operatorname{spec}(\Sigma)}
       \frac{s}{1+t^2s^2/4}\le b^2.
\]
The size and frequency restrictions give $r/n\le1$ and
$t^2r/n\le1$, hence $b^4\ge r/n$. The sixth projection estimate
in Lemma~\ref{lem:directional-moments} therefore yields
\[
 \|u^\top RW\|_{L^q}\le C_q\|u\|b,\qquad 1\le q\le6.
\]
Since $b^4\ge r/n$, the second term in
$C\|u\|^6\{b^6+(r/n)b^2\}$ is absorbed by the first.
Complex projections are estimated through their real and imaginary parts.

Gaussian integration gives
\[
 \mathbb E_{G_0}\Phi_t(G_0+x)
 =\det(I-it\Sigma/2)^{-1/2}\exp\{itx^\top Rx/2\}.
\]
For real $x$ the exponential has modulus at most one. Its derivatives
are sums over partitions into singletons and pairs. A singleton
contributes $it\,u_a^\top R(W+z)$, and a pair contributes
$it\,u_a^\top Ru_b$. H\"older's inequality bounds products of at
most six random singleton factors by the displayed projection
moments. Since $tb\le\sqrt t$, each singleton is bounded by
$C\sqrt t\|u_a\|(1+\sqrt t\|z\|)$, while each pair is bounded
by $t\|u_a\|\|u_b\|$. The determinant supplies $D_{1/2}(t)$.
Summing the finitely many partitions proves
\eqref{moderate:fixed-reserve-derivatives}.
\end{proof}

\begin{lemma}\label{lem:complete-remainders}
For a smooth function $F$ define
\[
 \mathcal R_4(F,x)=F(x)-\sum_{k=0}^3\frac1{k!}
                         D^kF(0)[x^{\otimes k}].
\]
Under the hypotheses of Lemma~\ref{lem:unused-block}, let
$F(z)=\mathbb E\Phi_t(G_Q+W+z)$ and put
$\Lambda=\sqrt{|t|}+|t|b_\alpha(t)$. Then
\begin{equation}\label{moderate:complete-prefix-remainder}
 |\mathcal R_4(F,x)|\le CH(t)\Lambda^4\|x\|^4.
\end{equation}
If instead $F(z)=\mathbb E\Phi_t(G_0+W+z)$ under the hypotheses of
Lemma~\ref{lem:fixed-gaussian-reserve}, then
\begin{equation}\label{moderate:complete-fourth-remainder}
 |\mathcal R_4(F,x)|\le CD_{1/2}(t)t^2\|x\|^4.
\end{equation}
For this latter function $F$ and a deterministic real vector $a$, write
$h_a(x)=D^3F(x)[a,a,a]$ and
\[
 \mathcal R_3(h_a,x)=h_a(x)-h_a(0)-Dh_a(0)[x]
                         -\tfrac12D^2h_a(0)[x,x].
\]
Then
\begin{equation}\label{moderate:complete-cubic-remainder}
 |\mathcal R_3(h_a,x)|
 \le CD_{1/2}(t)|t|^3\|a\|^3\|x\|^3.
\end{equation}
\end{lemma}
\begin{proof}
Assume $t>0$. The prefix derivative estimate implies, for
$0\le k\le4$,
\[
 |D^kF(z)[u_1,\ldots,u_k]|
 \le CH(t)(\Lambda+t\|z\|)^k\prod_a\|u_a\|.
\]
Indeed each pair contributes $t\le\Lambda^2$. If
$\|x\|\le\Lambda^{-1}$, then $t\|sx\|\le\Lambda$ for
$0\le s\le1$, so Taylor's integral remainder proves
\eqref{moderate:complete-prefix-remainder}. If
$\|x\|>\Lambda^{-1}$, bound its defining difference:
\[
 |\mathcal R_4(F,x)|
 \le CH(t)\left(1+\sum_{k=0}^3(\Lambda\|x\|)^k\right)
 \le CH(t)\Lambda^4\|x\|^4.
\]
The estimate $|F(x)|\le CH(t)$ holds after taking the expectation;
the integrand $\Phi_t(G_Q+W+x)$ itself has modulus one.

In the case $F(z)=\mathbb E\Phi_t(G_0+W+z)$, put $q=\sqrt t\|x\|$.
If $q\le1$, Taylor's integral formula and
\eqref{moderate:fixed-reserve-derivatives}, along $sx$ with
$0\le s\le1$, give \eqref{moderate:complete-fourth-remainder}
using $k=4$, and \eqref{moderate:complete-cubic-remainder} using
$k=6$. If $q>1$, use the defining differences instead. For the
fourth remainder, $|F(x)|\le D_{1/2}(t)$, and the Taylor term of
degree $k\le3$ is at most $CD_{1/2}(t)q^k$. Their sum is bounded
by $CD_{1/2}(t)q^4$.
For the cubic remainder,
\[
 |h_a(x)|\le CD_{1/2}(t)t^{3/2}\|a\|^3(1+q)^3.
\]
Its degree-$j$ Taylor term, $j=0,1,2$, is bounded by
$CD_{1/2}(t)t^{3/2}\|a\|^3q^j$. Since $q>1$, the whole
difference is at most
$CD_{1/2}(t)t^{3/2}\|a\|^3q^3$, proving
\eqref{moderate:complete-cubic-remainder}.
Thus the sixth derivative is needed only for small shifts;
large shifts are controlled by the complete remainder itself.
\end{proof}

The next lemma bounds the contribution of a pair of distinct increments at small
frequencies. The operator norm in \textup{(\ref{ass:A1})} enters when one of the three
directions from the first increment is paired with the second increment.

\begin{lemma}\label{lem:sharp-small-pair}
Suppose that $\|\Sigma\|=1$, that \textup{(\ref{ass:A1})--(\ref{ass:A2})} hold and that
$\|\boldsymbol X_k\|\le L_0\sqrt r$ almost surely for every $k$.
Let $A_k$ be as in Lemma~\ref{lem:third-moment-proxy} and let the
$B_k\sim N(0,\Sigma_k/2)$ be mutually independent and independent of
the $A_k$. Suppose $n\ge r$ and write $P=r/n$.
Take distinct indices $i,j$, let $a=A_i$, and let $b$ have the
law of either $A_j$ or $B_j$.
Let $W$ be a sum of $A_k$'s and $B_k$'s over disjoint index sets
not containing $i,j$. Suppose that $a,b,W$ and
$G_0\sim N(0,\Sigma/2)$ are mutually independent. For $0<t\le1$, set
\[
 F(x)=\mathbb E\exp\{it\|G_0+W+x\|^2/2\},\qquad
 h_a(x)=D^3F(x)[a,a,a].
\]
With the complete remainder $\mathcal R_3$ from
Lemma~\ref{lem:complete-remainders}, one has
\begin{equation}\label{sharp:small-pair}
 \mathbb E\bigl|\mathcal R_3(h_a,b)\bigr|
 \le C D_{1/2}(t)P^3\bigl(t^4+r^{-1/2}t^3\bigr).
\end{equation}
The constant depends only on the covariance balance, fourth moment
and truncation constants.
\end{lemma}

\begin{proof}
Write $R=(I-it\Sigma/2)^{-1}$ and $D=D_{1/2}(t)$. Conditional on
$W$, integrating the Gaussian gives the determinant factor with
modulus $D$, times $\exp\{it(W+x)^\top R(W+x)/2\}$, of modulus
at most one for real $x$. In the derivative of order $k$, each
partition of the directions into $q$ pairs and $m=k-2q$ singletons
has a coefficient of modulus $t^{k-q}$. A pair contributes the
bilinear factor $u^\top Ru'$, and a singleton at $x=sb$, $0\le s\le1$,
can be expanded into $u^\top RW$ and $s u^\top Rb$.
Let $l$ denote the number of singletons evaluated at $sb$ rather
than at $W$ in a term of this expansion.

We use two bounds throughout. First, $\|R\|\le1$ and
$\|R\Sigma R^*\|\le1$. By Lemma~\ref{lem:directional-moments},
with its projection parameter equal to one and $r/n\le1$,
$\|u^\top RW\|_{L^6}\le C\|u\|$ for real $u$. Conditional on
$a,b$, H\"older's inequality therefore bounds the product of all
$m-l\le6$ projections of $W$ by the product of the corresponding
direction norms times a constant. Second, by
\eqref{moderate:increment-moments-explicit},
$\mathbb E\|a\|^j+\mathbb E\|b\|^j\le C P^{j/2}$ for $1\le j\le4$.
If one factor is $a^\top Rb$, or $b^\top Ra=a^\top Rb$ since $R$
is symmetric, we estimate it conditionally on $b$. Conditional Cauchy--Schwarz gives
\begin{equation}\label{sharp:one-cross-factor}
 \mathbb E_a\{\|a\|^2|a^\top Rb|\}
 \le(\mathbb E\|a\|^4)^{1/2}
       \bigl(b^\top R^*\operatorname{Cov}(a)Rb\bigr)^{1/2}
 \le C P n^{-1/2}\|b\|
 = C P^{3/2}r^{-1/2}\|b\|.
\end{equation}
Thus any such factor saves $r^{-1/2}$ against the bound using three
norms of $a$. This remains true after multiplication by any
nonnegative function of $b$, including either shift indicator below.
We now estimate the expectation of the absolute remainder.

On $\sqrt t\|b\|\le1$, Taylor's integral formula uses
$D^6F(sb)[a,a,a,b,b,b]$, averaged over $0\le s\le1$ with a bounded
weight. For a partition with $m=6-2q$, expanding its singletons as
above and bounding the projections of $W$ gives the majorant
\[
 C D\, t^{6-q}\|a\|^3\|b\|^{3+l}.
\]
If $l=0$ and $m\ge2$, then $6-q\ge4$, so its expectation is at
most $C D P^3t^4$. If $l=m=0$, the three occurrences of $a$ cannot
all pair among themselves. There is a cross pair $a^\top Rb$, and
\eqref{sharp:one-cross-factor} gives $C D P^3r^{-1/2}t^3$.
For $l\ge1$, bound $\|b\|^{l-1}\le t^{-(l-1)/2}$ on this event.
The remaining degree in $b$ is four; the resulting power of $t$ is
\[
 6-q-\frac{l-1}{2}=\frac{7+m-l}{2}.
\]
When $m>l$ this power is at least four. When $m=l$, no singleton
uses $W$. Either a pair joins $a$ to $b$, or at least one of the
three $a$ directions is a singleton and gives $s a^\top Rb$.
In either case we keep one such factor when bounding the
monomial, and use \eqref{sharp:one-cross-factor}. More explicitly,
after setting it aside the power of $\|b\|$ outside that factor is $2+l$;
discarding $l-1$ powers on the small-shift event leaves three powers,
and \eqref{sharp:one-cross-factor} restores the fourth. These two
cases are bounded by $C D P^{7/2}t^4$ and
$C D P^{7/2}r^{-1/2}t^{7/2}$, respectively.

On $\sqrt t\|b\|>1$, use the defining difference
\[
 \mathcal R_3(h_a,b)
 =h_a(b)-\sum_{j'=0}^2\frac{1}{j'!}
      D^{3+j'}F(0)[a,a,a,b^{\otimes j'}].
\]
For $h_a(b)$ one has $m=3-2q$. A monomial has degree $l\le3$ in
$b$ and carries the factor $t^{3-q}$. Bound
$\|b\|^l\le t^{(4-l)/2}\|b\|^4$ on this event.
The resulting exponent is $(7+m-l)/2$, at least four if $m>l$.
If $m=l$, then $m$ is odd and hence positive: one shifted singleton
gives $a^\top Rb$. Keep that factor and use
$\|b\|^{l-1}\le t^{(4-l)/2}\|b\|^3$ on the other powers of
$\|b\|$ before applying \eqref{sharp:one-cross-factor}.
For the Taylor term indexed by $j'$, put $m=3+j'-2q$. Its
$b$-degree is $j'$, so the same large-shift bound gives the exponent
\[
 3+j'-q+\frac{4-j'}2=\frac{7+m}{2}.
\]
If $m\ge1$, this is at least four. Otherwise $j'=1$, all directions
are paired, and one pair joins $a$ to $b$. Estimating that pair as
before gives the factor $r^{-1/2}$. In all cases the averaged bound
on the large-shift event is at most
$C D P^{7/2}(t^4+r^{-1/2}t^{7/2})$.
Since $P\le1$ and $t\le1$, these bounds prove
\eqref{sharp:small-pair}. Only moments through order four of the replacement vectors $a,b$ were used.
\end{proof}

\begin{proposition}\label{prop:moderate}
Let $\boldsymbol X_1,\ldots,\boldsymbol X_n$ be independent centered
vectors satisfying \textup{(\ref{ass:A1})--(\ref{ass:A3})}, with $\|\Sigma\|=1$,
$v\ge13/2$, and $\|\boldsymbol X_i\|\le L_0\sqrt r$ almost surely
for every $i$.
Write $\mathfrak f(t)=\mathbb E e^{it\|\boldsymbol S_n\|^2/2}$ and
$\widetilde{\mathfrak f}(t)=
\mathbb E e^{it\|\widetilde{\boldsymbol S}\|^2/2}$.
Use the parameters $v_0=6+\delta$, $\varepsilon=\delta/(8v_0)$,
$\kappa=1-4\varepsilon$, and $c_*=1/(8\cdot14^2)$, and the
function $H$ defined in \eqref{moderate:H-balanced}.
If $n\ge Cr$, then for every $t\ne0$ satisfying
$h(t)\ge C(r+1)$,
\begin{equation}\label{moderate:full-bound}
 |\mathfrak f(t)-\widetilde{\mathfrak f}(t)|
 \le CH(t)\left\{\frac{r^2}{n}|t|^2+
                         \frac{r^3}{n}|t|^3\right\}.
\end{equation}
For every $C_H>0$ there is a constant $N_0$ such that, if
$n\ge N_0r$ and $\tau=C_Hn^{1/3}/\sqrt v$, then
\begin{equation}\label{moderate:integrated}
 \int_0^\tau\frac{|\mathfrak f(t)-\widetilde{\mathfrak f}(t)|}{t}\,dt
 \le C\frac rn.
\end{equation}
The constants depend only on $C_1,K_4,C_\Sigma$,
$L_0,\delta$, and, in the integrated assertion, $C_H$.
\end{proposition}
\begin{proof}
Use the order in Lemma~\ref{lem:prefix-order} and take $Z_i=A_i+G_i$ from Lemma~\ref{lem:third-moment-proxy}. Put
$\mathfrak f_Z(t)=\mathbb E\Phi_t(\sum_iZ_i)$.
By conjugation assume $t>0$. The size and frequency restrictions
imply, after increasing their constant,
\begin{equation}\label{moderate:shift-absorption}
 t^2r/n\le1,\qquad n\ge r,\qquad t\le\sqrt{n/r}\le n.
\end{equation}

\emph{Comparison with the moment-matching proxy.}
For each $i$ let
\[
 W_i=\sum_{j<i}Z_j+\sum_{j>i}\boldsymbol X_j,\qquad
 F_i(x)=\mathbb E\Phi_t(W_i+x).
\]
The sum $W_i$ is independent of $\boldsymbol X_i$ and $Z_i$. Telescoping
and exact Taylor expansion through degree three give
\begin{equation}\label{moderate:proxy-replacement}
 \mathfrak f(t)-\mathfrak f_Z(t)
 =\sum_i\mathbb E\{\mathcal R_4(F_i,\boldsymbol X_i)
                         -\mathcal R_4(F_i,Z_i)\},
\end{equation}
because all terms through degree three cancel by
\eqref{moderate:proxy-matching}.
Decompose $W_i=\sum_{j<i}G_j+W_i^A$. Lemma~\ref{lem:proxy-prefix}
and \eqref{moderate:complete-prefix-remainder}, with
$\gamma_i=(i-1)/(2n)$, yield
\[
 |\mathfrak f(t)-\mathfrak f_Z(t)|
 \le CH(t)(r/n)^2\sum_{i=1}^n
                  (\sqrt t+tb_{\gamma_i}(t))^4.
\]
Here we used only the fourth moments in
\eqref{moderate:increment-moments-explicit}.
Since $(a+b)^4\le8(a^4+b^4)$ and
$b_{\gamma_i}(t)=b_{(i-1)/n}(t/2)$, the position sum
in Lemma~\ref{lem:position-sums} bounds the last sum by
$Cn(t^2+t^3)$; its condition $t/2\le n$ follows from
\eqref{moderate:shift-absorption}. We conclude that
\begin{equation}\label{moderate:proxy-comparison}
 |\mathfrak f(t)-\mathfrak f_Z(t)|
 \le CH(t)\frac{r^2}{n}(t^2+t^3).
\end{equation}

\emph{Comparison with a common independent Gaussian component.}
Let $G_0=\sum_iG_i\sim N(0,\Sigma/2)$ and let
$B_i\sim N(0,\Sigma_i/2)$ be independent of all previous
variables. Then $G_0+\sum_iB_i\sim N(0,\Sigma)$, whereas
$G_0+\sum_iA_i$ defines $\mathfrak f_Z$.
We replace the $A_i$ successively by the $B_i$, leaving $G_0$ unchanged.

For the first ordered replacement set
\begin{gather*}
 V_i=G_0+\sum_{j<i}B_j+\sum_{j>i}A_j,\qquad
 F_i^0(x)=\mathbb E\Phi_t(V_i+x),\\
 h_i(y)=\mathbb E_{A_i}D^3\Phi_t(y)[A_i,A_i,A_i].
\end{gather*}
The means and covariances of $A_i,B_i$ agree, and the cubic
Gaussian tensor vanishes. Therefore
\begin{equation}\label{moderate:first-replacement-explicit}
 \begin{split}
 \mathfrak f_Z(t)-\widetilde{\mathfrak f}(t)
 &=\sum_i\mathbb E\{\mathcal R_4(F_i^0,A_i)
                             -\mathcal R_4(F_i^0,B_i)\}\\
 &\quad+\frac16\sum_i\mathbb Eh_i(V_i).
 \end{split}
\end{equation}
By \eqref{moderate:complete-fourth-remainder} the first sum is
bounded in absolute value by $CD_{1/2}(t)t^2r^2/n$.

Since $\Phi_t$ is even, $h_i$ is odd. Consequently
\[
 \mathbb Eh_i\left(G_0+\sum_{j\ne i}B_j\right)=0.
\]
Replace the remaining indices $j>i$ in $V_i$ successively by
their Gaussian counterparts. For $i<j$ define
\[
 U_{ij}=G_0+\sum_{\substack{k<j\\k\ne i}}B_k+\sum_{k>j}A_k,
 \quad F_{ij}(x)=\mathbb E\Phi_t(U_{ij}+x),
 \quad h_{ij,a}(x)=D^3F_{ij}(x)[a,a,a].
\]
The sum $U_{ij}$ is independent of $A_i,A_j,B_j$.
Expanding through degree two, whose terms cancel by equal means
and covariances, gives the exact identity
\begin{equation}\label{moderate:double-replacement-explicit}
 \mathbb Eh_i(V_i)
 =\sum_{j>i}\mathbb E\{
     \mathcal R_3(h_{ij,A_i},A_j)-\mathcal R_3(h_{ij,A_i},B_j)\}.
\end{equation}
Apply \eqref{moderate:complete-cubic-remainder} to deterministic directions
before averaging them. Independence and the third-moment consequence
of \eqref{moderate:increment-moments-explicit} give
\[
 \begin{split}
 &\left|\mathbb E\{
       \mathcal R_3(h_{ij,A_i},A_j)-\mathcal R_3(h_{ij,A_i},B_j)\}\right|\\
 &\quad\le CD_{1/2}(t)t^3\mathbb E\|A_i\|^3
       \{\mathbb E\|A_j\|^3+\mathbb E\|B_j\|^3\}
 \le CD_{1/2}(t)t^3(r/n)^3.
 \end{split}
\]
There are at most $n^2/2$ pairs. Thus
\begin{equation}\label{moderate:reserve-comparison}
 |\mathfrak f_Z(t)-\widetilde{\mathfrak f}(t)|
 \le CD_{1/2}(t)
       \left\{\frac{r^2}{n}t^2+\frac{r^3}{n}t^3\right\}.
\end{equation}
The spectral bound gives
\[
 D_{1/2}(t)\le(1+t^2/4)^{-v/4}
       \le(1+c_*t^2)^{-\kappa v/4}\le H(t).
\]
Combining \eqref{moderate:proxy-comparison} and
\eqref{moderate:reserve-comparison}, and using $r\ge1$, proves
\eqref{moderate:full-bound}.

For $0<t\le1$, Lemma~\ref{lem:sharp-small-pair} gives a sharper
estimate for each pair on the right of
\eqref{moderate:double-replacement-explicit}: its contribution in
absolute value is at most
$CD_{1/2}(t)(r/n)^3(t^4+r^{-1/2}t^3)$.
Indeed the base sum $U_{ij}$ contains the independent $N(0,\Sigma/2)$
component, and its remaining summands are $B_k$ or $A_k$ at
indices distinct from $i,j$. Thus
\begin{equation*}
 |\mathfrak f_Z(t)-\widetilde{\mathfrak f}(t)|
 \le CD_{1/2}(t)
 \left\{\frac{r^2}{n}t^2+
         \frac{r^3}{n}(t^4+r^{-1/2}t^3)\right\},
 \qquad 0<t\le1.
\end{equation*}

\emph{Integration over $(0,\tau]$.}

Fix $C_H>0$ and let $n\ge N_0r$, $\tau=C_Hn^{1/3}/\sqrt v$.
Since $r\le C_\Sigma v$ and $r\le n$,
\[
 h(\tau)\ge\tfrac12\min\{C_H^{-2}n^{1/3}v,\;
                                 C_H^{-1}n^{2/3}\sqrt v\}
         \ge c n^{1/6}r .
\]
Here and below the constants may also depend on $C_H$.
Because $r\ge1$, increasing $N_0$ ensures
$h(\tau)\ge M\{r+\log(2n)\}$ for any prescribed $M$.
All comparison bounds therefore apply for $0<t\le\tau$.
We integrate the two replacements separately.

First consider \eqref{moderate:proxy-comparison}.
By Lemma~\ref{lem:A.3}, the spectral part of $H$ contributes at most
$C r^2(v^{-1}+v^{-3/2})/n\le Cr/n$.
The exponential part contributes at most
$C(r^2/n)(\tau^2+\tau^3)e^{-ch(\tau)}\le Cr/n$:
choose $M$ so that $e^{-ch(\tau)}\le(2n)^{-10}$ and note that
\eqref{moderate:shift-absorption} gives $\tau\le\sqrt n$.
The polynomial part of $H$ contributes at most
\[
 \frac{Cr^{7/2}}{n^3}(\tau^5+\tau^6)
 \le C\left\{\frac{r^{7/2}}{n^{4/3}v^{5/2}}
             +\frac{r^{7/2}}{nv^3}\right\}
 \le C\left\{\frac r{n^{4/3}}+\frac{\sqrt r}{n}\right\}
 \le C\frac rn .
\]
In particular, this polynomial error occurs only in the first
comparison.

For the second comparison, $D_{1/2}(t)\le(1+t^2/4)^{-v/4}$ by
Lemma~\ref{lem:A.3}. The fourth-order remainders contribute at most
$C(r^2/n)\int_0^\infty tD_{1/2}(t)\,dt\le Cr^2/(nv)\le Cr/n$.
On $0<t\le\min(1,\tau)$, the sharper pair estimate applies.
Since $\log(1+t^2/4)\ge c t^2$ on $[0,1]$, its total contribution
is at most
\[
 \frac{Cr^3}{n}\int_0^1(t^3+r^{-1/2}t^2)e^{-cvt^2}\,dt
 \le \frac{Cr^3}{n}(v^{-2}+r^{-1/2}v^{-3/2})
 \le C\frac rn .
\]
If $\tau>1$, use \eqref{moderate:reserve-comparison} for the pairs on
$1<t\le\tau$. Here the strict margin $v\ge v_0=13/2>6$ gives
\[
 \int_1^\infty t^2D_{1/2}(t)\,dt
 \le (5/4)^{-(v-v_0)/4}
       \int_1^\infty t^2(1+t^2/4)^{-v_0/4}\,dt
 \le C v^{-2}.
\]
The last integral is finite because $v_0>6$, and the prefactor
decays exponentially in $v$. The pairs on this range thus contribute
at most $Cr^3/(nv^2)\le Cr/n$. Combining both comparisons proves
\eqref{moderate:integrated}.

\end{proof}

\section{Random cross frames}\label{sec:appendixB}
The frame-selection argument below uses conditional
Paley--Zygmund estimates and Gram--Schmidt residuals. We give a direct
proof and then the modification for sampling the actual nonidentical
array without replacement, which is needed in
Proposition~\ref{prop:high-frequency}.

\begin{lemma}\label{lem:cross-frames}
Let $\Sigma\ge0$, $\|\Sigma\|=1$, and
$v=\operatorname{tr}(\Sigma^2)\ge v_0>s-1$, where $s$ is a positive integer
and $v_0$ is a lower bound independent of $\Sigma$. Put $\eta=1-(s-1)/v_0>0$.
Suppose there exist constants $0<a\le b<\infty$ and $K<\infty$ such that,
conditional on the previously exposed vectors, each new real vector $Y$ satisfies
\[
 \mathbb E[Y\mid\mathcal F]=0,\qquad
 a\Sigma\le\mathbb E[YY^\top\mid\mathcal F]\le b\Sigma,
 \qquad
 \mathbb E[|u^\top Y|^4\mid\mathcal F]
 \le K(u^\top\Sigma u)^2
\]
for every real vector $u$.
Expose $L$ vectors on the left and then $L$ more on the right.
For every $0<\theta<1$ one can choose a finite $L=L(s,a,b,K,\theta)$
such that, with probability at least $\theta$, there exist $s$ distinct
left vectors $Z_1,\ldots,Z_s$ and $s$ distinct right vectors
$Z'_1,\ldots,Z'_s$ for which all singular values of
\[
 \bigl(\langle Z_j,Z'_k\rangle\bigr)_{j,k=1}^s
\]
belong to $[c\sqrt v,C\sqrt v]$.
The positive constants $c,C$ depend only on $s,v_0,a,b,K$, and can be chosen
independently of $\theta$.
In particular the conclusion holds for independent draws from any
centered law with covariance comparable to $\Sigma$ and the stated
uniform directional fourth moments.
\end{lemma}
\begin{proof}
For a real matrix $A$, conditional Minkowski's inequality gives
\[
 \mathbb E[\|AY\|^4\mid\mathcal F]
 \le K\{\operatorname{tr}(A\Sigma A^\top)\}^2.
\]
Put $M=8K/a^2$ and $\alpha=a^2/(8K)$. Since $K\ge a^2$,
$M\ge8$ and $0<\alpha\le1/8$.

Suppose that $k<s$ left vectors have been selected, and let $P$ be the
orthogonal projection onto the span of their images under
$\Sigma^{1/2}$. For a fresh $Y$, write $\boldsymbol{X}=\Sigma^{1/2}Y$ and
$R=\|(I-P)\boldsymbol{X}\|^2$. The conditional mean of $R$ is at least
\[
 a\operatorname{tr}\{(I-P)\Sigma^2\}
 \ge a(v-k)\ge a\eta v.
\]
Indeed, $k\le s-1$ and $v\ge v_0$ imply
$v-k\ge v-s+1\ge\{1-(s-1)/v_0\}v=\eta v$.
The conditional fourth-moment inequality above implies
\[
 \mathbb E[R^2\mid\mathcal F]
 \le\frac K{a^2}\{\mathbb E[R\mid\mathcal F]\}^2.
\]
Paley--Zygmund therefore gives
\[
 \Pr\{R\ge a\eta v/2\mid\mathcal F\}\ge a^2/(4K)=2\alpha.
\]
On the other hand, conditional Markov's inequality gives
\[
 \Pr\{\|\boldsymbol{X}\|^2>Mbv\mid\mathcal F\}\le1/M=\alpha.
\]
Thus a vector satisfying both
$R\ge a\eta v/2$ and $\|\boldsymbol{X}\|^2\le Mbv$ occurs with conditional probability
at least $\alpha$, uniformly over the previously selected frame.

Divide the left draws into $s$ consecutive blocks of $m$ vectors each.
In each block select the first vector satisfying these two conditions.
The conditional probability of failure in any block is at most
$(1-\alpha)^m$. Hence, except on an event of probability at most
$s(1-\alpha)^m$, all $s$ selections succeed.
Let $Z=(Z_1,\ldots,Z_s)$ be the selected column matrix. The successive
Gram--Schmidt residuals give
\[
 \det(Z^\top\Sigma Z)\ge(a\eta v/2)^s,
 \qquad
 \lambda_{\max}(Z^\top\Sigma Z)
 \le\operatorname{tr}(Z^\top\Sigma Z)\le sMbv.
\]
Consequently,
\[
 c_LvI_s\le Z^\top\Sigma Z\le C_LvI_s,
 \qquad
 c_L=\frac{(a\eta/2)^s}{(sMb)^{s-1}},\quad C_L=sMb.
\]

Condition now on the full left history. For a new right vector define
$U=Z^\top Y/\sqrt v$ and $A_Z=Z^\top\Sigma Z/v$.
Its conditional covariance lies between $aA_Z$ and $bA_Z$, and
\[
 \mathbb E[|w^\top U|^4\mid\mathcal F]
 \le K(w^\top A_Zw)^2.
\]
After $k<s$ right selections, let $P$ project onto their span in
$\mathbb R^s$. The same Paley--Zygmund and Markov argument, now using
$A_Z\ge c_LI_s$, shows that a fresh right vector simultaneously
satisfies
\[
 \|(I-P)U\|^2\ge ac_L/2,
 \qquad \|U\|^2\le Mb\,sC_L
\]
with conditional probability at least $\alpha$.
Use $s$ further blocks of length $m$. Conditional on successful left
selection, all right selections succeed with probability at least
$1-s(1-\alpha)^m$.
For the square matrix $R=(U_1,\ldots,U_s)$ of selected projected
right columns,
\[
 |\det R|^2\ge(ac_L/2)^s,
 \qquad \|R\|^2\le\operatorname{tr}(RR^\top)
                    \le Mb\,s^2C_L.
\]
It follows that its singular values belong to $[c,C]$, where one may
take
\[
 C=(Mb\,s^2C_L)^{1/2},\qquad
 c=\frac{(ac_L/2)^{s/2}}{(Mb\,s^2C_L)^{(s-1)/2}}.
\]
The desired cross Gram matrix is $\sqrt v R$.
Choose
\[
 m\ge\frac{\log\{2s/(1-\theta)\}}{-\log(1-\alpha)},\qquad L=sm.
\]
Both selections then succeed with probability at least
$1-2s(1-\alpha)^m\ge\theta$.
\end{proof}

\paragraph*{Sampling the actual nonidentical array without replacement.}
Under \textup{(\ref{ass:A1})}, we modify the selection argument by also controlling the
Euclidean norms of the selected left vectors. Assume the comparison
$r\le C_\Sigma v$ from \textup{\ref{ass:A3-i}}, write
$r=\tr\Sigma$, and take $s=6$, $v\ge11/2$.
Let $K_4$ satisfy $\mathbb E|u^\top\boldsymbol X_i|^4
\le K_4(u^\top\Sigma_i u)^2$, as required by \textup{(\ref{ass:A2})}, and let
$\boldsymbol X_i'$ be independent copies. Expose $Y_i=\sqrt n(\boldsymbol X_i-\boldsymbol X_i')/2$ in a uniformly
random order without replacement. The filtration reveals only exposed
indices and vectors, so the conditional laws of the unexposed vectors are unchanged.
After $k\le2L$ exposures, for $n\ge4L$ the conditional covariance is
\[
 Q_k=\frac n{2(n-k)}\sum_{i\text{ unexposed}}\Sigma_i,\qquad
 \frac12\Sigma-\frac{C_1k}{n}I\preceq Q_k\preceq\Sigma.
\]
The directional fourth moments satisfy
\[
 \begin{split}
 \mathbb E[|u^\top Y_{\rm next}|^4\mid\mathcal F]
 &\le\frac{n^2(2K_4+6)}{16(n-k)}
       \sum_{i\text{ unexposed}}(u^\top\Sigma_i u)^2\\
 &\le C\|u\|^2\,u^\top\Sigma u.
 \end{split}
\]
Here and below, constants may depend on $C_1,K_4,C_\Sigma$.
Conditional Cauchy--Schwarz, followed by Cauchy--Schwarz over the
rows of a real matrix $A$, consequently gives
\[
 \mathbb E[\|AY_{\rm next}\|^4\mid\mathcal F]
 \le C\tr(AA^\top)\tr(A\Sigma A^\top).
\]

Suppose $j<6$ left vectors have been selected, let $P$ project onto
the span of their images under $\Sigma^{1/2}$, and set
$R=\|(I-P)\Sigma^{1/2}Y_{\rm next}\|^2$. Then
\[
 \mathbb E[R\mid\mathcal F]
 \ge\frac12(v-j)-\frac{2C_1Lr}{n}\ge\frac v{44},
 \qquad
 \mathbb E[R^2\mid\mathcal F]\le Crv\le CC_\Sigma v^2,
\]
provided $n\ge88C_1LC_\Sigma$, since $v-j\ge v-5\ge v/11$.
Paley--Zygmund gives a uniform positive lower bound on the probability
of $R\ge v/88$. Also $\mathbb E[\|Y_{\rm next}\|^2\mid\mathcal F]\le r$.
For a sufficiently large $M_L$, Markov's inequality therefore
shows that
\[
 R\ge v/88,\qquad \|Y_{\rm next}\|^2\le M_Lr
\]
hold simultaneously with conditional probability at least
$\alpha_L>0$.
Use six consecutive blocks of $m$ left draws and select the first
acceptable vector in each block. Except on an event of probability
at most $6(1-\alpha_L)^m$, the resulting matrix $Z$ satisfies
\[
 \det(Z^\top\Sigma Z)\ge(v/88)^6,\qquad
 \tr(Z^\top\Sigma Z)\le6M_Lr,
\]
and hence, for positive constants $c_L,C_L,C_Z$,
\[
 c_LvI_6\preceq Z^\top\Sigma Z\preceq C_LvI_6,\qquad
 Z^\top Z\preceq C_ZrI_6.
\]
The last bound is the additional control needed under \textup{(\ref{ass:A1})}.

Condition on a successful left history and put
$U=Z^\top Y_{\rm next}/\sqrt v$ and $A_Z=Z^\top\Sigma Z/v$.
For $n\ge8C_1LC_ZC_\Sigma/c_L$,
\[
 \frac{c_L}{4}I_6
 \preceq\frac12A_Z-\frac{2C_1L}{nv}Z^\top Z
 \preceq\Cov(U\mid\mathcal F)\preceq C_LI_6.
\]
Moreover,
\[
 \mathbb E[|w^\top U|^4\mid\mathcal F]
 \le\frac C{v^2}\|Zw\|^2\,w^\top Z^\top\Sigma Zw
 \le C'\|w\|^4.
\]
For every projection $P'$ of rank at most five, these bounds give
a uniform positive lower bound on $\mathbb E\|(I-P')U\|^2$,
a uniform upper bound on $\mathbb E\|(I-P')U\|^4$, and a uniform
upper bound on $\mathbb E\|U\|^2$, conditionally on the history.
Paley--Zygmund and Markov thus give constants
$a_R,M_R,\alpha_R>0$ such that
\[
 \|(I-P')U\|^2\ge a_R,\qquad \|U\|^2\le M_R
\]
hold with conditional probability at least $\alpha_R$.
The same six-block selection on the right gives a matrix of projected
columns with determinant squared at least $a_R^6$ and operator norm
squared at most $6M_R$. There exist $0<c\le C<\infty$, independent of the selected
vectors, such that its singular values lie in $[c,C]$. Hence the singular
values of the cross Gram matrix lie in
$[c\sqrt v,C\sqrt v]$.

For any $0<\theta<1$, choose $m$ so that
$6(1-\alpha_L)^m+6(1-\alpha_R)^m\le1-\theta$, and put $L=6m$.
Then choose $n_F$ to cover all displayed sample-size
conditions. For $n\ge n_F$, two disjoint collections of $L$ actual
summands therefore yield the required six-column cross matrix
with probability at least $\theta$.
The constants $c,C$ can be chosen independently of $\theta$.

\paragraph*{The size condition for $s=6$.}
The inequality $v-5\ge v/11$, together with $r\le C_\Sigma v$,
ensures the residual bound above uniformly for $v\ge11/2$.
No lower bound on the smallest eigenvalue and no additional condition
on $p$ is used. The constants now may depend on $C_\Sigma$, but remain
uniform for $v\ge11/2$.

\section{The arithmetic integral}\label{sec:appendixC}
This appendix bounds the integral of $|\mathfrak f(t)|/t$ over
the frequencies $C_Hn^{1/3}/\sqrt v\le t\le U$ that remain after
Proposition~\ref{prop:moderate}. The argument has two parts.
First, Lemma~\ref{lem:arithmetic} gives an integral bound for a bilinear
form in integer-valued sign sums, uniformly over matrices of the given size satisfying the stated singular-value bounds.
Second, Proposition~\ref{prop:high-frequency} reduces the original,
possibly nonidentically distributed array to these bilinear forms.
The random cross frames in Appendix~\ref{sec:appendixB} give normalized matrices of order six
whose singular values are bounded above and below independently of $p$.

The source results used below are the real symmetrization inequality
stated by G\"otze and Zaitsev~\cite[Lemma~6.1]{GoetzeZaitsev2014}
(from Bentkus and G\"otze~\cite[Lemma~5.1]{BentkusGoetze1997})
and the lattice rotation estimate of Buterus et al.~\cite[Theorem~5.12]{ButerusGoetzeHilleMargulis2022}. We state their relevant
forms at the point of use. The other reductions, including all changes of scale,
are given below. No bound on the characteristic function at every
large frequency is asserted: the required estimate is an integral
bound and permits peaks caused by arithmetic structure.

For reference, the changes of scale in the array application are
\begin{gather*}
 M=\lfloor n/(3L)\rfloor,\qquad
 q=\lfloor3M/8\rfloor,\qquad R=\sqrt q,\\
 u=\frac{4t\sqrt v}{n},\qquad
 u_0=q^{-2/3},\qquad
 t_0'=\frac n{4\sqrt v}q^{-2/3}.
\end{gather*}
Here $L$ is the frame size chosen independently of $n$ and $p$, $s=6$, $\beta=1/2$, and
$dt/t=du/u$. Selecting rows with sufficiently many good neighbors
before applying H\"older preserves the exponent $1/2$.
This selection is related to \cite[Lemma~6.5]{BentkusGoetze1999};
we keep the H\"older products until after the final square root.
The two frequency ranges in the arithmetic lemma contribute as follows:
\[
\begin{array}{c|c|c}
 \text{range of }u & \text{estimate used} &
            \displaystyle\int K_q(u;A)^{1/2}\,du/u\\ \hline
 \lbrack u_0,q^{-1/2}\rbrack & K_q(u;A)\le C(1+q^2u^2)^{-s/2} & C/q\\
 \lbrack q^{-1/2},1\rbrack & \text{lattice rotation estimate} & C/q
\end{array}
\]
The lower-frequency comparison for the original sum remains in
Appendix~\ref{sec:appendixA}; this appendix starts where that comparison
can be joined to the arithmetic bound.

\subsection{The arithmetic estimate for a deterministic matrix}
The next lemma expresses the lattice-rotation input in the
sign-sum normalization used here. We prove this formulation from
Buterus et al.~\cite[Theorem~5.12]{ButerusGoetzeHilleMargulis2022},
allowing $A$ to be nonsymmetric and specifying the scaling;
the rotation theorem itself is quoted in
\eqref{eq:HF-lattice-moment}.

\begin{lemma}\label{lem:arithmetic}
Fix an integer $s\ge1$ and $0<\beta\le1$ with $\beta s>2$.
Let $A\in\mathbb R^{s\times s}$ be real and satisfy
\[
 a_0\le s_{\min}(A)\le s_{\max}(A)\le a_1,
 \qquad 0<a_0\le a_1<\infty.
\]
For an integer $M\ge1$, let $R_M,R'_M$ be independent random vectors
whose coordinates are independent copies of
\[
 \frac12\sum_{j=1}^M(\epsilon_j-\epsilon'_j),
\]
where all signs are independent and uniform on $\{-1,1\}$. Define
\[
 K_M(u;A)=\mathbb E e^{iuR_M^\top A R'_M},
 \qquad u_0=M^{-1+1/(\beta s)}.
\]
Then $0\le K_M(u;A)\le1$ for every real $u$, and
\begin{equation}\label{eq:C-arithmetic-integral}
 \int_{u_0}^1K_M(u;A)^\beta\,\frac{du}{u}\le\frac{C}{M}.
\end{equation}
The constant depends only on $s,\beta,a_0,a_1$.
The matrix $A$ need not be symmetric.
\end{lemma}

\begin{proof}
The case $M=1$ has an empty integration interval; its positivity claim
follows from the same conditional characteristic-function argument in
Step~\ref{step:arithmetic-1} below. For the integral estimate assume $M\ge2$,
and put $R=\sqrt M$. All constants in this proof may depend on the
parameters in the statement.

\medskip\noindent\refstepcounter{arithmeticstep}\label{step:arithmetic-1}\textit{Step \thearithmeticstep. Comparison with a discrete Gaussian.}
The mass function and characteristic function of one coordinate of
$R_M$ are
\[
 p_M(k)=4^{-M}\binom{2M}{M+k},\quad |k|\le M,
 \qquad \widehat p_M(x)=\cos^{2M}(x/2)\ge0.
\]
Conditioning on $R'_M$ therefore expresses $K_M$ as an expectation of
a product of nonnegative characteristic functions; in particular
$0\le K_M\le1$. For $0\le k\le M$,
\begin{align*}
 \frac{p_M(k)}{p_M(0)}
 &=\prod_{j=1}^k\left(1-\frac{2j-1}{M+j}\right)\\
 &\le\exp\left(-\sum_{j=1}^k\frac{2j-1}{M+j}\right)
 \le e^{-k^2/(2M)}.
\end{align*}
Symmetry gives the same inequality for negative $k$.
Since $p_M(0)\le C/\sqrt M$, the discrete Gaussian probability
\[
 q_M(k)=Z_M^{-1}e^{-k^2/(2M)},\qquad
 Z_M=\sum_{k\in\mathbb Z}e^{-k^2/(2M)}\asymp\sqrt M
\]
satisfies $p_M(k)\le Cq_M(k)$ for all integers $k$.
Poisson summation gives
\begin{equation}\label{eq:C-q-transform}
 \widehat q_M(x)=\frac{\sqrt{2\pi M}}{Z_M}
 \sum_{m\in\mathbb Z}e^{-M(x-2\pi m)^2/2}\ge0.
\end{equation}

Let $\zeta_M,\zeta'_M$ be independent of one another, each having
independent coordinates with mass $q_M$.
The two comparisons of laws must be made in the following order.
Write $\varphi_p(x)=\prod_{j=1}^s\widehat p_M(x_j)$ and
$\varphi_q(x)=\prod_{j=1}^s\widehat q_M(x_j)$. Then
\begin{align*}
 K_M(u;A)
 &=\mathbb E_{R'_M}\varphi_p(uAR'_M)\\
 &\le C_s\mathbb E_{\zeta'_M}\varphi_p(uA\zeta'_M)
 =C_s\mathbb E_{R_M}\varphi_q(uA^\top R_M)\\
 &\le C_s'\mathbb E_{\zeta_M}\varphi_q(uA^\top\zeta_M)
 =C_s'\mathbb E e^{iu\zeta_M^\top A\zeta'_M}.
\end{align*}
The first inequality uses $\varphi_p\ge0$, and the second uses
\eqref{eq:C-q-transform}. Thus no pointwise domination is applied to
an oscillating signed integrand.

\medskip\noindent\refstepcounter{arithmeticstep}\textit{Step \thearithmeticstep. Poisson summation and its normalization.}
Put $B=A/(2\pi)$. Conditioning on $\zeta'_M=k$ and applying
\eqref{eq:C-q-transform} coordinate by coordinate gives the exact identity
\begin{align*}
 \mathbb E e^{iu\zeta_M^\top A\zeta'_M}
 =\frac{(2\pi M)^{s/2}}{Z_M^{2s}}
 \sum_{m,k\in\mathbb Z^s}
 \exp\left\{-2\pi^2M\|m-uBk\|^2-\frac{\|k\|^2}{2M}\right\}.
\end{align*}
Consequently,
\begin{equation}\label{eq:HF-poisson}
 K_M(u;A)\le C R^{-s}
 \sum_{m,k\in\mathbb Z^s}
 \exp\{-cR^2\|m-uBk\|^2-cR^{-2}\|k\|^2\}.
\end{equation}
The factor $R^{-s}$ comes from
$(2\pi M)^{s/2}/Z_M^{2s}\asymp M^{-s/2}$.
We estimate \eqref{eq:HF-poisson} first on $[u_0,R^{-1}]$ and then
on $[R^{-1},1]$. These intervals are in the stated order because
$\beta s>2$ implies $u_0\le M^{-1/2}=R^{-1}$.

\medskip\noindent\refstepcounter{arithmeticstep}\label{step:arithmetic-3}\textit{Step \thearithmeticstep. Completing the square when $u\le R^{-1}$.}
For each $m$, define
\[
 H=R^{-2}I_s+u^2R^2B^\top B,\qquad
 z_m=H^{-1}uR^2B^\top m,\qquad
 S=R^2(I_s+u^2R^4BB^\top)^{-1}.
\]
Expansion of the right-hand side, or the matrix inversion identity,
gives
\begin{equation}\label{eq:C-square-completion}
 R^2\|m-uBk\|^2+R^{-2}\|k\|^2
 =(k-z_m)^\top H(k-z_m)+m^\top S m.
\end{equation}
Writing $h=R^{-2}(1+u^2R^4)$, the singular-value bounds on $B$ imply
\[
 c h I_s\preceq H\preceq C h I_s,
 \qquad S\succeq\frac{cR^2}{1+u^2R^4}I_s.
\]
On the range $u\le R^{-1}$ and $R\ge1$, we have $h\le2$ and
$R^2/(1+u^2R^4)\ge1/2$. The elementary one-dimensional estimate
\[
 \sup_{z\in\mathbb R}\sum_{k\in\mathbb Z}e^{-ch(k-z)^2}
 \le C(1+h^{-1/2})\le C h^{-1/2},\qquad 0<h\le2,
\]
follows by grouping points according to their distance from $z$ and
comparing the Gaussian tail with its integral. Applying it in each
coordinate to \eqref{eq:C-square-completion}, and then summing over $m$,
proves
\begin{equation}\label{eq:C-small-u}
 K_M(u;A)\le C(1+M^2u^2)^{-s/2},\qquad 0\le u\le M^{-1/2}.
\end{equation}
Since $Mu_0=M^{1/(\beta s)}\ge1$,
\[
 \int_{u_0}^{M^{-1/2}}K_M(u;A)^\beta\,\frac{du}{u}
 \le C M^{-\beta s}\int_{u_0}^{\infty}u^{-\beta s-1}\,du
 =\frac{C}{\beta s}(Mu_0)^{-\beta s}\le\frac CM.
\]

\medskip\noindent\refstepcounter{arithmeticstep}\textit{Step \thearithmeticstep. Lattice height and the source rotation estimate.}
A full-rank lattice in $\mathbb R^d$ is a set $L=G\mathbb Z^d$ for an
invertible matrix $G$. Its covolume is $|\det G|$.
For a rank-$j$ sublattice $L'$ with basis $v_1,\ldots,v_j$, its covolume
in its linear span is
\[
 \operatorname{covol}(L')
 =\det\bigl((\langle v_a,v_b\rangle)_{a,b=1}^j\bigr)^{1/2}.
\]
Using the Euclidean exterior-power norms in
\cite[(4.2)--(4.3)]{ButerusGoetzeHilleMargulis2022}, set
\[
 \alpha_j(L)=\sup_{\substack{L'\subset L\\\operatorname{rank}L'=j}}
            \operatorname{covol}(L')^{-1},\qquad
 \alpha_0(L)=1,\qquad \alpha(L)=\max_{0\le j\le d}\alpha_j(L).
\]
The rank-zero contribution ensures $\alpha(L)\ge1$. For positive
ranks it is enough to take primitive sublattices, because enlarging
a sublattice inside its span can only decrease its covolume.

We use the standard Gaussian lattice bound
\begin{equation}\label{eq:C-lattice-gaussian}
 \sum_{x\in L}e^{-c\|x\|^2}\le C_{d,c}\alpha(L).
\end{equation}
Here is how the hypotheses in the cited source are met. If the shortest
nonzero vector has length at most one, Lemmas~7.3--7.4 of
\cite{GoetzeZaitsev2014} give the bound by $C\alpha(L)$ through their
lattice-point counting estimates. If that length exceeds one, the
balls of radius $1/2$ around lattice points are disjoint, whence
$\#\{x\in L:\|x\|\le t\}\le(2t+1)^d$.
Summing over integer shells proves a constant bound in this case,
which is absorbed because $\alpha(L)\ge1$.

In dimension $d=2s$, introduce
\[
 D_\rho=\begin{pmatrix}\rho I_s&0\\0&\rho^{-1}I_s\end{pmatrix},\quad
 U_u=\begin{pmatrix}I_s&-uI_s\\0&I_s\end{pmatrix},\quad
 O_\theta=\begin{pmatrix}\cos\theta I_s&-\sin\theta I_s\\
                  \sin\theta I_s&\cos\theta I_s\end{pmatrix}.
\]
The nontrivial external input, in the present coordinate ordering, is
\begin{equation}\label{eq:HF-lattice-moment}
 \int_0^{2\pi}\alpha(D_\rho O_\theta L)^\beta\,d\theta
 \le C_{s,\beta}\alpha(L)^\beta
                 \max(\rho,\rho^{-1})^{\beta s-2},
\end{equation}
for every lattice $L\subset\mathbb R^{2s}$, every $\rho>0$, and
$\beta>0$ with $\beta s>2$. This is Theorem~5.12 of
\cite{ButerusGoetzeHilleMargulis2022}, with Euclidean exterior-power norms.
The normalized Haar measure on $\mathrm{SO}(2)$ is $d\theta/(2\pi)$,
so using $d\theta$ only changes the constant.
To match the source, take its $2\times2$ matrix
$H=\operatorname{diag}(\rho,\rho^{-1})\in\operatorname{SL}(2,\mathbb R)$
and its parameter $d=s$. Permuting coordinates from
$(x_1,y_1,\ldots,x_s,y_s)$ to $(x_1,\ldots,x_s,y_1,\ldots,y_s)$
is orthogonal and preserves all covolumes. The source statement is
for arbitrary lattices, so it imposes no symmetry condition on $A$.

\medskip\noindent\refstepcounter{arithmeticstep}\textit{Step \thearithmeticstep. The lattice in the Poisson sum.}
Let
\[
 \Lambda=\operatorname{diag}(I_s,B)\mathbb Z^{2s}.
\]
A vector of $D_RU_u\Lambda$ has the form
$(R(m-uBk),R^{-1}Bk)$. Because $B$ and $B^{-1}$ have bounded operator
norms, its squared norm is comparable, uniformly in $m,k,u,R$, to
$R^2\|m-uBk\|^2+R^{-2}\|k\|^2$.
Also
\[
 \operatorname{covol}(D_RU_u\Lambda)=|\det B|
 \in[(a_0/(2\pi))^s,(a_1/(2\pi))^s].
\]
Equations \eqref{eq:HF-poisson} and \eqref{eq:C-lattice-gaussian}
therefore give
\begin{equation}\label{eq:C-kernel-height}
 K_M(u;A)\le C R^{-s}\alpha(D_RU_u\Lambda).
\end{equation}
It remains to integrate this expression for $R^{-1}\le u\le1$.

Partition that interval using the intervals
$[1/j,1/(j-1)]$, $2\le j\le\lceil R\rceil+1$; extending an end
interval causes no problem because the integrand is nonnegative.
For each such integer define
\begin{equation}\label{eq:C-reference-lattice}
 \Lambda_j=D_jU_{1/j}\Lambda
 =\begin{pmatrix}jI_s&-B\\0&j^{-1}B\end{pmatrix}\mathbb Z^{2s}.
\end{equation}
We now verify the uniform bound $\alpha(\Lambda_j)\le C$.
If $(x,y)=(jm-Bk,j^{-1}Bk)\in\Lambda_j$ has sufficiently small norm,
then
\[
 m=x/j+y\in\mathbb Z^s
\]
has norm less than one and must vanish. For $m=0$ and $k\ne0$,
$\|x\|=\|Bk\|\ge a_0/(2\pi)$, a contradiction.
Thus there is an $\varepsilon>0$ such that
$\|w\|\ge\varepsilon$ for every $j$ and every nonzero $w\in\Lambda_j$. Minkowski's first theorem, applied in the
span of any rank-$l$ sublattice $L'\subset\Lambda_j$, gives
\[
 \operatorname{covol}(L')\ge 2^{-l}\omega_l\varepsilon^l,
 \qquad 1\le l\le2s,
\]
where $\omega_l$ is the volume of the unit Euclidean ball in
$\mathbb R^l$. This proves
\begin{equation}\label{eq:C-reference-height}
 \alpha(\Lambda_j)\le
 \max\left\{1,\max_{1\le l\le2s}\frac{2^l}{\omega_l\varepsilon^l}\right\}\le C.
\end{equation}

\medskip\noindent\refstepcounter{arithmeticstep}\textit{Step \thearithmeticstep. Each reciprocal interval is a rotation integral.}
On $[1/j,1/(j-1)]$, write
\[
 u=\frac1j+\frac w{j^2},\qquad
 0\le w\le\frac j{j-1}\le2,
 \qquad \frac{du}{u}=\frac{dw}{j+w}\le\frac{dw}{j}.
\]
Direct multiplication yields
\begin{equation}\label{eq:C-cell-factorization}
 D_RU_u\Lambda=D_{R/j}U_w\Lambda_j.
\end{equation}
Put $\rho=R/j$. Since $j\le\lceil R\rceil+1$ and $R\ge1$,
$\rho\ge1/3$. Define
\[
 E_{\rho,w}=\begin{pmatrix}I_s&0\\
                    w\rho^{-2}I_s&(1+w^2)I_s\end{pmatrix}.
\]
Both $E_{\rho,w}$ and its inverse have bounded operator norms for
$\rho\ge1/3$ and $0\le w\le2$, and
\begin{equation}\label{eq:C-shear-rotation}
 E_{\rho,w}D_\rho U_w
 =D_\rho\begin{pmatrix}I_s&-wI_s\\wI_s&I_s\end{pmatrix}
 =\sqrt{1+w^2}\,D_\rho O_{\arctan w}.
\end{equation}
A linear map with itself and its inverse bounded by $K\ge1$ changes
any rank-$l$ covolume by a factor in $[K^{-l},K^l]$.
Thus \eqref{eq:C-shear-rotation}, including its bounded scalar factor,
implies
\[
 \alpha(D_\rho U_w\Lambda_j)
 \le C\alpha(D_\rho O_{\arctan w}\Lambda_j).
\]
For $\theta=\arctan w$, the Jacobian is
$dw=(1+w^2)d\theta\le5d\theta$.
Using \eqref{eq:HF-lattice-moment} and
\eqref{eq:C-reference-height}, we obtain
\begin{align*}
 \int_{1/j}^{1/(j-1)}\alpha(D_RU_u\Lambda)^\beta\,\frac{du}{u}
 &\le\frac Cj\int_0^{2\pi}
       \alpha(D_{R/j}O_\theta\Lambda_j)^\beta\,d\theta\\
 &\le\frac Cj\max(R/j,j/R)^{\beta s-2}
 \le\frac Cj(R/j)^{\beta s-2}.
\end{align*}
The last inequality uses $R/j\ge1/3$, so that replacing the maximum
by $R/j$ changes the bound by at most a constant factor.

\medskip\noindent\refstepcounter{arithmeticstep}\textit{Step \thearithmeticstep. Summation of the intervals.}
The series $\sum_{j\ge2}j^{-(\beta s-1)}$ converges because
$\beta s>2$. Applying \eqref{eq:C-kernel-height} and the preceding
bound gives
\begin{align*}
 \int_{R^{-1}}^1K_M(u;A)^\beta\,\frac{du}{u}
 &\le C R^{-\beta s}R^{\beta s-2}
         \sum_{j=2}^{\lceil R\rceil+1}j^{-(\beta s-1)}\\
 &\le C R^{-2}=C/M.
\end{align*}
Together with Step~\ref{step:arithmetic-3}, this proves \eqref{eq:C-arithmetic-integral}.
\end{proof}

\subsection{Reduction of the nonidentical array to deterministic matrices}

\begin{proposition}\label{prop:high-frequency}
Assume \textup{(\ref{ass:A1}), (\ref{ass:A2}), \ref{ass:A3-i}}, $\|\Sigma\|=1$, and
$v=\tr(\Sigma^2)\ge11/2$.
There are constants $n_H,C_H,C>0$, depending only on the
constants $C_1,K_4,C_\Sigma$, such that for $n\ge n_H$, with
\[
 t_0=C_H n^{1/3}/\sqrt v,\qquad 0<U\le n/(4\sqrt v),
\]
one has
\begin{equation}\label{eq:C-high-frequency-integral}
 \int_{t_0}^U|\mathfrak f(t)|\,\frac{dt}{t}\le C/n
 \qquad\text{if }U>t_0.
\end{equation}
The choice $U=n/(r\sqrt v)$, where $r=\tr\Sigma$, is admissible
since $r\ge v\ge11/2>4$.
\end{proposition}

\begin{proof}
\medskip\noindent\textit{Step 1. Choice of frame parameters.}
Take $s=6$, and let $K_4$ be the directional fourth-moment constant
from \textup{(\ref{ass:A2})}:
\[
 \mathbb E|u^\top\boldsymbol X_i|^4
 \le K_4(u^\top\Sigma_i u)^2.
\]
Use the without-replacement construction following
Lemma~\ref{lem:cross-frames}, with success probability $15/16$.
This gives integers $L,n_F$ and constants $a_0,a_1>0$
such that, for $n\ge n_F$, two randomly selected disjoint frames of
$L$ actual symmetrized increments contain a $6\times6$ cross matrix
whose singular values lie in $[a_0\sqrt v,a_1\sqrt v]$, with probability
at least $15/16$. These constants depend only on $C_1,K_4,C_\Sigma$,
not on $p,n$, the smallest eigenvalue, or $\delta$.
We include $n\ge\max(n_F,24L)$ in $n_H$.
All lower bounds on $n$ below are included in $n_H$.

\medskip\noindent\textit{Step 2. Exact sign representation of the array.}
Let $\boldsymbol X'_i$ be independent copies of $\boldsymbol X_i$, and put
\[
 B_i=\frac{\sqrt n}{2}(\boldsymbol X_i+\boldsymbol X'_i),\qquad
 V_i=\frac{\sqrt n}{2}(\boldsymbol X_i-\boldsymbol X'_i).
\]
With independent Rademacher signs $\epsilon_i$,
$(B_i+\epsilon_iV_i)/\sqrt n$ is $\boldsymbol X_i$ or
$\boldsymbol X'_i$, each with probability $1/2$.
These vectors are independent over $i$ and have exactly the original
marginal laws. Conditional on the pairs $\mathcal D=((B_i,V_i))_{i=1}^n$,
write
\[
 b=\sum_i B_i,\qquad
 \Psi_{\mathcal D}(t)=\mathbb E_\epsilon
     \exp\left\{\frac{it}{2n}
             \left\|b+\sum_i\epsilon_iV_i\right\|^2\right\}.
\]
It follows that
\begin{equation}\label{eq:C-outer-average}
 \mathfrak f(t)=\mathbb E_{\mathcal D}\Psi_{\mathcal D}(t),
 \qquad |\mathfrak f(t)|\le\mathbb E_{\mathcal D}|\Psi_{\mathcal D}(t)|.
\end{equation}
The vectors $B_i$ and $V_i$ need not be independent within a pair.
No such independence is used: the subsequent analysis is conditional
on the entire collection $\mathcal D$.

\medskip\noindent\textit{Step 3. Fixing the frames and obtaining many good pairs.}
Set $M=\lfloor n/(3L)\rfloor$. Given a permutation $\pi$ of the
indices, partition its first $3ML$ positions into three groups, each
containing $M$ frames of $L$ indices. Denote their index sets by
$\mathcal I_1,\mathcal I_2,\mathcal I_3$ and put the remaining
indices in $\mathcal I_4$.
A pair of frames is good if it contains $s$ columns from each side
whose cross Gram matrix has all singular values in
$[a_0\sqrt v,a_1\sqrt v]$. For $(a,b)=(1,2),(1,3)$, define
\[
 \begin{gathered}
 G_{ab}=\frac1{M^2}\sum_{k,l=1}^M
 \boldsymbol 1\{\text{frame }k\text{ in group }a
                    \text{ and frame }l\text{ in group }b\text{ are good}\},\\
 g_{ab}(\pi)=\mathbb E_V G_{ab}.
 \end{gathered}
\]

Temporarily give $\pi$ the uniform distribution, independently of the
vectors. The $2L$ indices in each pair of frame positions then
have the uniform without-replacement law. The application following
Lemma~\ref{lem:cross-frames} shows that each such pair is good with
probability at least $15/16$, provided $n\ge n_F$. Consequently,
\begin{equation}\label{eq:C-mean-good-fraction}
 \mathbb E_\pi g_{12}(\pi)\ge15/16,
 \qquad \mathbb E_\pi g_{13}(\pi)\ge15/16.
\end{equation}
There is therefore a deterministic permutation $\pi_0$ such that
$g_{12}(\pi_0)+g_{13}(\pi_0)\ge15/8$. Since each summand is at
most one, both are at least $7/8$. Fix this permutation for the
rest of the proof. It depends only on the laws of the vectors,
so the $V_i$ remain independent and the frames are deterministic.

Changing one $V_i$ affects at most $M$ of the $M^2$ indicators in
$G_{ab}$, and hence changes $G_{ab}$ by at most $1/M$. Put
$\eta_0=7/8-3/4=1/8$. Bounded differences gives
\[
 \mathbb P\{G_{ab}<3/4\}
 \le \exp(-2\eta_0^2M^2/n).
\]
For $n\ge6L$ we have $M\ge n/(6L)$, and a union bound yields
\begin{equation}\label{eq:HF-good-fraction}
 \mathbb P(\mathcal E^c)\le2\exp\{-n/(1152L^2)\},\qquad
 \mathcal E=\{G_{12}\ge3/4,\ G_{13}\ge3/4\}.
\end{equation}
The constant depends only on $L$.
The event $\mathcal E$ depends only on the vectors $V_i$, not on
the signs or on the frequency $t$. No independence between different
frame-pair indicators is needed.

\medskip\noindent\textit{Step 4. The real symmetrization inequality.}
For independent real random vectors $Z,U,V,W$, a symmetric operator
$Q$, and a real quadratic polynomial
$P(x)=\langle Qx,x\rangle+\langle\ell,x\rangle+c$, Lemma~6.1 of
\cite{GoetzeZaitsev2014} states
\[
 2|\mathbb E e^{itP(Z+U+V+W)}|^2
 \le\mathbb E e^{2it\langle Q(Z-Z'),U-U'\rangle}
    +\mathbb E e^{2it\langle Q(Z-Z'),V-V'\rangle},
\]
where primes denote independent copies.

Conditionally on $\mathcal D$, take the four independent sign sums over
$\mathcal I_1,\ldots,\mathcal I_4$ as $Z,U,V,W$, and take
\[
 P(x)=\|b+x\|^2/(2n),\qquad Q=I_p/(2n).
\]
This choice accounts both for the deterministic shift $b$ and for the
unused indices. The source inequality then gives
\begin{equation}\label{eq:C-real-symmetrization}
 |\Psi_{\mathcal D}(t)|
 \le\sqrt{\{F_{12}(t)+F_{13}(t)\}/2}
 \le2^{-1/2}\{F_{12}(t)^{1/2}+F_{13}(t)^{1/2}\},
\end{equation}
where
\[
 F_{ab}(t)=\mathbb E_\epsilon
 \exp\left\{\frac{it}{n}
   \left\langle\sum_{i\in\mathcal I_a}(\epsilon_i-\epsilon'_i)V_i,
                  \sum_{j\in\mathcal I_b}(\epsilon_j-\epsilon'_j)V_j
   \right\rangle\right\}.
\]
The factor is $t/n$, because the source bilinear form contains $2tQ$.
Conditioning on either side shows that $F_{ab}(t)\in[0,1]$.

\medskip\noindent\textit{Step 5. Selecting good pairs before two applications of H\"older.}
Fix $(a,b)$ and a realization in $\mathcal E$. Put
$q=\lfloor3M/8\rfloor$. There are $q$ rows of the good-pair graph,
each containing at least $q$ good neighbors. Indeed, otherwise the
number of good pairs would be at most
\[
 (q-1)M+(M-q+1)(q-1)<2qM\le3M^2/4,
\]
contrary to $G_{ab}\ge3/4$. Choose the first $q$ such rows in
lexicographic order, forming a set $\mathcal J$, and in each row $k$
choose its first $q$ good neighbors, forming $\mathcal K_k$.
These choices depend only on the frame data and not on $t$ or on the signs.
Since $n\ge24L$, one has $M\ge8$ and
\begin{equation}\label{eq:C-selected-size}
 q\ge M/4\ge n/(24L).
\end{equation}

Let $A_k$ and $B_l$ be the symmetrized sign sums in the corresponding
frames. Their characteristic functions $a_k(x)$ and $b_l(x)$ are
products of squared cosines, so lie in $[0,1]$. Put $\xi=t/n$ and
$B=\sum_{l=1}^M B_l$. Delete all first-side factors except the selected
rows, and apply H\"older with exponent $q$:
\begin{align*}
 F_{ab}(t)
 &=\mathbb E_B\prod_{k=1}^M a_k(\xi B)\\
 &\le\mathbb E_B\prod_{k\in\mathcal J}a_k(\xi B)\\
 &\le\prod_{k\in\mathcal J}
          \{\mathbb E_Ba_k(\xi B)^q\}^{1/q}.
\end{align*}
Introduce $\widehat A_k$ and $\widehat B_l$ as sums of $q$
independent copies of their respective frame sums. Reinterpreting the
power as a characteristic function gives
\[
 \mathbb E_B a_k(\xi B)^q
 =\mathbb E_{\widehat A_k}\prod_{l=1}^M b_l(\xi\widehat A_k).
\]
Delete all factors except those in $\mathcal K_k$, and apply H\"older
again, separately in each selected row:
\[
 \mathbb E_B a_k(\xi B)^q
 \le\prod_{l\in\mathcal K_k}
      \{\mathbb E_{\widehat A_k}b_l(\xi\widehat A_k)^q\}^{1/q}.
\]
Consequently,
\begin{equation}\label{eq:C-selected-holder}
 F_{ab}(t)\le
 \prod_{k\in\mathcal J}\prod_{l\in\mathcal K_k}F_{kl}(t)^{1/q^2},
 \qquad
 F_{kl}(t)=\mathbb E e^{i(t/n)\widehat A_k^\top\widehat B_l}.
\end{equation}
Each factor is nonnegative, as follows by conditioning on either side.
The second neighbor set is allowed to depend on the first row; the
second H\"older inequality is applied separately for each row.

\medskip\noindent\textit{Step 6. Removing columns and identifying the arithmetic kernel.}
For each selected pair, choose the first admissible cross submatrix in
lexicographic order. Write the selected columns of the two frames as $Z_1,\ldots,Z_s$
and $Z'_1,\ldots,Z'_s$, and put
\[
 C_{kl}=(\langle Z_j,Z'_m\rangle)_{j,m=1}^s,
 \qquad A_{kl}=C_{kl}/\sqrt v.
\]
Each normalized matrix $A_{kl}$ is independent of the frequency variable and has its
singular values in $[a_0,a_1]$.

To justify deleting the other columns, first condition on
$\widehat B_l$:
\[
 F_{kl}(t)=\mathbb E_{\widehat B_l}
   \prod_{i\text{ in frame }k}
         \cos^{2q}\bigl((t/n)V_i^\top\widehat B_l\bigr).
\]
Every factor lies in $[0,1]$, so omitting the factors for all unselected left
columns can only increase the integrand. Reinterpret the resulting product as
the characteristic function of the reduced left sum, condition on
that sum, and delete the unwanted right columns in the same way.
This sequential deletion is legitimate; deletion inside the original
oscillating exponential would not give the claimed inequality.

The two reduced sums are $2\sum_{j=1}^s(R_q)_jZ_j$ and
$2\sum_{m=1}^s(R'_q)_mZ'_m$. Therefore
\begin{equation}\label{eq:C-kernel-identification}
 F_{kl}(t)\le K_q(4t\sqrt v/n;A_{kl}).
\end{equation}
The factor four comes from the two factors two in these sums.
Take a square root in \eqref{eq:C-selected-holder}, and only then
apply the arithmetic--geometric mean inequality to its $q^2$ factors:
\begin{align*}
 F_{ab}(t)^{1/2}
 &\le\left\{\prod_{k\in\mathcal J}\prod_{l\in\mathcal K_k}
          K_q(4t\sqrt v/n;A_{kl})^{1/2}\right\}^{1/q^2}\\
 &\le\frac1{q^2}\sum_{k\in\mathcal J}\sum_{l\in\mathcal K_k}
          K_q(4t\sqrt v/n;A_{kl})^{1/2}.
\end{align*}
The arithmetic exponent is exactly $\beta=1/2$. With $s=6$,
$\beta s=3>2$ and $1/(\beta s)=1/3$.
The good-pair fraction reduces the number of artificial copies from
$M$ to $q$; it does not reduce the exponent. Using the two
H\"older products is essential: a square root of an arithmetic mean
could not instead be bounded by the mean of square roots.

\medskip\noindent\textit{Step 7. Integrating before averaging over the data.}
For each realization in $\mathcal E$, all matrices $A_{kl}$
do not depend on $t$. With
\[
 t_0'=\frac n{4\sqrt v}q^{-2/3},\qquad u=4t\sqrt v/n,
\]
we have $dt/t=du/u$. Since $U\le n/(4\sqrt v)$, the upper endpoint
in the $u$ variable is at most one. Lemma~\ref{lem:arithmetic} bounds
every integral in the last average by $C/q$ for a lower endpoint
at least $t_0'$. Because $q\ge n/(24L)$,
\[
 t_0'\le\frac{(24L)^{2/3}}4\frac{n^{1/3}}{\sqrt v}.
\]
Choose $C_H\ge(24L)^{2/3}/4$. Thus $t_0\ge t_0'$, and
\eqref{eq:C-real-symmetrization} implies on $\mathcal E$ that
\[
 \int_{t_0}^U|\Psi_{\mathcal D}(t)|\,\frac{dt}{t}\le C/q\le C/n.
\]
On $\mathcal E^c$, use $|\Psi_{\mathcal D}(t)|\le1$.
The ratio of endpoints satisfies
$U/t_0\le n^{2/3}/(4C_H)$, so its logarithm is at most $C\log n$
whenever the interval is nonempty.
Taking expectations in \eqref{eq:C-outer-average}, and using Tonelli's
theorem for the nonnegative integrands, now gives
\begin{align*}
 \int_{t_0}^U|\mathfrak f(t)|\,\frac{dt}{t}
 &\le\frac Cn+\mathbb P(\mathcal E^c)\log(U/t_0)\\
 &\le\frac Cn+C e^{-cn}\log n\le\frac Cn.
\end{align*}
This proves \eqref{eq:C-high-frequency-integral}.
\end{proof}

\section{Comparison estimates for the range
\texorpdfstring{$6\le v<7$}{6 <= v < 7}}
\label{sec:reserve6}
This appendix proves the two comparison bounds used in
Section~\ref{sec:case-v-near6}. They require a bounded trace and a
uniform positive lower bound on the sixth eigenvalue, so they remain valid
after truncation even if the normalized squared spectral mass is
slightly below six. At each replacement step we specify the independent Gaussian summand. All Fourier integrals are
estimated in absolute value.

\begin{lemma}
\label{reserve6:sixth-eigenvalue}
Let $1\ge s_1\ge s_2\ge\cdots\ge0$ be the eigenvalues of $\Sigma$,
$r=\operatorname{tr}\Sigma$, and $v=\operatorname{tr}\Sigma^2$.
If $v\ge6$, then
\begin{equation}\label{reserve6:sixth-lower}
 s_6\ge\frac{v-5}{r-5}.
\end{equation}
If also $r\le C_\Sigma v$, then
\begin{equation}\label{reserve6:sixth-uniform}
 s_6\ge\frac1{6C_\Sigma-5}.
\end{equation}
\end{lemma}
\begin{proof}
For $j\le5$, the inequality $s_j\le1$ gives
\[
 s_j^2-s_6s_j\le1-s_6;
\]
indeed $s_j^2-s_6s_j-(1-s_6)=(s_j-1)(s_j+1-s_6)\le0$.
For $j\ge6$, $s_j^2\le s_6s_j$. Adding yields
$v\le5(1-s_6)+s_6r=5+s_6(r-5)$.
Since $r\ge v\ge6$, division by $r-5$ proves
\eqref{reserve6:sixth-lower}. Finally $C_\Sigma\ge1$, and
$(x-5)/(C_\Sigma x-5)$ is nondecreasing for $x\ge6$.
This proves \eqref{reserve6:sixth-uniform}.
\end{proof}

For $6\le v<7$, assumptions \textup{(\ref{ass:A1})--(\ref{ass:A3})} therefore provide
the bounds $r\le R_0=7C_\Sigma$ and
$s_6\ge a_0=(6C_\Sigma-5)^{-1}$.
The following proposition uses only these two spectral bounds. Its
hypotheses also hold for the normalized bounded proxy used in the
endpoint truncation argument, with adjusted values of $R_0,a_0$.
This refines the proxy-to-Gaussian comparison in
Proposition~\ref{prop:moderate}: at six nondegenerate directions, the
coarse cubic bound would produce an integral of order $n^{-1}\log n$.
The proof uses the covariance contractions in the sixth derivative
to obtain an integrable bound of order $n^{-1}$.

\begin{proposition}
\label{reserve6:comparison}
Suppose $\Sigma=\sum_i\Sigma_i$, $\|\Sigma\|\le1$,
$\|\Sigma_i\|\le C_1/n$, and
$\operatorname{tr}\Sigma_i\le C_1r/n$.
The independent centered vectors $A_i$ satisfy
\[
 \operatorname{Cov}(A_i)=\Sigma_i/2,\qquad
 \|A_i\|\le L\sqrt r,\qquad
 \mathbb E|u^\top A_i|^4\le2K_4(u^\top\Sigma_i u)^2
 \quad\text{for real }u.
\]
Let $G_i\sim N(0,\Sigma_i/2)$ be mutually independent and independent
of the array $(A_i)$, and set
\[
 \begin{gathered}
 Z_i=A_i+G_i,\qquad G_0=\sum_iG_i,\qquad G\sim N(0,\Sigma),\\
 \mathfrak f_Z(t)=\mathbb E e^{it\|\sum_iZ_i\|^2/2},\qquad
 \widetilde{\mathfrak f}(t)=\mathbb E e^{it\|G\|^2/2}.
 \end{gathered}
\]
Suppose $r\le R_0$ and $s_6\ge a_0>0$. There is a constant $C$,
depending only on $R_0,\allowbreak a_0,\allowbreak L,\allowbreak C_1,\allowbreak K_4$, such that
\begin{equation}\label{reserve6:integral}
 \int_0^\infty
 \frac{|\mathfrak f_Z(t)-\widetilde{\mathfrak f}(t)|}{t}\,dt
 \le\frac Cn.
\end{equation}
Consequently,
\begin{equation}\label{reserve6:kolmogorov}
 \sup_x\left|\mathbb P\{\|\textstyle\sum_iZ_i\|^2/2\le x\}
             -\mathbb P\{\|G\|^2/2\le x\}\right|\le\frac Cn.
\end{equation}
\end{proposition}
\begin{proof}
All constants below depend only on the constants listed in the
proposition. It is enough to treat sufficiently large $n$. For bounded
$n$, the fourth moment bounds and $r\le R_0$ control the integral near
zero, while the independent Gaussian component bounds the remaining integral by a
constant independent of $n$; increasing $C$ then proves \eqref{reserve6:integral}.
Set $\tau=n^{1/3}$ and introduce independent
$B_i\sim N(0,\Sigma_i/2)$. The sums $V_i$ and $U_{ij}$ in the two replacements
both contain the independent Gaussian summand $G_0$. The exact identities
\eqref{moderate:first-replacement-explicit} and
\eqref{moderate:double-replacement-explicit} apply to this array.

\paragraph*{Projection moments and derivatives through order seven.}
Put
\[
 R=(I-it\Sigma/2)^{-1},\quad
 D(t)=\det(I+t^2\Sigma^2/4)^{-1/4},\quad
 \beta_t^2=\min(1,t^{-1}).
\]
For disjoint $I,J\subseteq\{1,\ldots,n\}$, put
$W=\sum_{i\in I}A_i+\sum_{j\in J}B_j$. Consequently, for real $u$, scalar
Rosenthal inequalities and the fourth moment bound give, with
$\rho_u=u^\top R\Sigma R^*u$,
\begin{align*}
 \mathbb E|u^\top RW|^6
 &\le C\left\{\rho_u^3+\frac rn\|R^*u\|^4\rho_u\right\},\\
 \mathbb E|u^\top RW|^8
 &\le C\left\{\rho_u^4+\frac {r^2}n\|R^*u\|^6\rho_u\right\}.
\end{align*}
Indeed, for a bounded increment $Y_i$, its sixth or eighth projection
moment is bounded by $L^2r\|R^*u\|^2$ or
$L^4r^2\|R^*u\|^4$ times its fourth projection moment. Furthermore,
\[
 \sum_i\mathbb E|u^\top RY_i|^4
 \le C\max_i(u^\top R\Sigma_iR^*u)\rho_u
 \le\frac Cn\|R^*u\|^2\rho_u.
\]
The Gaussian increments obey the same inequalities after increasing $C$;
their sixth and eighth moment sums are bounded by
$C n^{-2}\|R^*u\|^4\rho_u$ and
$C n^{-3}\|R^*u\|^6\rho_u$, respectively.
Complex projections follow by separating real and imaginary parts.
Since $\rho_u\le C\beta_t^2\|u\|^2$ and $r^2t^3/n\le C$ for $t\le\tau$,
\[
 \|u^\top RW\|_{L^\ell}\le C_\ell\beta_t\|u\|,\qquad 1\le\ell\le8.
\]
The Gaussian integration formula and its singleton/pair derivative
expansion therefore prove, for $0\le k\le7$,
\[
 \begin{gathered}
 |D^kF(z)[u_1,\ldots,u_k]|
 \le C_kD(t)t^{k/2}(1+\sqrt t\|z\|)^k
       \prod_{\ell=1}^k\|u_\ell\|,\\
 F(z)=\mathbb E e^{it\|G_0+W+z\|^2/2}.
 \end{gathered}
\]
For $t\le1$ use $t\beta_t\le\sqrt t$; for $t\ge1$ use
$\beta_t=t^{-1/2}$.

\paragraph*{One more term in the second replacement.}
For $h_a(x)=D^3F(x)[a,a,a]$, let $\mathcal R_k$ denote the complete
Taylor remainder after degree $k-1$. The exact identity
\[
 \mathcal R_3(h_a,b)
 =\frac16D^6F(0)[a,a,a,b,b,b]+\mathcal R_4(h_a,b)
\]
is useful because the last remainder uses fourth moments of $b$ only.
The derivative estimate through order seven implies
\begin{equation}\label{reserve6:fourth-of-third}
 |\mathcal R_4(h_a,b)|
 \le CD(t)t^{7/2}\|a\|^3\|b\|^4.
\end{equation}
For $\sqrt t\|b\|\le1$, use the integral Taylor remainder and the
seventh derivative. For $\sqrt t\|b\|>1$, bound the defining difference
using the third derivative at $b$ and the derivatives of orders three
through six at zero. Every term is bounded by
$CD(t)t^{3/2}\|a\|^3(\sqrt t\|b\|)^4$.

Apply this identity to the exact double replacement
formula \eqref{moderate:double-replacement-explicit}, with $F=F_{ij}$. The sixth derivative term
with $b=B_j$ vanishes because the centered Gaussian cubic tensor is zero.
The sum of the errors \eqref{reserve6:fourth-of-third}, for $a=A_i$ and $b=A_j$ or $B_j$, is at most
\begin{equation}\label{reserve6:seventh-error}
 CD(t)t^{7/2}n^2(r/n)^{3/2}(r/n)^2
 \le CD(t)t^{7/2}n^{-3/2}.
\end{equation}
The sixth eigenvalue bound gives $D(t)\le C\min(1,t^{-3})$.
Hence the integral of \eqref{reserve6:seventh-error} against $dt/t$ up to $\tau$ is
$C n^{-3/2}(1+\tau^{1/2})\le C/n$.
The complete fourth remainders in
\eqref{moderate:first-replacement-explicit} already
have bound $CD(t)t^2r^2/n$ and are integrable with the same rate.
It remains to treat
\begin{equation}\label{reserve6:leading-sixth}
 \sum_{i<j}\mathbb E D^6F_{ij}(0)[A_i,A_i,A_i,A_j,A_j,A_j].
\end{equation}

\paragraph*{Covariance contractions in the sixth derivative.}
We work on $1\le t\le\tau$; on $(0,1)$ the original derivative bound
already suffices. Put $\delta_n=(r/n)^{1/3}$ and
\[
 M=R(\Sigma+\delta_n I)R^*,\qquad
 c(t)=\operatorname{tr}(\Sigma R\Sigma R^*),\qquad
 \bar c(t)=\operatorname{tr}(\Sigma M).
\]
The sixth projection estimate, $\rho_u\le\|R^*u\|^2$, and
Lyapunov's inequality give
\begin{equation}\label{reserve6:inflated-projections}
 \mathbb E|u^\top RW|^k\le C_k(u^\top Mu)^{k/2},
 \qquad 1\le k\le6.
\end{equation}
For the covariance bookkeeping define
\[
 a_i=\operatorname{tr}(|R|\Sigma_i),\quad
 c_i=\operatorname{tr}(M\Sigma_i),\quad
 c_{ij}=\operatorname{tr}(\Sigma_iR\Sigma_jR^*).
\]
All are nonnegative. Their sums and maxima obey
\begin{equation}\label{reserve6:trace-bookkeeping}
\begin{gathered}
 \sum_i c_i=\bar c,\quad \sum_{i,j}c_{ij}=c,\quad \sum_i a_i\le r,\\
 \max_i a_i\le Cr/n,\quad \max_{i,j}c_{ij}\le Cr/n^2,\\
 \max_i c_i\le\frac{Cr}{n}(t^{-1}+\delta_n)\le\frac C{nt},
 \qquad \bar c\le c+\delta_n r.                  
\end{gathered}
\end{equation}
The last maximum uses $\delta_n\tau\le C$ and $r\le R_0$.

After integrating $G_0$, the sixth derivative is a sum over partitions
of the six labeled directions $a,a,a,b,b,b$ into pairs and singletons.
Each pair contributes $it\,u^\top Ru'$ and each singleton contributes
$it\,u^\top RW$. The remaining exponential has modulus at most one.
The number of terms is an absolute constant. Apply \eqref{reserve6:inflated-projections} to the
singleton factors before averaging independent $a=A_i,b=A_j$.
The following display lists every possible type, up to interchange of
$a,b$. The powers of $t$ and the common determinant $D(t)$ are omitted.
Here ``$aa$'' denotes an internal pair, ``$ab$'' a cross pair, and
unpaired directions are singletons.
\begin{equation}\label{reserve6:partition-table}
\begin{array}{c|c|c}
\text{pairs}&\text{singletons}&\text{bound after all expectations}\\ \hline
\varnothing&a^3b^3&C c_i^{3/2}c_j^{3/2}\\
aa&ab^3&C a_i c_i^{1/2}c_j^{3/2}\\
ab&a^2b^2&C c_{ij}^{1/2}c_i c_j\\
aa,bb&ab&C a_i a_j c_i^{1/2}c_j^{1/2}\\
aa,ab&b^2&C a_i c_{ij}^{1/2}c_j\\
ab,ab&ab&C c_{ij}c_i^{1/2}c_j^{1/2}\\
aa,bb,ab&\varnothing&C a_i a_j c_{ij}^{1/2}\\
ab,ab,ab&\varnothing&C c_{ij}^{3/2}
\end{array}
\end{equation}
These bounds use fourth moments only. For example,
$\mathbb E|A_i^\top RA_j|^4\le Cc_{ij}^2$ by conditioning first
on $A_j$, and the quadratic form fourth moment inequality follows
from the directional fourth moment assumption. Similarly,
$\mathbb E|A_i^\top RA_i|^2\le Ca_i^2$ and
$\mathbb E(A_i^*MA_i)^2\le Cc_i^2$.
For the $aa,ab$ row, first condition on $b$ and use
\[
 \mathbb E_a|a^\top Ra|\,|a^\top Rb|
 \le Ca_i(b^*R^*\Sigma_iRb)^{1/2},
\]
then apply Cauchy--Schwarz in $b$ with $b^*Mb$.
For the one-cross-pair row, one may condition successively:
\[
 \begin{aligned}
 \mathbb E_a|a^\top Rb|(a^*Ma)
 &\le Cc_i(b^*R^*\Sigma_iRb)^{1/2},\\
 \mathbb E_b(b^*Mb)(b^*R^*\Sigma_iRb)^{1/2}
 &\le Cc_j\sqrt{c_{ij}}.
 \end{aligned}
\]
The other rows follow directly from these same second and fourth
moment inequalities and Cauchy--Schwarz.

Although the sum $W_{ij}$ varies with $(i,j)$, the bound
\eqref{reserve6:inflated-projections} is uniform in the pair.
We may therefore sum the resulting estimates over $i<j$.
Using \eqref{reserve6:trace-bookkeeping}, the sum of the no-pair row is at most
$C\bar c^2/(nt)$. The sum of either one-pair type is at most
$C\bar c^{3/2}/(n\sqrt t)$: for example,
\[
 (\sum_i a_i\sqrt{c_i})(\sum_j c_j^{3/2})
 \le (\max_i a_i)\sqrt{n\bar c}\,
       \sqrt{\max_i c_i}\,\bar c
 \le\frac C{n\sqrt t}\bar c^{3/2}.
\]
For the cross-pair row, Cauchy--Schwarz gives the sharper bound
$\sqrt c\sum_i c_i^2\le C\bar c^{3/2}/(nt)$.
Every two-pair row sums to at most $C\bar c/n$. In particular,
\[
 (\sum_i a_i\sqrt{c_i})^2\le n(\max_i a_i)^2\bar c
 \le C\bar c/n,
\]
and
\[
 \sum_{i,j}a_i\sqrt{c_{ij}}c_j
 \le(\max_i a_i)\sqrt c\sqrt{n\sum_jc_j^2}
 \le C\bar c/n.
\]
The two-cross-pair sum is bounded by
$\sum_{i,j}c_{ij}\sqrt{c_i c_j}\le(\max_i c_i)c\le C\bar c/n$.
The two three-pair rows sum to at most
$C(\sqrt c+c)/n$, since
\[
 \sum_{i,j}a_i a_j\sqrt{c_{ij}}
 \le\sqrt c\sum_i a_i^2\le C\sqrt c/n,
 \qquad
 \sum_{i,j}c_{ij}^{3/2}\le\sqrt{\max c_{ij}}\,c\le Cc/n.
\]
We have enlarged $i<j$ to all ordered pairs only in nonnegative
bounds. Restoring the powers of $t$, the absolute value of \eqref{reserve6:leading-sixth} is
therefore at most
\begin{equation}\label{reserve6:sixth-aggregate}
 \frac{CD(t)}n\left\{
 t^5\bar c(t)^2+t^{9/2}\bar c(t)^{3/2}
 +t^4\bar c(t)+t^3[\sqrt{c(t)}+c(t)]\right\}.
\end{equation}

\paragraph*{Spectral integration.}
If $s_j$ are the eigenvalues of $\Sigma$, then
\[
 c(t)=\sum_j\frac{s_j^2}{1+t^2s_j^2/4},\qquad c(t)\le r/t.
\]
Direct integration and Minkowski's inequality in the indicated
weighted $L^p$ spaces give
\begin{align*}
 \int_0^\infty c(t)\,dt&=\pi r,\\
 \int_0^\infty t\,c(t)^2\,dt&\le2r^2,\\
 \int_0^\infty t^{1/2}c(t)^{3/2}\,dt
 &\le \sqrt2 B(3/4,3/4)r^{3/2}.
\end{align*}
For example the weighted $L^2$ norm of the $j$th summand is
$\sqrt2s_j$, and its weighted $L^{3/2}$ norm is an absolute constant
times $s_j$.
Because $\bar c\le c+\delta_n r$, $\delta_n\tau\le C$, and $r\le R_0$,
\[
 \int_1^\tau\{t\bar c^2+t^{1/2}\bar c^{3/2}+\bar c\}\,dt\le C.
\]
Also $c(t)\le r/t$ yields
$\int_1^\infty[\sqrt{c(t)}+c(t)]\,dt/t\le2\sqrt r+r$.
Multiplying \eqref{reserve6:sixth-aggregate} by $1/t$ and using $D(t)\le Ct^{-3}$ proves
that its integral is $C/n$. Together with \eqref{reserve6:seventh-error}, the first fourth
remainders, and the elementary estimates on $(0,1)$, this proves the integral estimate up to $\tau$.

For every $t$, conditioning on $(A_i)$ and integrating over
the independent vector $G_0$ gives
$|\mathfrak f_Z(t)|\le D(t)$; the same bound holds for
$|\widetilde{\mathfrak f}(t)|$. Since
$D(t)\le Ct^{-3}$ for $t\ge1$,
\[
 \int_\tau^\infty
 \frac{|\mathfrak f_Z(t)-\widetilde{\mathfrak f}(t)|}{t}\,dt
 \le C\tau^{-3}=C/n.
\]
This proves \eqref{reserve6:integral}. Both characteristic functions
are integrable, and their difference divided by $t$ is absolutely
integrable by this estimate. Fourier inversion for their distribution
function difference yields \eqref{reserve6:kolmogorov}.
\end{proof}

\subsection{A second-order replacement expansion at the endpoint}
\label{sec:second-order-endpoint}

We expand the original-to-proxy comparison to second order.
In the $i$th linear term, the sum $\sum_{k\ne i}Z_k$ contains an
independent Gaussian summand with covariance $(\Sigma-\Sigma_i)/2$.
The quadratic term has an additional factor $n^{-1}$. Both terms are integrable
under the sixth-eigenvalue bound.
The argument applies to nonidentically distributed increments and
allows arbitrarily many additional small eigenvalues.

We first work with a bounded normalized array. Fix constants
$R_0,a_0,T_0>0$ and suppose, in addition to the balance and fourth
moment conditions, that
\begin{equation}\label{endpoint6:local-assumptions}
 r\le R_0,\qquad s_6(\Sigma)\ge a_0,\qquad
 \|X_i\|\le L\sqrt r.
\end{equation}
All constants in this appendix can depend on $R_0,a_0,T_0$
and on $C_1,K_4,L$. They never depend on the ambient dimension.
Set $\tau=T_0n^{1/3}$. The estimates on $0<t\le\tau$ are stated for sufficiently large $n$; the remaining finite range is absorbed into
the constant in the final probability bound. Use the order from
Lemma~\ref{lem:prefix-order} and the proxies $Z_i=A_i+G_i$ from
Lemma~\ref{lem:third-moment-proxy}.

\subsubsection{Two elementary replacement identities}

The next lemma is an elementary two-variable Taylor estimate;
its complete-remainder form uses fourth moments in each of two
independent directions. We also verify the finite second-order
convolution identity below, so its use requires no convergence of
an infinite expansion.

For a smooth function $F$ and a vector $a$, define translation and
Taylor remainder operators by
\[
 T_aF(z)=F(z+a),\qquad
 \mathcal R_aF(z)=T_aF(z)-\sum_{k=0}^3\frac1{k!}D^kF(z)[a^{\otimes k}].
\]
All these constant-coefficient translation and differential operators
commute.

\begin{lemma}\label{endpoint6:double-remainder}
Let $H,\Lambda,t>0$, with $t\le\Lambda^2$. Suppose that $F$ is eight
times continuously differentiable and
\begin{align*}
 |F(z)|&\le C H,\\
 |D^kF(z)[u_1,\ldots,u_k]|
 &\le C_kH(\Lambda+t\|z\|)^k\prod_{\ell=1}^k\|u_\ell\|,
 \qquad 1\le k\le8.
\end{align*}
Then for every pair of real vectors $a,b$,
\begin{equation}\label{endpoint6:double-complete}
 |\mathcal R_a\mathcal R_bF(0)|
 \le C H\Lambda^8\|a\|^4\|b\|^4.
\end{equation}
\end{lemma}
\begin{proof}
Put $x=\Lambda\|a\|$ and $y=\Lambda\|b\|$.
If $x,y\le1$, use Taylor's integral remainder in both variables.
The resulting eighth derivative is evaluated at $sa+ub$,
$0\le s,u\le1$, where
$\Lambda+t\|sa+ub\|\le3\Lambda$. This proves the bound.

If $x,y\ge1$, expand the defining difference. The terms are
$F(a+b)$, derivatives of orders at most three at $a$ or $b$, and
mixed derivatives at zero of order at most six. For example,
\[
 |D^kF(b)[a^{\otimes k}]|
 \le C H x^k(1+y)^k\le C Hx^4y^4\quad(0\le k\le3),
\]
where the case $k=0$ uses the uniform bound on $F$.
The origin terms are bounded by $CHx^ky^\ell$, $k,\ell\le3$,
and obey the same estimate.

Finally suppose $x\ge1\ge y$. Write
\[
 \mathcal R_a\mathcal R_bF(0)
 =\mathcal R_bF(a)
   -\sum_{k=0}^3\frac1{k!}D_a^k\mathcal R_bF(0).
\]
Taylor expansion only in $b$ bounds the first term by
$CHy^4(2+x)^4$ and the term indexed by $k$ by $CHx^ky^4$.
The latter uses derivatives through order $k+4\le7$ at $ub$.
All these quantities are at most $CHx^4y^4$. Interchange $a,b$
for the remaining case. Zero directions give zero identically.
\end{proof}

Write $\mu_i=\mathcal L(X_i)$, $\nu_i=\mathcal L(Z_i)$ and
$\delta_i=\mu_i-\nu_i$. Products of measures below denote convolution.
The exact second-order identity is
\begin{equation}\label{endpoint6:bergstrom-identity}
 \prod_{i=1}^n\mu_i-\prod_{i=1}^n\nu_i
 =\sum_i\delta_i\prod_{k\ne i}\nu_k
  +\sum_{i<j}\delta_i\delta_j
     \prod_{\substack{k<j\\k\ne i}}\nu_k
     \prod_{k>j}\mu_k.
\end{equation}
To verify it, expand $\mu_i=\nu_i+\delta_i$. A term with exactly
one $\delta$ occurs in the first sum. A term with at least two
$\delta$ factors occurs in the second sum at the positions of its
first two such factors. Thus every term occurs exactly once.
This is an algebraic identity for finite signed measures; no limiting
series is used.

Since $\delta_i$ annihilates all polynomials of degree at most three,
\begin{equation}\label{endpoint6:annihilation}
 \int\!\!\int F(a+b)\,\delta_i(da)\delta_j(db)
 =\int\!\!\int\mathcal R_a\mathcal R_bF(0)
                       \,\delta_i(da)\delta_j(db).
\end{equation}
In particular the right side requires only fourth moments in each
of the two independent directions, once
\eqref{endpoint6:double-complete} has been established.
For fixed $b$, each term $D^kF(b)[a^{\otimes k}]$, $0\le k\le3$,
is a polynomial in $a$ of degree at most three and integrates to zero
against $\delta_i(da)$. For the terms with $a$ and $b$ interchanged,
integrate first against $\delta_j(db)$. The mixed derivatives at zero
are polynomial in both variables. The derivative bounds above and the
fourth moments of $|\delta_i|$ and $|\delta_j|$ justify Fubini's theorem;
\eqref{endpoint6:double-complete} bounds the remaining integral.

\subsubsection{Eighth derivatives from the fourth moment assumption}

For $i<j$, the convolution of the probability measures multiplying
$\delta_i\delta_j$ in \eqref{endpoint6:bergstrom-identity} is the law
of $G_{Q_{ij}}+W_{ij}$, where
\begin{align*}
 Q_{ij}&=\frac12\sum_{\substack{k<j\\k\ne i}}\Sigma_k,&
 W_{ij}&=\sum_{\substack{k<j\\k\ne i}}A_k+\sum_{k>j}X_k,\\
 \alpha_j&=\frac{j-2}{2n},&
 S_{ij}:=\Cov(W_{ij})&=\Sigma-\Sigma_i-\Sigma_j-Q_{ij}.
\end{align*}
The ordering and balance conditions imply, uniformly in $i<j$,
\begin{equation}\label{endpoint6:prefix-errors}
 Q_{ij},S_{ij}\preceq\Sigma,\qquad
 \|Q_{ij}-\alpha_j\Sigma\|_{\mathrm{HS}}
 +\|S_{ij}-(1-\alpha_j)\Sigma\|_{\mathrm{HS}}
 \le Cn^{-1/2}.
\end{equation}
Indeed deleting one covariance from a prefix changes the
Hilbert--Schmidt norm by at most $C/n$, and changing its covariance
fraction from $(j-1)/(2n)$ to $(j-2)/(2n)$ has the same order.
Notice that $0\le\alpha_j\le1/2$.

\begin{lemma}
\label{endpoint6:eighth-derivatives}
There are positive constants $c_0,c,C$ such that, for
$0<t\le\tau$, the functions
\[
 F_{ij}(z)=\mathbb E\exp\{it\|G_{Q_{ij}}+W_{ij}+z\|^2/2\}
\]
satisfy the hypotheses of Lemma~\ref{endpoint6:double-remainder} with
\begin{gather}
 H_6(t)=(1+c_0t^2)^{-3/2}
       +C\exp\{-cn/(t^2+t)\}+Ct^3/n^2,
 \label{endpoint6:H6}\\
 \Lambda_j(t)=\sqrt t+t b_j(t),\qquad
 b_j(t)^2=\min\{1,C/(\alpha_jt)\},
 \label{endpoint6:lambda}
\end{gather}
where $b_j=1$ if $\alpha_j=0$.
\end{lemma}
\begin{proof}
We provide the additional moment and partition details needed to
extend Lemma~\ref{lem:unused-block} from order four to order eight.
Set $R=(I-itQ_{ij})^{-1}$. The proof of
\eqref{moderate:resolvent-bound}, using
\eqref{endpoint6:prefix-errors}, gives
\begin{equation}\label{endpoint6:resolvent8}
 \|R\|\le1,\qquad R\Sigma R^*\preceq Cb_j(t)^2 I.
\end{equation}
Here $t n^{-1/2}\le T_0n^{-1/6}$, so the perturbation is uniformly
controlled for sufficiently large $n$.

Let $V$ be any block of the independent bounded increments in
$W_{ij}$, and let $w$ be a real coefficient vector. Put
$q=w^\top\Sigma w$. The fourth moment and balance conditions give
\[
 \sum_{k\text{ in }V}\mathbb E|w^\top V_k|^4
 \le\frac Cn\|w\|^2 q.
\]
Since $\|V_k\|\le C\sqrt r$, its eighth projection moment is at
most $Cr^2\|w\|^4$ times its fourth projection moment. The scalar
Rosenthal inequality at order eight therefore gives
\begin{equation}\label{endpoint6:projection8}
 \mathbb E|w^\top V|^8
 \le C\left\{q^4+\frac{r^2}{n}\|w\|^6q\right\}.
\end{equation}
The same inequality for complex coefficients follows by splitting
real and imaginary parts, using the Hermitian quadratic form for $q$.
Apply this to $w=R^\top u$ (equivalently its real and imaginary parts).
By \eqref{endpoint6:resolvent8}, $q\le Cb_j^2\|u\|^2$ and
$\|w\|\le\|u\|$. Since $b_j^2\ge c\min(1,t^{-1})$ and
$r^2t^3/n\le R_0^2T_0^3$, we obtain
\begin{equation}\label{endpoint6:projection8-final}
 \|u^\top RV\|_{L^\ell}\le C_\ell b_j\|u\|,
 \qquad 1\le\ell\le8.
\end{equation}
The eighth moments in this estimate follow from truncation and \textup{(\ref{ass:A2})}.

Partition $W_{ij}$ into $m=18$ classes and pair them into nine
independent blocks. Write $\beta=1-\alpha_j\ge1/2$. The covariance
partition calculation in Lemma~\ref{lem:prefix-damping} gives, for
each two paired class covariances $S_\ell,S_{\ell'}$,
\[
 B_1=mS_\ell=\beta\Sigma+E,\qquad
 B_2=mS_{\ell'}=\beta\Sigma+E',\qquad
 \|E\|,\|E'\|\le\eta_n=C_mn^{-1/2}.
\]
There is a direct bound on the sixth eigenvalue of their product.
The variational characterization of eigenvalues gives
\begin{align*}
 \lambda_6(B_2^{1/2}B_1B_2^{1/2})
 &\ge\beta\lambda_6(B_2^{1/2}\Sigma B_2^{1/2})
                  -\eta_n\|B_2\|\\
 &=\beta\lambda_6(\Sigma^{1/2}B_2\Sigma^{1/2})
                  -\eta_n\|B_2\|\\
 &\ge\beta^2s_6(\Sigma)^2-2\beta\eta_n-\eta_n^2
 \ge a_0^2/8
\end{align*}
for sufficiently large $n$. The equality uses the coincidence of
the eigenvalues of $CC^*$ and $C^*C$. Thus the actual product
$S_{\ell'}^{1/2}S_\ell S_{\ell'}^{1/2}$ has at least six eigenvalues at least
$a_0^2/(8m^2)$. Its determinant contribution in
\eqref{new:shifted-cf} is at most
\[
 (1+c_0t^2)^{-3/2},\qquad c_0=a_0^2/(16m^2).
\]
Since $r$ is bounded, \eqref{new:shifted-cf} gives the other two
terms in \eqref{endpoint6:H6}: $C\exp\{-cn/(t^2+t)\}$ and $Ct^3/n^2$. This proves the exact sixth-dimensional
power without an exponent loss and remains valid after covariance
truncation.
All the partition statements are deterministic existence statements
for the particular pair $i,j$. The partition need not be common to
different pairs.

Gaussian integration gives exactly
\eqref{moderate:gaussian-polynomial-identity} through order eight:
each singleton contributes $it\,u^\top R(W_{ij}+z)$ and each pair
contributes $it\,u^\top Ru'$. Expanding the singleton factors over
the nine blocks, any monomial of degree at most eight leaves an
unused block. Condition on the Gaussian vector and all the other
blocks. The remaining phase has a real shift, so the unused block
provides $CH_6(t)$ uniformly in that shift. Apply H\"older's
inequality and \eqref{endpoint6:projection8-final} to the polynomial
in the other blocks. A singleton is bounded by
$Ct(b_j+\|z\|)\|u\|$, and a pair by $t\|u\|\|u'\|$.
Since $t\le\Lambda_j^2$, this gives
\[
 |D^kF_{ij}(z)[u_1,\ldots,u_k]|
 \le C_kH_6(t)(\Lambda_j+t\|z\|)^k\prod_\ell\|u_\ell\|,
 \quad 0\le k\le8.
\]
The case $k=0$ is uniform in $z$, as required. Every differentiation
is justified by bounded non-Gaussian increments and finite Gaussian
moments.
\end{proof}

\subsubsection{Summing the linear and quadratic terms}

\begin{proposition}
\label{endpoint6:first-comparison}
Under \eqref{endpoint6:local-assumptions}, let
\[
 \mathfrak f(t)=\mathbb E e^{it\|\sum_iX_i\|^2/2},\qquad
 \mathfrak f_Z(t)=\mathbb E e^{it\|\sum_iZ_i\|^2/2}.
\]
For $0<t\le\tau$,
\begin{equation}\label{endpoint6:first-pointwise}
 |\mathfrak f(t)-\mathfrak f_Z(t)|
 \le \frac Cn(1+t^2)^{-3/2}t^2
       +\frac C{n^2}H_6(t)(t^4+t^6).
\end{equation}
Consequently,
\begin{equation}\label{endpoint6:first-integrated}
 \int_0^\tau|\mathfrak f(t)-\mathfrak f_Z(t)|\,\frac{dt}{t}
 \le C/n.
\end{equation}
\end{proposition}
\begin{proof}
Apply \eqref{endpoint6:bergstrom-identity} to the bounded function
$\Phi_t(x)=e^{it\|x\|^2/2}$.

In the $i$th term of the first sum, write
$\sum_{k\ne i}Z_k=\sum_{k\ne i}A_k+\sum_{k\ne i}G_k$.
The two sums on the right are independent, and
$\sum_{k\ne i}G_k\sim N(0,Q_i)$ with $Q_i=(\Sigma-\Sigma_i)/2$. Eigenvalue perturbation gives
$s_6(Q_i)\ge a_0/2-C_1/(2n)\ge a_0/4$. Hence
\[
 D_i(t):=\det(I+t^2Q_i^2)^{-1/4}\le C(1+t^2)^{-3/2}.
\]
For $R_i=(I-itQ_i)^{-1}$, the identity $\Sigma=2Q_i+\Sigma_i$
implies
\[
 R_i\Sigma R_i^*
 \preceq (t^{-1}+C_1/n)I\preceq Ct^{-1}I\quad(1\le t\le\tau\le n).
\]
For $t<1$ use $\|R_i\|\le1$. The fourth-moment bounds for projections and Gaussian integration
give the derivative bounds through order four for
$z\mapsto\mathbb E\Phi_t(z+\sum_{k\ne i}Z_k)$, with determinant factor $D_i(t)$. The complete fourth
remainder argument of Lemma~\ref{lem:complete-remainders} gives
$CD_i(t)t^2\|x\|^4$ for all $x$. Since
\[
 \int\|x\|^4\,|\delta_i|(dx)
 \le\mathbb E\|X_i\|^4+\mathbb E\|Z_i\|^4\le Cn^{-2},
\]
the sum of the linear terms is at most
$Cn^{-1}(1+t^2)^{-3/2}t^2$.

For a term $i<j$, use \eqref{endpoint6:annihilation},
Lemmas~\ref{endpoint6:double-remainder} and
\ref{endpoint6:eighth-derivatives}, and the two fourth moment bounds.
Its modulus is at most $CH_6(t)\Lambda_j(t)^8n^{-4}$.
There are $j-1$ pairs with second index $j$. For $1\le t\le n$,
\begin{align*}
 \sum_{j=2}^n(j-1)b_j(t)^8
 &\le C\left\{1+\sum_{k=1}^{n-2}(k+1)
                       \min\left(1,\frac{Cn}{kt}\right)^4\right\}
 \le Cn^2/t^2.
\end{align*}
Indeed, split the sum at $k=\lfloor Cn/t\rfloor$; below this
point sum $k+1$, and above it sum $C(n/t)^4k^{-3}$.
If the splitting point exceeds $n$, sum all terms by $Cn^2$.
The term $k=0$ is also covered because $t\le n$.
It follows that
\begin{equation}\label{endpoint6:weighted-position-sum}
 \sum_{j=2}^n(j-1)\Lambda_j(t)^8
 \le Cn^2(t^4+t^6),\qquad 0<t\le n.
\end{equation}
For $t<1$ this follows directly from $b_j\le1$.
Summing the quadratic terms proves \eqref{endpoint6:first-pointwise}.

The linear term is integrable because
$\int_0^\infty t(1+t^2)^{-3/2}\,dt=1$.
The contribution of the factor $(1+c_0t^2)^{-3/2}$ to the integral
of the quadratic term is at most
\[
 \frac C{n^2}\int_0^\tau(t^3+t^5)(1+c_0t^2)^{-3/2}\,dt
 \le\frac C{n^2}(1+\tau+\tau^3)\le C/n.
\]
The contribution of $Ct^3/n^2$ in $H_6(t)$ is at most
\[
 \frac C{n^4}\int_0^\tau(t^6+t^8)\,dt
 \le Cn^{-4}(\tau^7+\tau^9)\le C/n.
\]
Finally $n/(\tau^2+\tau)\ge cn^{1/3}$, so the term $C\exp\{-cn/(t^2+t)\}$ contributes at most
$Cn^{-2}(\tau^4+\tau^6)e^{-cn^{1/3}}\le C/n$.
This proves \eqref{endpoint6:first-integrated}.
\end{proof}

\end{appendix}



\section*{Declaration of AI use}
The authors used ChatGPT (OpenAI) for proof assistance, verification arguments for correctness, and for polishing the text of the manuscript. The research questions and proposed method were developed by the authors. The authors take full responsibility for the content of the paper, including the correctness of all results, proofs, and references.

\bibliographystyle{spr-ims-number} 
\bibliography{CLT_refs}

\end{document}